\documentclass[11pt]{amsart}
\usepackage{amsfonts}
\usepackage{amssymb, amstext, amscd, amsmath,mathtools}
\usepackage{amsthm}
\usepackage{times}
\usepackage{fullpage}
\usepackage{txfonts} 
\usepackage{indentfirst}
\usepackage[all]{xy}
\usepackage{subfigure}
\usepackage{graphicx}
\usepackage{tikz}
\usepackage{nicematrix} 
\usepackage{amsfonts}
\usepackage{mathrsfs}
\usepackage{lscape}
\usepackage{setspace} 
\usepackage{graphicx}      
\usepackage{subfigure}     
\usepackage{color}              
\usepackage{indentfirst}

\usepackage{url}
\usepackage[colorlinks = true,
linkcolor=blue,
urlcolor=blue,
citecolor=blue,
anchorcolor=blue,
backref=page]{hyperref}

\usepackage{url}

\usepackage[nocompress]{cite}

\newtheoremstyle{noparens}%
  {}{}%
  {\itshape}{}%
  {\bfseries}{.}%
  { }%
  {\thmname{#1}\thmnumber{ #2}\mdseries\thmnote{ #3}}
\theoremstyle{noparens}

\allowdisplaybreaks

\newtheorem{theorem}{Theorem}[section]
\newtheorem{corollary}{Corollary}[section]
\newtheorem{proposition}{Proposition}[section]
\newtheorem{lemma}{Lemma}[section]
\newtheorem{remark}{Remark}[section]
\newtheorem{definition}{Definition}[section]
\newtheorem{example}{Example}[section]

\numberwithin{equation}{section}
\numberwithin{equation}{section}

\begin{document}
\baselineskip 18pt 

\title[Vassiliev invariants for virtual knotoids]{Vassiliev invariants for virtual knotoids}

\author{Siqi Ding}
\address{School of Mathematical Sciences, Dalian University of Technology, Dalian 116024, P. R. China}
\email{sqding@yeah.net}

\author{Xiaobo Jin}
\address{School of Mathematical Sciences, Dalian University of Technology, Dalian 116024, P. R. China}
\email{jxb1104@163.com}

\author{Fengchun Lei}
\address{Beijing Institute of Mathematical Sciences and Applications, Beijing 101408, P.R. China; School of Mathematical Sciences, Dalian University of Technology, Dalian 116024, P.R. China. }
\email{leifengchun@bimsa.cn}

\author{Fengling Li}
\address{School of Mathematical Sciences, Dalian University of Technology, Dalian 116024, P. R. China}
\email{fenglingli@dlut.edu.cn}

\author{Andrei Vesnin}
\address{Sobolev Institute of Mathematics of the Siberian Branch of the Russian Academy of Sciences, Novosibirsk 630090, Russia}
\email{vesnin@math.nsc.ru}

\thanks{F.\,Lei supported in part by a grant of NSFC (No. 12331003); F.\,Li supported by the Fundamental Research Funds for the Central Universities (No. DUT25LAB302); A.\,V. supported by the state task to the Sobolev Institute of Mathematics (No. FWNF-2026-0031).}

\date{\today} 

\subjclass[2020]{57K12, 57K16}
\keywords{Virtual knotoids; flat virtual knotoids; Vassiliev invariants; singular based matrix} 

\begin{abstract}
In this paper, we introduce the 0-smoothing invariant $\mathcal{F}$ of virtual knotoids constructed from local modification at classical crossings, which take values in a free $\mathbb Z$-module generated by non-oriented flat virtual knotoids. We prove that $\mathcal{F}$ is a Vassiliev invariant of order one. It was observed by Henrich that  smoothing invariant she constructed for virtual knots provides less information than the gluing invariant. We demonstrate the same property for the 0-smoothing invariant of virtual knotoids: $\mathcal{F}$ provides less information than the gluing invariant introduced by Petit. To prove this result, we use the extension of the singular based matrix invariant originally introduced by Turaev for singular virtual strings.
\end{abstract}

\maketitle


\section{Introduction}
In 2012, Turaev introduced knotoid diagram, defining it as an open-ended topological diagram in $S^2$ (or $\mathbb{R}^2$), which is a generic immersion of the unit interval with finitely many double points. Knotoids are defined to be equivalent classes of knotoid diagrams modulo the Reidemeister moves applied away from the endpoints~\cite{Tur1}.  Knotoid theory can be considered as a generalization of classical knot theory. G\"{u}g\"{u}mc\"{u} and Kauffman later generalized knotoids to virtual knotoids by adding virtual crossings in~\cite{GK}. In recent years, there have been numerous studies on invariants of knotoids, see, for example,~\cite{Tar,FLV}.

In 2013, Kauffman defined the affine index polynomial invariant of virtual knots in~\cite{Kau}. Kim further defined the affine index polynomial invariant for flat virtual knots in~\cite{Kim}.  In 2017, G\"{u}g\"{u}mc\"{u} and Kauffman defined the affine index polynomial invariants for knotoids and virtual knotoids in~\cite{GK}.

A history and theory of Vassiliev invariants are well-presented in~\cite{CDM}. We only recall that  a notion of a finite type invariant was introduced by Vassiliev in the end of the 1980's, see~\cite{Vas1, Vas2}. In~\cite{BL} Birman and Lin realized Vassiliev's approach, established the relation between the Jones polynomial and finite type invariants and emphasized the algebra of chord diagrams (called circular configurations in the paper).  Bar-Natan described the theory of knot invariants of finite type (Vassiliev invariants) in \cite{BN}. Goussarov, Polyak and Viro discussed finite type invariants of classical and virtual knots in \cite{GPV}. It is known~\cite[Ch~3.3]{CDM} that no non-trivial Vassiliev invariants of order one exist for classical knots. This property fails in the case of virtual knots, since Sawollek~\cite{Saw} constructed Vassiliev invariants of order one distinguishing virtual knots from their inverses. Im, Kim, and Lee~\cite{IKL} introduced a Vassiliev invariant for virtual knots and flat virtual knots. Subsequent studies have also been conducted on the Vassiliev invariants of virtual knots. Moltmaker and Kauffman discussed Vassiliev invariants for virtual knots in~\cite{MK}. 

In~\cite{Tch}, Tchernov introduced the universal Vassiliev invariant of order one of knots. Subsequently, Henrich \cite{Hen} and Petit \cite{Pet} extended finite-type invariants of order one and their universality to virtual knots, and framed virtual knots, respectively. In particular, Henrich provided a  construction giving a ``gluing invariant'' that is the universal Vassiliev invariant of order one for virtual knots. As noted in~\cite[Remark 7.9]{MLK}, one can adapt this construction for knotoids by considering the virtual closure of the knotoid.

Cheng~\cite{Che} established the chord index axiom, assigning integers to classical crossings via arc weights to generate invariants for virtual knots. In~\cite{CGX}, Cheng, Gao and Xu defined chord indices via smoothing classical crossings, yielding Vassiliev invariants valued in flat virtual knot diagrams. Moreover, Gill, Ivanov, Prabhakar and Vesnin introduced weight function, which may be thought of as a generalization of the chord index axioms in~\cite{GIPV}.  

In this paper we study Vassiliev invariant for virtual knotoids. We recall basis definitions in Section~\ref{section2}. In Section~\ref{section3} we define a polynomial $\mathcal{F}$ of virtual knotoids by formula~(\ref{eqn:defF}), which is valued in the free $\mathbb Z$-module generated by non-oriented flat virtual knotoids and prove that $\mathcal{F}$ is a virtual knotoid invariant, see Theorem~\ref{th3.1}. In Theorem~\ref{th3.3} we describe how the invariant changes when a knotoid diagram is replaced by the orientation reversed image and the mirror image. Furthermore, we prove that $\mathcal{F}$ is a Vassiliev invariant of order one, see Theorem~\ref{th3.4}. It was observed by Henrich~\cite[Th.~3.11]{Hen}, that the smoothing invariant she constructed for virtual knots provides less information than the gluing invariant. In Section~\ref{section4}, we demonstrate that the same property holds for virtual knotoids. Namely, it is shown in Theorem~\ref{th4.2}, that the $0$-smoothing invariant $\mathcal{F}$ provides less information than the gluing invariant given by Petit~\cite{Pet}, see formula~(\ref{eqn:defG}), which is known the universal Vassiliev invariant of order one. To prove Theorem~\ref{th4.2}, we use the extension of the singular based matrix invariant originally introduced by Turaev for singular virtual strings.

\section{Preliminaries} \label{section2}

\subsection{Classical knotoids}
In this subsection, we recall necessary definitions and fundamental results about knotoids will be used below, see \cite{GK}. 

\begin{definition}{\rm 
A \textit{knotoid diagram} $D$ in an oriented surface $\Sigma$ is a generic immersion of the unit segment $[0,1]$ into $\Sigma$ with finitely many transversal double points endowed with over/under-crossing information. Such a double point is called a \textit{classical crossing} of $D$. The \textit{endpoints} of $D$ are the images of $0$ and $1$, called the \textit{tail} and \textit{head}, respectively, which are assumed to be different from each other and any crossing. The \textit{orientation} of $D$ is from the tail to the head. A \textit{trivial knotoid diagram} is an embedding of the unit interval into $\Sigma$. 
} 
\end{definition}

Few examples of knotoid diagrams are shown in Fig.~\ref{fig1}, where (a) presents a trivial knotoid diagram.
\begin{figure}[htbp]
\begin{center}
\tikzset{every picture/.style={line width=1.0pt}}
\begin{tikzpicture}[x=0.75pt,y=0.75pt,yscale=-1.0,xscale=1.0] 
\draw    (9.66,75.66) .. controls (49.66,45.66) and (69.66,105.66) .. (109.66,75.66) ;
\draw  [fill={rgb, 255:red, 13; green, 13; blue, 13 }  ,fill opacity=1 ] (8,75.66) .. controls (8,74.74) and (8.74,74) .. (9.66,74) .. controls (10.57,74) and (11.31,74.74) .. (11.31,75.66) .. controls (11.31,76.57) and (10.57,77.31) .. (9.66,77.31) .. controls (8.74,77.31) and (8,76.57) .. (8,75.66) -- cycle ;
\draw  [fill={rgb, 255:red, 13; green, 13; blue, 13 }  ,fill opacity=1 ] (108,75.66) .. controls (108,74.74) and (108.74,74) .. (109.66,74) .. controls (110.57,74) and (111.31,74.74) .. (111.31,75.66) .. controls (111.31,76.57) and (110.57,77.31) .. (109.66,77.31) .. controls (108.74,77.31) and (108,76.57) .. (108,75.66) -- cycle ;
\draw   (55.14,69.34) -- (59.63,76.04) -- (51.59,76.51) ;
\draw    (175.2,101.38) .. controls (152.09,103.82) and (127.27,70.84) .. (142.07,46.04) .. controls (156.87,21.24) and (184.09,31.16) .. (186.09,57.16) ;
\draw    (186.31,99.6) .. controls (260.98,87.16) and (213.56,-8.58) .. (166.76,27.82) ;
\draw    (159.76,37.44) .. controls (153.36,51.04) and (180.07,70.1) .. (200,59.14) ;
\draw  [fill={rgb, 255:red, 13; green, 13; blue, 13 }  ,fill opacity=1 ] (198.34,59.14) .. controls (198.34,58.22) and (199.08,57.48) .. (200,57.48) .. controls (200.92,57.48) and (201.66,58.22) .. (201.66,59.14) .. controls (201.66,60.05) and (200.92,60.8) .. (200,60.8) .. controls (199.08,60.8) and (198.34,60.05) .. (198.34,59.14) -- cycle ;
\draw  [fill={rgb, 255:red, 13; green, 13; blue, 13 }  ,fill opacity=1 ] (179.05,113.84) .. controls (179.05,112.92) and (179.79,112.18) .. (180.71,112.18) .. controls (181.62,112.18) and (182.37,112.92) .. (182.37,113.84) .. controls (182.37,114.75) and (181.62,115.49) .. (180.71,115.49) .. controls (179.79,115.49) and (179.05,114.75) .. (179.05,113.84) -- cycle ;
\draw    (272.76,87.6) .. controls (255.57,80.28) and (258.26,49.62) .. (291.85,57.54) ;
\draw    (342.57,94.08) .. controls (358.86,88.36) and (368.88,66.28) .. (348.76,55.38) ;
\draw    (331.76,90.64) .. controls (318.76,71.64) and (320.7,47.71) .. (339.2,52.71) ;
\draw    (281.42,62.93) .. controls (258.92,135) and (346.93,127.11) .. (336.56,101.04) ;
\draw    (284.31,50.93) .. controls (295.81,3.36) and (354.37,24.25) .. (343.27,62.41) ;
\draw   (187.56,85.69) -- (182.45,91.93) -- (179.67,84.36) ;
\draw    (283.2,91.6) .. controls (293.2,95.6) and (316,103.54) .. (342.57,94.08) ;
\draw   (302.68,93.22) -- (309.03,98.18) -- (301.53,101.14) ;
\draw  [fill={rgb, 255:red, 13; green, 13; blue, 13 }  ,fill opacity=1 ] (290.19,57.54) .. controls (290.19,56.62) and (290.93,55.88) .. (291.85,55.88) .. controls (292.76,55.88) and (293.5,56.62) .. (293.5,57.54) .. controls (293.5,58.45) and (292.76,59.2) .. (291.85,59.2) .. controls (290.93,59.2) and (290.19,58.45) .. (290.19,57.54) -- cycle ;
\draw  [fill={rgb, 255:red, 13; green, 13; blue, 13 }  ,fill opacity=1 ] (341.61,62.41) .. controls (341.61,61.49) and (342.35,60.75) .. (343.27,60.75) .. controls (344.18,60.75) and (344.93,61.49) .. (344.93,62.41) .. controls (344.93,63.33) and (344.18,64.07) .. (343.27,64.07) .. controls (342.35,64.07) and (341.61,63.33) .. (341.61,62.41) -- cycle ;
\draw    (185.87,68.3) .. controls (183.87,95.3) and (179.51,94.24) .. (180.71,113.84) ;
\draw (48,124) node [anchor=north west][inner sep=0.75pt]    {(a)};
\draw (175,124) node [anchor=north west][inner sep=0.75pt]    {(b)};
\draw (305,124) node [anchor=north west][inner sep=0.75pt]    {(c)};
\end{tikzpicture}
\caption{Examples of knotoid diagrams.\label{fig1}}
\end{center}
\end{figure}
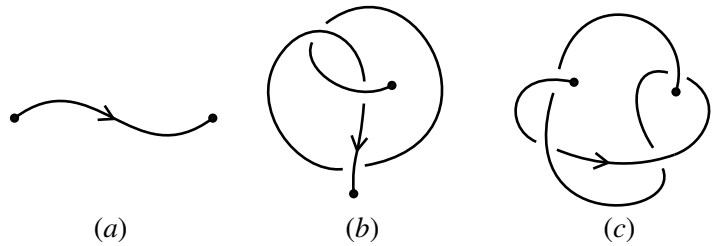

Analogously to the classical knot theory, three \textit{Reidemeister moves} $\mathcal{R} = \{ \Omega_1, \Omega_2, \Omega_3\}$ are presented in Fig.~\ref{fig2} are defined on knotoid diagrams. These moves modify a knotoid diagram within small disks surrounding the local diagrammatic regions and it is assumed that they do not utilize the endpoints. Since arcs of knotoids are oriented, we consider oriented versions of Reidemeister moves with all possible orientations. 
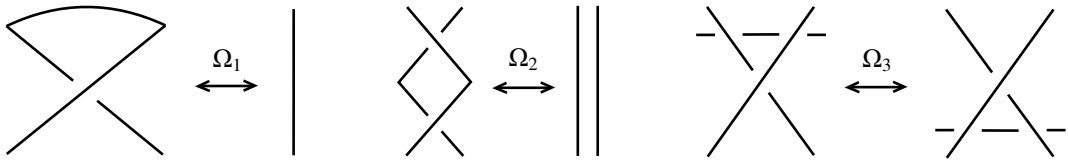
\begin{figure}[htbp]
\begin{center}
\tikzset{every picture/.style={line width=1.0pt}}
\begin{tikzpicture}[x=0.75pt,y=0.75pt,yscale=-1.0,xscale=1.0]
\draw    (224.91,241.03) -- (235.83,253.48) ;
\draw    (217.99,199.06) -- (203.27,216.33) -- (217.99,233.18) ;
\draw    (235.4,178.53) -- (224.81,191.03) ;
\draw    (5.45,187.97) .. controls (32.28,175.15) and (58.96,175.15) .. (85.64,187.61) ;
\draw    (5.6,252.51) -- (85.64,187.61) ;
\draw    (5.45,187.97) -- (39.09,215.52) ;
\draw   (125,215.49) -- (130.7,218.02) -- (125,220.54) ;
\draw    (129.99,218.18) -- (102.89,218.18) ;
\draw   (108,220.72) -- (102.18,218.31) -- (107.76,215.67) ;
\draw    (150.31,179.38) -- (150.27,253.15) ;
\draw    (207.08,253.3) -- (239.64,216.15) -- (207.51,178.35) ;
\draw    (293.6,177.71) -- (293.6,253.48) ;
\draw    (303.64,177.71) -- (303.64,253.48) ;
\draw    (412.73,178.24) -- (359.11,253.81) ;
\draw    (353.3,192.39) -- (362.79,192.39) ;
\draw    (51.22,225.64) -- (84.35,252.81) ;
\draw    (358.77,178.24) -- (381.86,210.31) ;
\draw    (389.9,221.57) -- (412.69,253.1) ;
\draw    (376.67,192.21) -- (395.2,192.14) ;
\draw    (409.29,192.37) -- (418.77,192.37) ;
\draw   (274.89,216.4) -- (280.59,218.93) -- (274.89,221.46) ;
\draw    (279.88,219.09) -- (252.78,219.09) ;
\draw   (257.89,221.64) -- (252.07,219.23) -- (257.65,216.58) ;
\draw   (452.6,216.07) -- (458.31,218.6) -- (452.6,221.13) ;
\draw    (457.59,218.76) -- (430.49,218.76) ;
\draw   (435.6,221.3) -- (429.78,218.89) -- (435.36,216.25) ;
\draw    (533.51,179.58) -- (479.71,254.48) ;
\draw    (473.88,240.46) -- (483.4,240.45) ;
\draw    (479.36,179.58) -- (502.54,211.37) ;
\draw    (510.6,222.53) -- (533.47,253.77) ;
\draw    (497.32,240.63) -- (515.92,240.7) ;
\draw    (530.06,240.47) -- (539.58,240.47) ;
\draw (107.88,197.62) node [anchor=north west][inner sep=0.75pt]  [font=\small]  {$\Omega_1$};
\draw (435.47,197.9) node [anchor=north west][inner sep=0.75pt]  [font=\small]  {$\Omega_3$};
\draw (257.23,197.56) node [anchor=north west][inner sep=0.75pt]  [font=\small]  {$\Omega_2$};
\end{tikzpicture}
\caption{Reidemeister moves $\Omega_1$, $\Omega_2$ and $\Omega_3$.\label{fig2}}
\end{center}
\end{figure}

\begin{definition} {\rm 
Two knotoid diagrams in $\Sigma$ are said to be \textit{equivalent} if they are related to each other by a finite sequence of moves, where each move is either from $\mathcal{R}$ or an isotopy of $\Sigma$. The equivalence classes of knotoid diagrams are called \textit{knotoids}. If $\Sigma$ is either $S^2$ or $\mathbb{R}^2$ then the knotoid is said to be \textit{spherical} or \textit{planar}, respectively. 
} 
\end{definition}

\begin{definition} {\rm 
Moves $\Phi_{+}$ and $\Phi_{-}$ of knotoid diagrams, presented in Fig.~\ref{fig3},  are called \textit{forbidden knotoid moves}.
}
\end{definition}

The forbidden knotoid moves are not allowed since, obviously, by applying a finite number of  both $\Phi_{+}$ and $\Phi_{-}$ moves any knotoid diagram in $\Sigma$ can be transformed into the trivial knotoid diagram.
\begin{figure}[htbp]
\begin{center}
\tikzset{every picture/.style={line width=1.0pt}} 
\begin{tikzpicture}[x=0.75pt,y=0.75pt,yscale=-1.0,xscale=1.0]  			
\draw    (42.82,349.32) -- (117.64,348.82) ;
\draw    (80.96,304.68) -- (81.26,339.78) ;
\draw    (81.26,358.86) -- (81.26,390.55) ;
\draw    (205.82,350.39) -- (238.57,350.39) ;
\draw    (255.07,305.75) -- (255.64,392.55) ;
\draw    (401.07,350.18) -- (431.11,350.28) ;
\draw    (392.04,306.06) -- (392.57,391.55) ;
\draw    (349.64,350.29) -- (380.38,350.28) ;
\draw   (170.33,351.01) -- (176.04,353.54) -- (170.33,356.06) ;
\draw    (175.32,353.7) -- (148.22,353.7) ;
\draw   (153.33,356.24) -- (147.51,353.83) -- (153.1,351.19) ;
\draw   (317,350.34) -- (322.7,352.87) -- (317,355.4) ;
\draw    (321.99,353.03) -- (294.89,353.03) ;
\draw   (300,355.58) -- (294.18,353.17) -- (299.76,350.52) ;
\draw  [fill={rgb, 255:red, 13; green, 13; blue, 13 }  ,fill opacity=1 ] (115.98,348.82) .. controls (115.98,347.9) and (116.72,347.16) .. (117.64,347.16) .. controls (118.55,347.16) and (119.29,347.9) .. (119.29,348.82) .. controls (119.29,349.73) and (118.55,350.47) .. (117.64,350.47) .. controls (116.72,350.47) and (115.98,349.73) .. (115.98,348.82) -- cycle ;
\draw  [fill={rgb, 255:red, 13; green, 13; blue, 13 }  ,fill opacity=1 ] (236.91,350.39) .. controls (236.91,349.48) and (237.65,348.73) .. (238.57,348.73) .. controls (239.48,348.73) and (240.23,349.48) .. (240.23,350.39) .. controls (240.23,351.31) and (239.48,352.05) .. (238.57,352.05) .. controls (237.65,352.05) and (236.91,351.31) .. (236.91,350.39) -- cycle ;
\draw  [fill={rgb, 255:red, 13; green, 13; blue, 13 }  ,fill opacity=1 ] (429.45,350.28) .. controls (429.45,349.36) and (430.2,348.62) .. (431.11,348.62) .. controls (432.03,348.62) and (432.77,349.36) .. (432.77,350.28) .. controls (432.77,351.19) and (432.03,351.93) .. (431.11,351.93) .. controls (430.2,351.93) and (429.45,351.19) .. (429.45,350.28) -- cycle ;
\draw (152.97,333.09) node [anchor=north west][inner sep=0.75pt]  [font=\small]  {$\Phi _{+}$};
\draw (299.63,333.63) node [anchor=north west][inner sep=0.75pt]  [font=\small]  {$\Phi _{-}$};
\end{tikzpicture}
\caption{Forbidden knotoid moves $\Phi_{+}$ and $\Phi_{-}$.\label{fig3}}	
\end{center}
\end{figure}
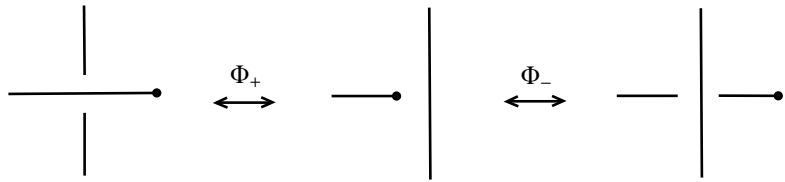

\subsection{Virtual knotoids}

It was mentioned in~\cite{Tur1} that notions of a virtual knot diagram and a virtual knot can be expressed to knotoids. We will follow terminology from~\cite{GK}.

\begin{definition} {\rm 	
A \textit{virtual knotoid diagram} is a knotoid diagram in $S^2$ with an extra combinatorial structure called  \textit{virtual crossings}. 
}
\end{definition}

A virtual knotoid diagram has two types of crossings: classical, see Fig.~\ref{fig4}~(a), and virtual, indicated by an enclosed circle without over/under-crossing information, see Fig.~\ref{fig4}~(b). 
\begin{figure}[!ht]
\begin{center}
\tikzset{every picture/.style={line width=1.0pt}} 
\begin{tikzpicture}[x=0.75pt,y=0.75pt,yscale=-1,xscale=1]
\draw    (201.26,306.05) -- (142.17,366.42) ;
\draw   (142.41,311.9) -- (141.66,305.97) -- (147.43,307.17) ;
\draw   (195.79,307.32) -- (201.51,305.91) -- (200.98,311.87) ;
\draw    (176.09,340.88) -- (203.09,367.71) ;
\draw    (288.81,306.61) -- (264.76,331.15) ;
\draw    (229.6,306.67) -- (288.93,366.92) ;
\draw   (230.25,312.46) -- (229.51,306.53) -- (235.27,307.73) ;
\draw   (283.34,307.88) -- (289.06,306.47) -- (288.53,312.43) ;
\draw    (254.09,342.15) -- (229.72,366.98) ;
\draw    (398.82,306.85) -- (339.74,367.22) ;
\draw    (339.62,306.91) -- (398.94,367.16) ;
\draw   (340.27,312.7) -- (339.52,306.77) -- (345.29,307.97) ;
\draw   (393.35,308.12) -- (399.07,306.71) -- (398.54,312.67) ;
\draw    (141.17,305.66) -- (166.76,331.54) ;
\draw    (500.82,306.83) -- (441.74,367.21) ;
\draw    (441.62,306.9) -- (500.94,367.14) ;
\draw   (442.27,312.69) -- (441.52,306.76) -- (447.29,307.95) ;
\draw   (495.35,308.11) -- (501.07,306.7) -- (500.54,312.65) ;
\draw   (363.86,337.04) .. controls (363.86,334.04) and (366.29,331.62) .. (369.28,331.62) .. controls (372.27,331.62) and (374.7,334.04) .. (374.7,337.04) .. controls (374.7,340.03) and (372.27,342.45) .. (369.28,342.45) .. controls (366.29,342.45) and (363.86,340.03) .. (363.86,337.04) -- cycle ;
\draw    (596.82,306.65) -- (537.74,367.03) ;
\draw    (537.62,306.72) -- (596.94,366.96) ;
\draw   (538.27,312.51) -- (537.52,306.58) -- (543.29,307.77) ;
\draw   (591.35,307.93) -- (597.07,306.52) -- (596.54,312.47) ;
\draw  [fill={rgb, 255:red, 14; green, 13; blue, 13 }  ,fill opacity=1 ] (561.86,336.84) .. controls (561.86,333.85) and (564.29,331.42) .. (567.28,331.42) .. controls (570.27,331.42) and (572.7,333.85) .. (572.7,336.84) .. controls (572.7,339.83) and (570.27,342.26) .. (567.28,342.26) .. controls (564.29,342.26) and (561.86,339.83) .. (561.86,336.84) -- cycle ;
\draw (179.67,331.25) node [anchor=north west][inner sep=0.75pt]  [font=\large]  {$c$};
\draw (267.6,331.25) node [anchor=north west][inner sep=0.75pt]  [font=\large]  {$c$};
\draw (143.32,370.02) node [anchor=north west][inner sep=0.75pt]  [font=\small]  {$\text{sgn}(c) =1$};
\draw (226.44,370.02) node [anchor=north west][inner sep=0.75pt]  [font=\small]  {$\text{sgn}(c) =-1$};
\draw (205.83,387.01) node [anchor=north west][inner sep=0.75pt]  [font=\large] [align=left] {{\fontfamily{ptm}\selectfont (a)}};
\draw (360.53,387.01) node [anchor=north west][inner sep=0.75pt]  [font=\large] [align=left] {{\fontfamily{ptm}\selectfont (b)}};
\draw (462.53,387.01) node [anchor=north west][inner sep=0.75pt]  [font=\large] [align=left] {{\fontfamily{ptm}\selectfont (c)}};
\draw (558.53,387.01) node [anchor=north west][inner sep=0.75pt]  [font=\large] [align=left] {{\fontfamily{ptm}\selectfont (d)}};
\end{tikzpicture}
\caption{Crossings: (a) classical, (b) virtual, (c) flat and (d) singular.} \label{fig4}
\end{center}
\end{figure}
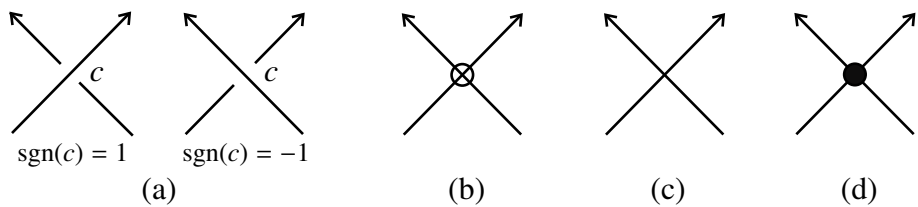
For any classical crossing $c$ there is defined a sign function $\operatorname{sgn}(c) \in \{ 1,-1\}$ as presented in Fig.~\ref{fig4}~(a). For a knotoid diagram $D$, either classical or virtual, the \textit{writhe} $w(D)$ is defined as 
\begin{equation}
w(D) = \sum_{c \in D} \mathrm{sgn} (c), 
\end{equation}
where the sum is taken over all classical crossings.

By \textit{generalized Reidemeister moves} we call a collection 
$
\textit{g} \mathcal{R} = \{ \Omega_1, \Omega_2, \Omega_3, \Omega_1^v, \Omega_2^v, \Omega_3^v, \Omega_3^m, \Omega_v \} 
$  
of Reidemeister moves $\{\Omega_1, \Omega_2, \Omega_3\}$, see Fig.~\ref{fig2}, virtual Reidemeister moves $\{ \Omega_1^v, \Omega_2^v, \Omega_3^v\}$, the mixed move $\Omega_3^{m}$ in which classical and virtual crossings are involved, and the move $\Omega_v$ enables to slide the strand which is adjacent to the tail or the head, see Fig.~\ref{fig5}. Since arcs of virtual knotoids are oriented, we consider oriented versions of generalized Reidemeister moves with all possible orientations.  
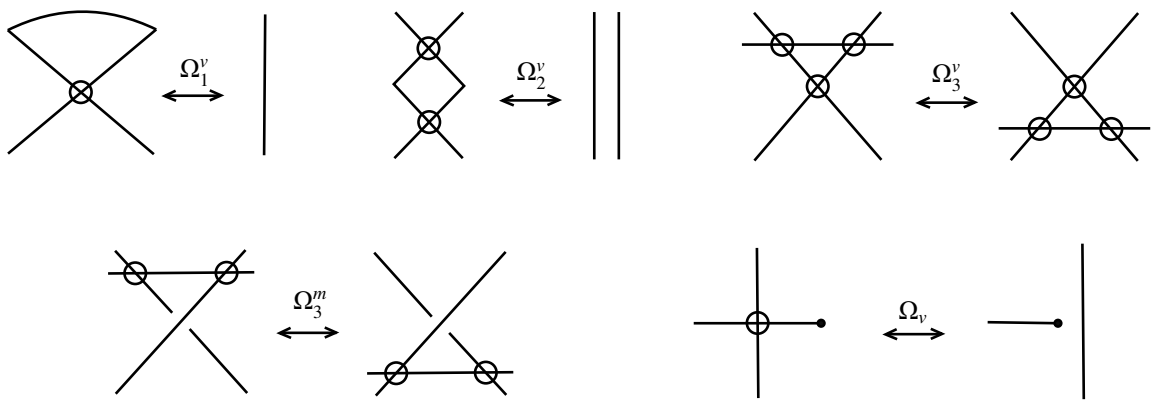
\begin{figure}[!htbp]
\begin{center}
\tikzset{every picture/.style={line width=1.0pt}} 
\begin{tikzpicture}[x=0.75pt,y=0.75pt,yscale=-1,xscale=1]	
\draw    (13.45,486.3) -- (87.13,548.9) ;
\draw    (143.28,478.35) -- (142.73,549.43) ;
\draw    (208.43,550.92) -- (243.59,514.45) -- (208.9,477.32) ;
\draw    (243.12,477.36) -- (208.15,513.55) -- (242.92,550.92) ;
\draw    (309.25,477.16) -- (309.25,551.52) ;
\draw    (321.63,477.16) -- (321.63,551.52) ;
\draw    (453.88,477.16) -- (390.01,552.12) ;
\draw    (389.82,477.51) -- (454.07,551.76) ;
\draw    (383.71,493.57) -- (460.16,493.57) ;
\draw    (133.36,597.29) -- (67.91,669.82) ;
\draw    (64.04,609.1) -- (137.77,608.5) ;
\draw    (359.4,634.22) -- (423.66,634.22) ;
\draw    (391.28,596.2) -- (391.98,672) ;
\draw    (507.64,633.7) -- (543.11,634.22) ;
\draw    (555.43,594.06) -- (556.13,672.55) ;
\draw   (434.75,493.5) .. controls (434.75,490.51) and (437.17,488.08) .. (440.17,488.08) .. controls (443.16,488.08) and (445.59,490.51) .. (445.59,493.5) .. controls (445.59,496.49) and (443.16,498.92) .. (440.17,498.92) .. controls (437.17,498.92) and (434.75,496.49) .. (434.75,493.5) -- cycle ;
\draw   (398.64,493.69) .. controls (398.64,490.7) and (401.06,488.27) .. (404.06,488.27) .. controls (407.05,488.27) and (409.47,490.7) .. (409.47,493.69) .. controls (409.47,496.68) and (407.05,499.11) .. (404.06,499.11) .. controls (401.06,499.11) and (398.64,496.68) .. (398.64,493.69) -- cycle ;
\draw   (416.53,514.64) .. controls (416.53,511.64) and (418.95,509.22) .. (421.94,509.22) .. controls (424.94,509.22) and (427.36,511.64) .. (427.36,514.64) .. controls (427.36,517.63) and (424.94,520.06) .. (421.94,520.06) .. controls (418.95,520.06) and (416.53,517.63) .. (416.53,514.64) -- cycle ;
\draw   (203.92,659.45) .. controls (203.92,656.46) and (206.35,654.04) .. (209.34,654.04) .. controls (212.33,654.04) and (214.76,656.46) .. (214.76,659.45) .. controls (214.76,662.45) and (212.33,664.87) .. (209.34,664.87) .. controls (206.35,664.87) and (203.92,662.45) .. (203.92,659.45) -- cycle ;
\draw   (220.22,495.46) .. controls (220.22,492.46) and (222.64,490.04) .. (225.64,490.04) .. controls (228.63,490.04) and (231.05,492.46) .. (231.05,495.46) .. controls (231.05,498.45) and (228.63,500.87) .. (225.64,500.87) .. controls (222.64,500.87) and (220.22,498.45) .. (220.22,495.46) -- cycle ;
\draw   (220.59,532.68) .. controls (220.59,529.69) and (223.02,527.27) .. (226.01,527.27) .. controls (229.01,527.27) and (231.43,529.69) .. (231.43,532.68) .. controls (231.43,535.68) and (229.01,538.1) .. (226.01,538.1) .. controls (223.02,538.1) and (220.59,535.68) .. (220.59,532.68) -- cycle ;
\draw   (386.21,634.1) .. controls (386.21,631.11) and (388.63,628.68) .. (391.63,628.68) .. controls (394.62,628.68) and (397.05,631.11) .. (397.05,634.1) .. controls (397.05,637.09) and (394.62,639.52) .. (391.63,639.52) .. controls (388.63,639.52) and (386.21,637.09) .. (386.21,634.1) -- cycle ;
\draw   (44.87,517.6) .. controls (44.87,514.61) and (47.3,512.18) .. (50.29,512.18) .. controls (53.28,512.18) and (55.71,514.61) .. (55.71,517.6) .. controls (55.71,520.59) and (53.28,523.02) .. (50.29,523.02) .. controls (47.3,523.02) and (44.87,520.59) .. (44.87,517.6) -- cycle ;
\draw    (67.58,597.64) -- (96.38,628.94) ;
\draw    (105.17,637.14) -- (134.2,669.45) ;
\draw   (249.12,659.15) .. controls (249.12,656.15) and (251.55,653.73) .. (254.54,653.73) .. controls (257.53,653.73) and (259.96,656.15) .. (259.96,659.15) .. controls (259.96,662.14) and (257.53,664.56) .. (254.54,664.56) .. controls (251.55,664.56) and (249.12,662.14) .. (249.12,659.15) -- cycle ;
\draw    (13.45,486.3) .. controls (38.51,473.92) and (63.42,473.92) .. (88.33,485.95) ;
\draw    (13.59,548.61) -- (88.33,485.95) ;
\draw   (115.33,516.7) -- (121.04,519.22) -- (115.33,521.75) ;
\draw    (120.32,519.39) -- (93.22,519.39) ;
\draw   (98.33,521.93) -- (92.51,519.52) -- (98.1,516.88) ;
\draw   (285.65,519.05) -- (291.36,521.58) -- (285.65,524.11) ;
\draw    (290.64,521.74) -- (263.54,521.74) ;
\draw   (268.65,524.28) -- (262.83,521.87) -- (268.41,519.23) ;
\draw    (583.36,552.27) -- (519.48,477.31) ;
\draw    (519.3,551.92) -- (583.54,477.67) ;
\draw    (513.19,535.86) -- (589.64,535.86) ;
\draw   (528.47,536.11) .. controls (528.47,539.11) and (530.9,541.53) .. (533.89,541.53) .. controls (536.88,541.53) and (539.31,539.11) .. (539.31,536.11) .. controls (539.31,533.12) and (536.88,530.7) .. (533.89,530.7) .. controls (530.9,530.7) and (528.47,533.12) .. (528.47,536.11) -- cycle ;
\draw   (564.61,536.17) .. controls (564.61,539.17) and (567.03,541.59) .. (570.03,541.59) .. controls (573.02,541.59) and (575.45,539.17) .. (575.45,536.17) .. controls (575.45,533.18) and (573.02,530.75) .. (570.03,530.75) .. controls (567.03,530.75) and (564.61,533.18) .. (564.61,536.17) -- cycle ;
\draw   (546,514.79) .. controls (546,517.78) and (548.43,520.21) .. (551.42,520.21) .. controls (554.41,520.21) and (556.84,517.78) .. (556.84,514.79) .. controls (556.84,511.8) and (554.41,509.37) .. (551.42,509.37) .. controls (548.43,509.37) and (546,511.8) .. (546,514.79) -- cycle ;
\draw   (495.58,520.19) -- (501.29,522.72) -- (495.58,525.25) ;
\draw    (500.58,522.88) -- (473.47,522.88) ;
\draw   (478.58,525.42) -- (472.76,523.01) -- (478.35,520.37) ;
\draw    (199.09,670.76) -- (264.55,598.22) ;
\draw    (268.41,658.94) -- (194.68,659.55) ;
\draw   (83.15,609.01) .. controls (83.15,612) and (80.73,614.42) .. (77.73,614.42) .. controls (74.74,614.42) and (72.32,612) .. (72.32,609.01) .. controls (72.32,606.01) and (74.74,603.59) .. (77.73,603.59) .. controls (80.73,603.59) and (83.15,606.01) .. (83.15,609.01) -- cycle ;
\draw   (129.2,608.76) .. controls (129.2,611.75) and (126.78,614.17) .. (123.78,614.17) .. controls (120.79,614.17) and (118.37,611.75) .. (118.37,608.76) .. controls (118.37,605.76) and (120.79,603.34) .. (123.78,603.34) .. controls (126.78,603.34) and (129.2,605.76) .. (129.2,608.76) -- cycle ;
\draw    (264.87,670.4) -- (236.07,639.1) ;
\draw    (227.29,630.91) -- (198.26,598.59) ;
\draw   (173.32,636.38) -- (179.02,638.91) -- (173.32,641.44) ;
\draw    (179.31,639.07) -- (152.21,639.07) ;
\draw   (156.32,641.62) -- (150.5,639.21) -- (156.08,636.56) ;
\draw   (478.98,637.05) -- (484.69,639.58) -- (478.98,642.11) ;
\draw    (483.98,639.74) -- (456.87,639.74) ;
\draw   (461.98,642.28) -- (456.16,639.87) -- (461.75,637.23) ;
\draw  [fill={rgb, 255:red, 13; green, 13; blue, 13 }  ,fill opacity=1 ] (422,634.22) .. controls (422,633.31) and (422.75,632.57) .. (423.66,632.57) .. controls (424.58,632.57) and (425.32,633.31) .. (425.32,634.22) .. controls (425.32,635.14) and (424.58,635.88) .. (423.66,635.88) .. controls (422.75,635.88) and (422,635.14) .. (422,634.22) -- cycle ;
\draw  [fill={rgb, 255:red, 13; green, 13; blue, 13 }  ,fill opacity=1 ] (541.45,634.22) .. controls (541.45,633.31) and (542.19,632.57) .. (543.11,632.57) .. controls (544.02,632.57) and (544.76,633.31) .. (544.76,634.22) .. controls (544.76,635.14) and (544.02,635.88) .. (543.11,635.88) .. controls (542.19,635.88) and (541.45,635.14) .. (541.45,634.22) -- cycle ;

\draw (98.76,498.96) node [anchor=north west][inner sep=0.75pt]  [font=\small]  {$\Omega_{1}^{v}$};
\draw (477.95,501.53) node [anchor=north west][inner sep=0.75pt]  [font=\small]  {$\Omega_{3}^{v}$};
\draw (268.58,500.81) node [anchor=north west][inner sep=0.75pt]  [font=\small]  {$\Omega_{2}^{v}$};
\draw (155.83,616.28) node [anchor=north west][inner sep=0.75pt]  [font=\small]  {$\Omega_{3}^{m}$};
\draw (461.19,621.19) node [anchor=north west][inner sep=0.75pt]  [font=\small]  {$\Omega_{v}$};
\end{tikzpicture}
\caption{Virtual Reidemeister moves $\Omega_1^v$, $\Omega_2^v$, $\Omega_3^v$, the mixed move $\Omega_3^{m}$ and $\Omega_{v}$}-move. \label{fig5}
\end{center}
\end{figure}

\begin{definition}{\rm 
Two virtual knotoid diagrams in $S^2$ are said to be \textit{equivalent} if they are related to each other by a finite sequence of moves, where each move is either from $\textit{g} \mathcal{R}$ or an isotopy of $S^2$. Their equivalence classes are called \textit{virtual knotoids}. 
}
\end{definition}

It is useful to encode knotoid diagrams by \textit{Gauss diagrams}, see \cite{FL, FLV}. A Gauss diagram $G(K)$ of a knotoid diagram $K$ is a counterclockwise oriented arc with chords connecting pre-images of crossings of~$K$. The chord of $G(K)$ corresponding to a crossing $c\in K$ is also denoted by $c$. A chord gets a direction from the over-crossing to the under-crossing.  The starting point of a chord $c$ according to the $\operatorname{sgn}(c)$ of the corresponding crossing of~$K$, also denoted by  $\operatorname{sgn}(c)$, see Fig.~\ref{fig6}. 
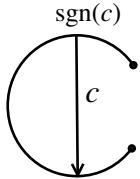
\begin{figure}[htbp]
\begin{center}
\tikzset{every picture/.style={line width=1.0pt}} 
\begin{tikzpicture}[x=0.75pt,y=0.75pt,yscale=-1.2,xscale=1.2]
\draw  [draw opacity=0] (293.09,96.51) .. controls (286.72,105.05) and (276.65,110.56) .. (265.31,110.56) .. controls (246.05,110.56) and (230.44,94.65) .. (230.44,75.03) .. controls (230.44,55.4) and (246.05,39.49) .. (265.31,39.49) .. controls (277.12,39.49) and (287.56,45.48) .. (293.87,54.63) -- (265.31,75.03) -- cycle ; \draw   (293.09,96.51) .. controls (286.72,105.05) and (276.65,110.56) .. (265.31,110.56) .. controls (246.05,110.56) and (230.44,94.65) .. (230.44,75.03) .. controls (230.44,55.4) and (246.05,39.49) .. (265.31,39.49) .. controls (277.12,39.49) and (287.56,45.48) .. (293.87,54.63) ;  
\draw    (265.05,39.87) -- (265.73,110.21) ;
\draw   (268.58,104.36) -- (265.78,110.05) -- (262.83,104.44) ;
\draw  [fill={rgb, 255:red, 13; green, 13; blue, 13 }  ,fill opacity=1 ] (292.36,54.56) .. controls (292.36,53.75) and (293.01,53.09) .. (293.82,53.09) .. controls (294.62,53.09) and (295.28,53.75) .. (295.28,54.56) .. controls (295.28,55.37) and (294.62,56.02) .. (293.82,56.02) .. controls (293.01,56.02) and (292.36,55.37) .. (292.36,54.56) -- cycle ;
\draw  [fill={rgb, 255:red, 13; green, 13; blue, 13 }  ,fill opacity=1 ] (291.58,96.59) .. controls (291.58,95.78) and (292.23,95.12) .. (293.03,95.12) .. controls (293.84,95.12) and (294.49,95.78) .. (294.49,96.59) .. controls (294.49,97.4) and (293.84,98.05) .. (293.03,98.05) .. controls (292.23,98.05) and (291.58,97.4) .. (291.58,96.59) -- cycle ;
\draw (252.7,21.59) node [anchor=north west][inner sep=0.75pt]  [font=\small]  {${\rm sgn}(c)$};
\draw (268.07,64.54) node [anchor=north west][inner sep=0.75pt]  [font=\large]  {$c$};				
\end{tikzpicture}
\caption{Elements of a Gauss diagram $G(K)$.\label{fig6}}		
\end{center}
\end{figure}

Virtual knotoid diagrams can also be encoded by Gauss diagrams. Namely, the Gauss diagram of a virtual knotoid is constructed in the same way as for a classical knotoid, but all virtual crossings are disregarded.

One can consider a weaker equivalence relation of virtual knotoids that is related to the homotopy of strings in thickened surfaces. While homotopy is a trivial notion of equivalence for classical knotoids, it is highly non-trivial for virtual knotoids. 

\begin{definition} {\rm 
A \textit{crossing change} is a move that switches the under- and over-types of a crossing in a diagram.  
Two virtual knotoid diagrams are said to be \textit{homotopic} if they are related by a finite sequence of moves, where each move is either from $\textit{g}\mathcal{R}$ or is a crossing changes. 
}
\end{definition}

The crossing change replaces a negative crossing for a positive one and vice versa, and it is equivalent to allowing one strand of a knotoid to pass through another. 

\subsection{Flat knotoids and flat virtual knotoids}

A crossing is called \textit{flat} if we do not distinguish the under- and over- types of the crossing. A flat crossing is indicated as in Fig.~\ref{fig4} (c).  

\begin{definition}{\rm 
A \textit{flat knotoid diagram} is a knotoid diagram in $\Sigma$ with each classical crossing replaced by a flat crossing.
}
\end{definition}

\textit{Flat Reidemeister moves} $\textit{f}\mathcal{R} = \{ f\Omega_1, f\Omega_2, f\Omega_3\}$ are defined on flat knotoid diagrams by ignoring the under/over-crossing information at the crossings of the move from $\mathcal{R}$, see Fig.~\ref{fig7}.

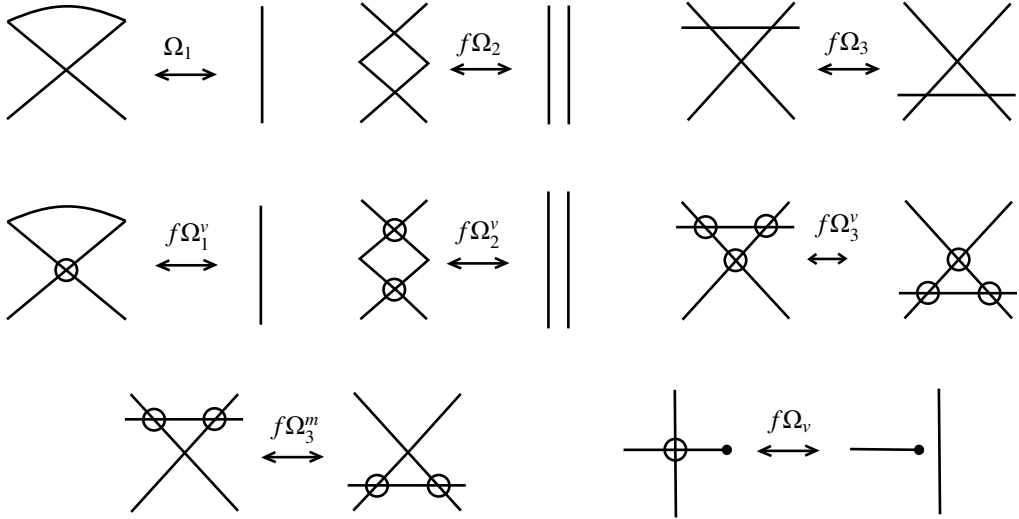
\begin{figure}[htbp]
\begin{center}
\tikzset{every picture/.style={line width=1.0pt}}  
\begin{tikzpicture}[x=0.75pt,y=0.75pt,yscale=-1,xscale=1]
\draw    (512.86,189.55) -- (567.06,130.63) ;
\draw    (60.45,140.7) -- (120.21,190.16) ;
\draw    (512.19,89.55) -- (566.39,30.63) ;
\draw    (566.43,89.56) -- (512.54,30.07) ;
\draw    (403.19,30.07) -- (457.39,89) ;
\draw    (400.53,130.74) -- (454.73,189.66) ;
\draw    (188.88,31.93) -- (188.88,91.04) ;
\draw    (238.43,90.65) -- (272.64,60.84) -- (238.89,30.49) ;
\draw    (257.47,43.91) -- (238.15,60.1) -- (257.47,76.83) ;
\draw    (272.19,30.52) -- (257.47,43.91) ;
\draw    (257.47,76.83) -- (271.98,90.65) ;
\draw    (333.52,31.63) -- (333.52,91.61) ;
\draw    (343.56,31.63) -- (343.56,91.61) ;
\draw    (457.43,30.07) -- (403.54,89.56) ;
\draw    (400.35,42.82) -- (459.93,42.81) ;
\draw    (188.24,132.64) -- (188.24,192.74) ;
\draw    (333.28,125.47) -- (333.28,194.58) ;
\draw    (343.32,125.47) -- (343.32,194.58) ;
\draw   (479.39,156.84) -- (483.02,159.36) -- (479.39,161.89) ;
\draw    (482.57,159.53) -- (465.34,159.53) ;
\draw   (468.58,162.07) -- (464.89,159.66) -- (468.43,157.02) ;
\draw   (84.71,165.16) .. controls (84.71,162.17) and (87.14,159.74) .. (90.13,159.74) .. controls (93.12,159.74) and (95.55,162.17) .. (95.55,165.16) .. controls (95.55,168.15) and (93.12,170.58) .. (90.13,170.58) .. controls (87.14,170.58) and (84.71,168.15) .. (84.71,165.16) -- cycle ;
\draw   (250.34,144.95) .. controls (250.34,141.96) and (252.77,139.53) .. (255.76,139.53) .. controls (258.76,139.53) and (261.18,141.96) .. (261.18,144.95) .. controls (261.18,147.94) and (258.76,150.37) .. (255.76,150.37) .. controls (252.77,150.37) and (250.34,147.94) .. (250.34,144.95) -- cycle ;
\draw   (250.12,175.03) .. controls (250.12,172.04) and (252.54,169.61) .. (255.54,169.61) .. controls (258.53,169.61) and (260.95,172.04) .. (260.95,175.03) .. controls (260.95,178.02) and (258.53,180.45) .. (255.54,180.45) .. controls (252.54,180.45) and (250.12,178.02) .. (250.12,175.03) -- cycle ;
\draw   (422.21,160.2) .. controls (422.21,157.21) and (424.64,154.78) .. (427.63,154.78) .. controls (430.62,154.78) and (433.05,157.21) .. (433.05,160.2) .. controls (433.05,163.19) and (430.62,165.62) .. (427.63,165.62) .. controls (424.64,165.62) and (422.21,163.19) .. (422.21,160.2) -- cycle ;
\draw   (437.68,142.85) .. controls (437.68,139.86) and (440.1,137.43) .. (443.1,137.43) .. controls (446.09,137.43) and (448.51,139.86) .. (448.51,142.85) .. controls (448.51,145.84) and (446.09,148.27) .. (443.1,148.27) .. controls (440.1,148.27) and (437.68,145.84) .. (437.68,142.85) -- cycle ;
\draw   (407.15,143.49) .. controls (407.15,140.49) and (409.58,138.07) .. (412.57,138.07) .. controls (415.56,138.07) and (417.99,140.49) .. (417.99,143.49) .. controls (417.99,146.48) and (415.56,148.9) .. (412.57,148.9) .. controls (409.58,148.9) and (407.15,146.48) .. (407.15,143.49) -- cycle ;
\draw   (550.35,176.95) .. controls (550.35,173.96) and (552.78,171.53) .. (555.77,171.53) .. controls (558.77,171.53) and (561.19,173.96) .. (561.19,176.95) .. controls (561.19,179.94) and (558.77,182.37) .. (555.77,182.37) .. controls (552.78,182.37) and (550.35,179.94) .. (550.35,176.95) -- cycle ;
\draw   (519.28,176.68) .. controls (519.28,173.68) and (521.7,171.26) .. (524.7,171.26) .. controls (527.69,171.26) and (530.12,173.68) .. (530.12,176.68) .. controls (530.12,179.67) and (527.69,182.09) .. (524.7,182.09) .. controls (521.7,182.09) and (519.28,179.67) .. (519.28,176.68) -- cycle ;
\draw   (534.74,159.81) .. controls (534.74,156.82) and (537.16,154.39) .. (540.15,154.39) .. controls (543.15,154.39) and (545.57,156.82) .. (545.57,159.81) .. controls (545.57,162.81) and (543.15,165.23) .. (540.15,165.23) .. controls (537.16,165.23) and (534.74,162.81) .. (534.74,159.81) -- cycle ;
\draw    (371.14,255.89) -- (423.28,255.89) ;
\draw  [fill={rgb, 255:red, 13; green, 13; blue, 13 }  ,fill opacity=1 ] (421.4,255.89) .. controls (421.4,254.81) and (422.24,253.93) .. (423.28,253.93) .. controls (424.32,253.93) and (425.16,254.81) .. (425.16,255.89) .. controls (425.16,256.97) and (424.32,257.85) .. (423.28,257.85) .. controls (422.24,257.85) and (421.4,256.97) .. (421.4,255.89) -- cycle ;
\draw    (397.01,225.6) -- (397.57,290) ;
\draw    (485.42,255.5) -- (520.19,255.89) ;
\draw  [fill={rgb, 255:red, 13; green, 13; blue, 13 }  ,fill opacity=1 ] (518.32,255.89) .. controls (518.32,254.81) and (519.16,253.93) .. (520.19,253.93) .. controls (521.23,253.93) and (522.07,254.81) .. (522.07,255.89) .. controls (522.07,256.97) and (521.23,257.85) .. (520.19,257.85) .. controls (519.16,257.85) and (518.32,256.97) .. (518.32,255.89) -- cycle ;
\draw    (530.19,225) -- (530.76,288.41) ;
\draw   (391.87,255.8) .. controls (391.87,252.81) and (394.3,250.38) .. (397.29,250.38) .. controls (400.28,250.38) and (402.71,252.81) .. (402.71,255.8) .. controls (402.71,258.79) and (400.28,261.22) .. (397.29,261.22) .. controls (394.3,261.22) and (391.87,258.79) .. (391.87,255.8) -- cycle ;
\draw    (509.35,76.81) -- (568.93,76.82) ;
\draw    (454.77,130.74) -- (400.88,190.22) ;
\draw    (397.68,143.48) -- (457.27,143.47) ;
\draw    (567.1,189.56) -- (513.21,130.07) ;
\draw    (510.01,176.81) -- (569.6,176.82) ;
\draw    (60.45,140.7) .. controls (83.24,130.56) and (97.38,130.46) .. (120.05,140.33) ;
\draw    (60.21,189.99) -- (120.05,140.33) ;
\draw    (238.43,189.94) -- (272.64,160.12) -- (238.89,129.78) ;
\draw    (257.47,143.19) -- (238.15,159.39) -- (257.47,176.12) ;
\draw    (272.19,129.81) -- (257.47,143.19) ;
\draw    (257.47,176.12) -- (271.98,189.94) ;
\draw    (122.53,227.74) -- (176.73,286.66) ;
\draw   (159.48,240.35) .. controls (159.48,237.36) and (161.9,234.93) .. (164.9,234.93) .. controls (167.89,234.93) and (170.31,237.36) .. (170.31,240.35) .. controls (170.31,243.34) and (167.89,245.77) .. (164.9,245.77) .. controls (161.9,245.77) and (159.48,243.34) .. (159.48,240.35) -- cycle ;
\draw   (128.95,240.69) .. controls (128.95,237.69) and (131.38,235.27) .. (134.37,235.27) .. controls (137.36,235.27) and (139.79,237.69) .. (139.79,240.69) .. controls (139.79,243.68) and (137.36,246.1) .. (134.37,246.1) .. controls (131.38,246.1) and (128.95,243.68) .. (128.95,240.69) -- cycle ;
\draw   (272.52,274.01) .. controls (272.52,271.02) and (274.95,268.6) .. (277.94,268.6) .. controls (280.93,268.6) and (283.36,271.02) .. (283.36,274.01) .. controls (283.36,277.01) and (280.93,279.43) .. (277.94,279.43) .. controls (274.95,279.43) and (272.52,277.01) .. (272.52,274.01) -- cycle ;
\draw   (241.41,274.01) .. controls (241.41,271.02) and (243.84,268.6) .. (246.83,268.6) .. controls (249.82,268.6) and (252.25,271.02) .. (252.25,274.01) .. controls (252.25,277.01) and (249.82,279.43) .. (246.83,279.43) .. controls (243.84,279.43) and (241.41,277.01) .. (241.41,274.01) -- cycle ;
\draw    (176.77,227.74) -- (122.88,287.22) ;
\draw    (119.68,240.48) -- (179.27,240.47) ;
\draw    (263.97,255.54) -- (289.06,227.63) ;
\draw    (289.1,286.56) -- (235.21,227.07) ;
\draw    (234.86,286.55) -- (263.97,255.54) ;
\draw    (232.01,273.81) -- (291.6,273.82) ;
\draw   (158.8,63.89) -- (164.5,66.42) -- (158.8,68.94) ;
\draw    (163.79,66.58) -- (136.69,66.58) ;
\draw   (141.8,69.12) -- (135.98,66.71) -- (141.56,64.07) ;
\draw   (307.8,61.09) -- (313.5,63.62) -- (307.8,66.14) ;
\draw    (312.79,63.78) -- (285.69,63.78) ;
\draw   (290.8,66.32) -- (284.98,63.91) -- (290.56,61.27) ;
\draw   (493.27,61.09) -- (498.98,63.62) -- (493.27,66.14) ;
\draw    (498.26,63.78) -- (471.16,63.78) ;
\draw   (476.27,66.32) -- (470.45,63.91) -- (476.04,61.27) ;
\draw   (159.3,160.05) -- (165,162.58) -- (159.3,165.11) ;
\draw    (164.29,162.74) -- (137.19,162.74) ;
\draw   (142.3,165.29) -- (136.48,162.88) -- (142.06,160.23) ;
\draw    (60.45,39.7) -- (120.21,89.16) ;
\draw    (60.45,39.7) .. controls (83.24,29.56) and (97.38,29.46) .. (120.05,39.33) ;
\draw    (60.21,88.99) -- (120.05,39.33) ;
\draw   (306.3,159.09) -- (312,161.62) -- (306.3,164.14) ;
\draw    (311.29,161.78) -- (284.19,161.78) ;
\draw   (289.3,164.32) -- (283.48,161.91) -- (289.06,159.27) ;
\draw   (212.8,256.88) -- (218.5,259.41) -- (212.8,261.93) ;
\draw    (217.79,259.57) -- (190.69,259.57) ;
\draw   (195.8,262.11) -- (189.98,259.7) -- (195.56,257.06) ;
\draw   (462.76,253.73) -- (468.47,256.25) -- (462.76,258.78) ;
\draw    (467.76,256.42) -- (440.65,256.42) ;
\draw   (445.76,258.96) -- (439.94,256.55) -- (445.53,253.91) ;

\draw (137.72,45.98) node [anchor=north west][inner sep=0.75pt]  [font=\small]  {$f\Omega_{1}$};
\draw (472.02,43.41) node [anchor=north west][inner sep=0.75pt]  [font=\small]  {$f\Omega_{3}$};
\draw (286.3,43.81) node [anchor=north west][inner sep=0.75pt]  [font=\small]  {$f\Omega_{2}$};
\draw (138.48,139.79) node [anchor=north west][inner sep=0.75pt]  [font=\small]  {$f\Omega_{1}^{v}$};
\draw (466.05,135.2) node [anchor=north west][inner sep=0.75pt]  [font=\small]  {$f\Omega_{3}^{v}$};
\draw (284.63,138.7) node [anchor=north west][inner sep=0.75pt]  [font=\small]  {$f\Omega_{2}^{v}$};
\draw (442.19,234.1) node [anchor=north west][inner sep=0.75pt]  [font=\small]  {$f\Omega_{v}$};
\draw (191.39,237.2) node [anchor=north west][inner sep=0.75pt]  [font=\small]  {$f\Omega_{3}^{m}$};
\end{tikzpicture}
\caption{Generalized flat Reidemeister moves.} \label{fig7}	
\end{center}
\end{figure}

\begin{definition} {\rm 
Two flat knotoid diagrams in $\Sigma$ are said to be \textit{flat-equivalent} if they are related to each other by a finite sequence of moves, where each move is either from $\textit{f}\mathcal{R}$ or an  isotopy of $\Sigma$. A \textit{flat knotoid in $\Sigma$} is defined to be a flat-equivalence class of flat knotoid diagrams in $\Sigma$. 
}
\end{definition}

We recall that the following property holds for flat knotoids in $S^2$, see~\cite[Prop.~3.6]{GK}, but does not hold for flat knotoids in $\mathbb R^2$. 

\begin{proposition} [\cite{GK}] \label{prop2}
Any flat knotoid in $S^2$ is flat-equivalent to the trivial knotoid.
\end{proposition}

By \textit{generalized flat Reidemeister moves} we call a collection 
$$
\textit{gf}\mathcal{R} = \{f\Omega_1, f\Omega_2, f\Omega_3, f\Omega_1^v, f\Omega_2^v, f\Omega_3^v, f\Omega_3^m, f\Omega_v\}
$$ 
of flat Reidemeister moves $f\Omega_1$, $f\Omega_2$, $f\Omega_3$, flat virtual Reidemeister moves $f\Omega_1^v$, $f\Omega_2^v$, $f\Omega_3^v$, $f\Omega_v$, the flat mixed move $f\Omega_3^m$, see Fig.~\ref{fig7}. Since arcs of flat knotoids are oriented, we consider oriented versions of generalized flat Reidemeister moves with all possible orientations.  

\begin{definition}{\rm 
A \textit{flat virtual knotoid diagram} is defined to be a flat knotoid diagram in $S^2$ with virtual crossings. Two flat virtual knotoid diagrams are said to be \textit{flat-equivalent} if they are related by a finite sequence of moves, where each move is either from $\textit{gf}\mathcal{R}$ or isotopy of $S^2$. A \textit{flat virtual knotoid} is defined as an flat-equivalence class of flat virtual knotoid diagrams.
}
\end{definition}

\begin{remark}{\rm 
The move $f\Phi_{\pm}$, see Fig.~\ref{fig103}, obtained by replacing classical crossings by flat crossings in forbidden moves $\Phi_+$, $\Phi_-$,  and virtual forbidden moves $\Phi _{\textrm{under}}$, $\Phi _{\textrm{over}}$, given by  Fig.~\ref{fig8}, are forbidden for flat virtual knotoid diagrams.
	}
\end{remark}
\begin{figure}[htbp]
\begin{center}
\tikzset{every picture/.style={line width=1.0pt}} 
\begin{tikzpicture}[x=0.75pt,y=0.75pt,yscale=-1,xscale=1]  			

\draw    (50.06,753.51) -- (50.38,839.54) ;
\draw    (12.55,798.24) -- (90.56,797.73) ;
\draw    (188.45,798.51) -- (224.82,798.02) ;
\draw    (234.48,753.75) -- (234.82,840.42) ;
\draw  [fill={rgb, 255:red, 13; green, 13; blue, 13 }  ,fill opacity=1 ] (87.25,797.73) .. controls (87.25,796.88) and (87.99,796.19) .. (88.91,796.19) .. controls (89.82,796.19) and (90.56,796.88) .. (90.56,797.73) .. controls (90.56,798.58) and (89.82,799.27) .. (88.91,799.27) .. controls (87.99,799.27) and (87.25,798.58) .. (87.25,797.73) -- cycle ;
\draw  [fill={rgb, 255:red, 13; green, 13; blue, 13 }  ,fill opacity=1 ] (223.16,798.02) .. controls (223.16,797.1) and (223.9,796.36) .. (224.82,796.36) .. controls (225.74,796.36) and (226.48,797.1) .. (226.48,798.02) .. controls (226.48,798.93) and (225.74,799.68) .. (224.82,799.68) .. controls (223.9,799.68) and (223.16,798.93) .. (223.16,798.02) -- cycle ;

\draw   (150.83,797.51) -- (156.54,800.04) -- (150.83,802.56) ;
\draw    (155.82,800.2) -- (128.72,800.2) ;
\draw   (133.83,802.74) -- (128.01,800.33) -- (133.6,797.69) ;

\draw (130.8,779.59) node [anchor=north west][inner sep=0.75pt]  [font=\small]  {$f\Phi _{\pm}$};
\end{tikzpicture}
\caption{Flat forbidden move $f\Phi_{\pm}$.} \label{fig103}	
\end{center}
\end{figure}
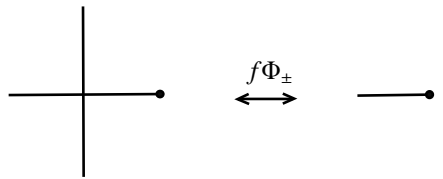

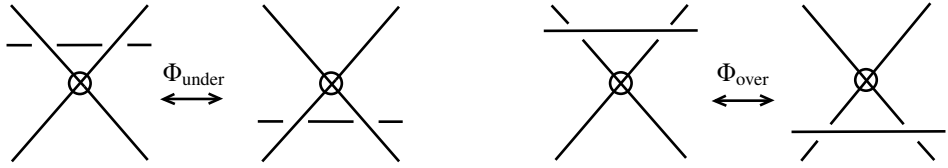
\begin{figure}[htbp]
\begin{center}
\tikzset{every picture/.style={line width=1.0pt}}  
\begin{tikzpicture}[x=0.75pt,y=0.75pt,yscale=-1,xscale=1]
\draw    (126.2,61.69) -- (158.12,98.69) ;
\draw    (158.17,20.58) -- (89.6,99.43) ;
\draw    (89.16,20.58) -- (126.2,61.69) ;
\draw    (147.83,40.64) -- (159.99,40.64) ;
\draw    (112.24,40.73) -- (135.2,40.67) ;
\draw    (86.9,40.76) -- (99.68,40.77) ;
\draw    (377.49,38) -- (429.61,97.66) ;
\draw    (361.8,20.69) -- (370.31,29.91) ;
\draw    (421.99,30.35) -- (430.45,20.41) ;
\draw    (415.57,38) -- (363.78,98.66) ;
\draw    (539.38,80.67) -- (488.47,20.66) ;
\draw    (554.83,98.76) -- (546.5,89.4) ;
\draw    (495.98,89) -- (487.72,99.1) ;
\draw    (502.48,80.67) -- (552.83,19.59) ;
\draw    (435.49,33.33) -- (357.85,33.53) ;
\draw    (560.45,84.25) -- (482.81,84.51) ;
\draw   (118.47,60) .. controls (118.47,57.01) and (120.89,54.59) .. (123.88,54.59) .. controls (126.88,54.59) and (129.3,57.01) .. (129.3,60) .. controls (129.3,63) and (126.88,65.42) .. (123.88,65.42) .. controls (120.89,65.42) and (118.47,63) .. (118.47,60) -- cycle ;
\draw   (514.95,58.31) .. controls (514.95,55.32) and (517.38,52.89) .. (520.37,52.89) .. controls (523.36,52.89) and (525.79,55.32) .. (525.79,58.31) .. controls (525.79,61.3) and (523.36,63.73) .. (520.37,63.73) .. controls (517.38,63.73) and (514.95,61.3) .. (514.95,58.31) -- cycle ;
\draw   (391.3,59.89) .. controls (391.3,56.9) and (393.73,54.47) .. (396.72,54.47) .. controls (399.71,54.47) and (402.14,56.9) .. (402.14,59.89) .. controls (402.14,62.88) and (399.71,65.31) .. (396.72,65.31) .. controls (393.73,65.31) and (391.3,62.88) .. (391.3,59.89) -- cycle ;
\draw    (252.92,58.32) -- (284.84,21.32) ;
\draw    (284.89,99.43) -- (216.32,20.58) ;
\draw    (215.88,99.43) -- (252.92,58.32) ;
\draw    (274.55,79.37) -- (286.71,79.37) ;
\draw    (238.96,79.28) -- (261.92,79.34) ;
\draw    (213.62,79.24) -- (226.4,79.24) ;
\draw   (245.19,60) .. controls (245.19,63) and (247.61,65.42) .. (250.6,65.42) .. controls (253.6,65.42) and (256.02,63) .. (256.02,60) .. controls (256.02,57.01) and (253.6,54.59) .. (250.6,54.59) .. controls (247.61,54.59) and (245.19,57.01) .. (245.19,60) -- cycle ;
\draw   (188.8,65.22) -- (194.5,67.75) -- (188.8,70.27) ;
\draw    (193.79,67.91) -- (166.69,67.91) ;
\draw   (171.8,70.45) -- (165.98,68.04) -- (171.56,65.4) ;
\draw   (466.93,66.08) -- (472.64,68.6) -- (466.93,71.13) ;
\draw    (471.92,68.76) -- (444.82,68.76) ;
\draw   (449.93,71.31) -- (444.11,68.9) -- (449.69,66.26) ;
\draw (164.03,48.43) node [anchor=north west][inner sep=0.75pt]  [font=\small]  {$\Phi _{\text{under}}$};
\draw (443.49,48.43) node [anchor=north west][inner sep=0.75pt]  [font=\small]  {$\Phi _{\text{over}}$};
\end{tikzpicture}
\caption{Virtual forbidden moves $\Phi _{\textrm{under}}$ and $\Phi _{\textrm{over}}$.\label{fig8}}
\end{center}
\end{figure} 

\begin{definition}{\rm 
We say that a virtual knotoid diagram $K$ \textit{overlies} a flat virtual knotoid diagram $\overline{K}$ if $K$ can be  obtained from $\overline{K}$ by choosing a crossing type as over or under for each flat crossing. The flat virtual diagram $\overline{K}$ is said to be the \textit{underlying flat virtual knotoid diagram} of $K$. }
\end{definition}

If a virtual knotoid diagram $K$ overlies a flat virtual knotoid diagram $\overline{K}$ then it is clear that any generalized Reidemeister move on $K$ induces a generalized flat Reidemeister move on $\overline{K}$ by changing all classical crossings to flat crossings.

\subsection{Affine index polynomial invariant of a  knotoid} 
The affine index polynomial was defined for virtual knots and links by Kauffman in~\cite{Kau}. The affine index polynomial for a knotoid, either classical or virtual was defined by  G\"{u}g\"{u}mc\"{u} and Kauffman in~\cite{GK}. 

The affine index polynomial is based on an integer labeling assigned to flat knotoid diagram as follows. A flat knotoid diagram, either classical or virtual, is associated with a graph (resp. virtual graph) where the flat (classical) crossings and the endpoints (the tail and the head) are regarded as vertices of a graph. An \textit{arc} of an oriented flat knotoid diagram is the edge of the graph it represents. Given a virtual knotoid diagram $D$, the labeling of each arc of the underling flat diagram  $\overline{D}$ of $D$, begins with the tail and the first flat crossing. At each flat crossing, the labels of the arcs change by one; if the incoming arc labeled by $b$, where $b \in \mathbb Z$, crosses the crossing towards left then the next arc is labeled by $b+1$, and if the incoming arc labeled by $a$, where $a \in \mathbb Z$, crosses the crossing towards right then the next arc is labeled by $a-1$. There is no change of labels at virtual crossings. Let $c$ be a classical crossing of $D$. Then arcs around a classical crossing $c \in D$ get the same labelings as arcs around the underling flat crossing $\overline{c} \in \overline{D}$. These arc labeling is similar to the Cheng's coloring for a virtual knot diagram, see~\cite{Che}. The integer labeling rule for $D$ and $\overline{D}$ is illustrated in~Fig.~\ref{fig11}. 
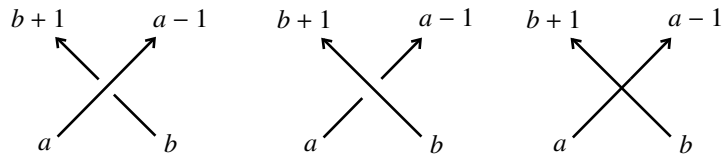
\begin{figure}[!ht]
\begin{center}
\tikzset{every picture/.style={line width=1.0pt}}  
\begin{tikzpicture}[x=0.75pt,y=0.75pt,yscale=-1,xscale=1]	
\draw    (274.04,328.33) -- (323.49,377.42) ;
\draw    (188.66,328.01) -- (139.99,377.45) ;
\draw    (139.89,327.91) -- (160.71,348.56) ;
\draw   (140.19,332.8) -- (139.57,327.94) -- (144.32,328.93) ;
\draw   (184.15,329.05) -- (188.86,327.9) -- (188.42,332.77) ;
\draw    (168.39,356.22) -- (189.34,377) ;
\draw   (274.33,333.22) -- (273.72,328.36) -- (278.47,329.34) ;
\draw   (318.3,329.47) -- (323.01,328.31) -- (322.57,333.19) ;
\draw    (322.8,328.42) -- (303.27,348.3) ; 
\draw    (293.67,358) -- (274.14,377.87) ;
\draw    (399.78,328.33) -- (449.23,377.42) ;
\draw   (400.07,333.22) -- (399.46,328.36) -- (404.21,329.34) ;
\draw   (444.04,329.47) -- (448.75,328.31) -- (448.31,333.19) ;
\draw    (448.54,328.42) -- (429.01,348.3) ;
\draw    (429.01,348.3) -- (399.88,377.87) ;
\draw (128.49,376.53) node [anchor=north west][inner sep=0.75pt]  [font=\small]  {$a$};
\draw (192.1,374.57) node [anchor=north west][inner sep=0.75pt]  [font=\small]  {$b$};
\draw (114.85,311.31) node [anchor=north west][inner sep=0.75pt]  [font=\small]  {$b+1$};
\draw (186.6,311.75) node [anchor=north west][inner sep=0.75pt]  [font=\small]  {$a-1$};
\draw (262.64,376.95) node [anchor=north west][inner sep=0.75pt]  [font=\small]  {$a$};
\draw (326.25,374.98) node [anchor=north west][inner sep=0.75pt]  [font=\small]  {$b$};
\draw (249,311.73) node [anchor=north west][inner sep=0.75pt]  [font=\small]  {$b+1$};
\draw (320.74,311.17) node [anchor=north west][inner sep=0.75pt]  [font=\small]  {$a-1$};
\draw (388.38,376.95) node [anchor=north west][inner sep=0.75pt]  [font=\small]  {$a$};
\draw (451.99,374.98) node [anchor=north west][inner sep=0.75pt]  [font=\small]  {$b$};
\draw (374.74,311.73) node [anchor=north west][inner sep=0.75pt]  [font=\small]  {$b+1$};
\draw (446.48,311.17) node [anchor=north west][inner sep=0.75pt]  [font=\small]  {$a-1$};			
\end{tikzpicture}
\caption{Labeling around classical crossings of $D$ and flat crossing in $\overline{D}$.} \label{fig11}
\end{center}
\end{figure}

If $\zeta =c$ is a classical crossing of $D$ (or $\zeta = \overline{c}$ is a flat crossing of $\overline{D}$), then we define two numbers, $W_+(\zeta)$ and $W_-(\zeta)$ at $\zeta$ resulting the labeling of $D$ (resp. $\overline{D}$) as follows.
$$
W_+ (\zeta) = a - (b+1) \qquad \textrm{and} \qquad W_-(\zeta) = - W_+(\zeta) = b - (a-1). 
$$ 
The \textit{weight} of a classical  crossing $c\in D$ is defined by 
\begin{equation}
W_D(c) =  \mathrm{sgn} (c) \, \Big( a-(b+1)\Big), 
\end{equation}
where $\mathrm{sgn} (c)$ is the sign of $c$.  Hence we can write 
$$
W_D(c) = W_{\mathrm{sgn} (c)} (c) = \mathrm{sgn} (c) W_+(c) = \mathrm{sgn} (c) W_+(\overline{c}).
$$ 
The affine index polynomial of virtual knotoid is defined as following. 

\begin{definition} [\cite{GK}] {\rm 
The \textit{affine index polynomial} of a virtual or classical knotoid diagram $D$ is defined by the following equation:  
\begin{equation} 
P_D (t) = \sum_{c \in D} \mathrm{sgn} (c) \Big(t^{W_D (c)} - 1\Big),
\end{equation} 
where the sum is taken over all classical crossings of $D$. 
}
\end{definition}

The following result was obtained in~\cite[Th.~5.1]{GK}. 

\begin{theorem} [\cite{GK}]   \label{th1}
The affine index polynomial is a virtual and classical knotoid invariant. 
\end{theorem}
 
Let $D$ be a knotoid diagram with integer labeling of arcs. For $n \in \mathbb Z$, the \textit{$n$-th writhe} of $D$ is defined as 
\begin{equation}
w_n(D) = \sum_{\substack W_{D}(c)=n} \mathrm{sgn}(c),  
\end{equation}
where the sum is taken over all classical crossings $c$ of $D$ such that $W_D(c) = n$.
Remark  that $n$-writhe of $D$ is a coefficient of $t^n$, where $n \in \mathbb Z \setminus \{0\}$: 
\begin{equation} 
P_{D}(t) = \sum_{n \in \mathbb{Z} } w_{n} (D) \left(t^n - 1 \right) = \sum_{n \in \mathbb{Z} \backslash \left\{ 0 \right\} } w_{n}(D) t^n + \left( w_{0}(D) - w(D) \right). \label{eqn:PD}
\end{equation}

In virtue of  Theorem~\ref{th1}, if $D$ is a diagram of a knotoid $K$,  the following notations are correct: $P_{K}(t)=P_ {D}(t)$. Since $w_n (K)=w_n(D)$, where $n \in \mathbb{Z} \setminus \left\{ 0 \right\}$, and  $w_{0}'(K)= w_{0}(D) - w(D)$ are coefficients of the affine index polynomial, they also are invariants of knotoid $K$.  Denoting 
$$ 
P_ {K}^{+}(t) = \sum_{n \in \mathbb{Z}_{>0} } w_{n}(K) t^n \qquad \mbox{and} \qquad P_{K}^{-} (t) = \sum_ {n \in \mathbb{Z}_{<0}} w_n (K) t^n, 
$$
we can write  $P_{K}(t)=P_{K}^{+}(t) + P_{K}^{-}(t) + w_{0}'(K).$ 

 We give the following definitions which are analogous to definitions for virtual knots given in ~\cite{Kim}.  

\begin{definition} {\rm 
Let $\overline{D}$ be a flat virtual knotoid diagram with an integer labeling of arcs. For $n\in \mathbb{Z}_{> 0}$ define  \textit{$n$-th flat writhe} $f_n (\overline{D})$ of $\overline{D}$ by  
\begin{equation} 
f_n (\overline{D}) = \sum_{|W_{+} (\overline{c})| =n} \mathrm{sign} \left( W_+(\overline{c}) \right),
\end{equation}
where the sum is taken over all flat crossings $\overline{c}$ of $\overline{D}$ such that $|W_+(\overline{c})|=n$, and $\mathrm{sign} (x)$ denotes the sign of $x \in \mathbb Z$.  
The \textit{flat affine index polynomial} $Q_{\overline{D}} ( t )\in \mathbb{Z} [ t , t^{-1}]$, is defined by 
\begin{equation}
Q _ {\overline{D}} ( t ) = \sum_{ n \in \mathbb{Z} _ {>0}} f_{n} (\overline{D}) t^n.
\end{equation}
}
\end{definition}

The following proposition gives a method to show that a value from a flat virtual knotoid diagram is a flat virtual knotoid invariant. It is analogous to a method for flat virtual knot diagrams formulated in~\cite[Prop.2.9]{Kim}. 

\begin{proposition}\label{prop1}
Let $F$ be a flat virtual knotoid, and $\overline{D}$ a diagram of $F$. Denote a value calculated from $\overline{D}$ by $I(\overline{D})$. Assume that $I(\overline{D})$ is an invariant of an overlying virtual knotoid $K$ of $\overline{D}$, and this invariant remains unchanged when the overlying virtual knotoid $K$ is replaced with another overlying virtual knotoid $K'$ of $\overline{D}$. Then $I(\overline{D})$ is an invariant of flat virtual knotoid $F$. 
\end{proposition}

\begin{proof}
The proof proceeds similarly to the Proposition 2.9 in \cite{Kim}. Let us denote the discussed invariant of $K$ by $\Phi(K)$, i.e. $I(\overline{D}) = \Phi(K)$. By the assumption, if $K'$ is another overlying virtual knotoid of $\overline{D}$, then $\Phi(K') = \Phi(K) = I(\overline{D})$. 

Let $\overline{D_1}$ be a flat virtual knot diagram equivalent to $\overline{D}$. Let $K$ be a virtual knotoid which is an overlying virtual knotoid of $\overline{D}$ and an overlying virtual knotoid of $\overline{D_1}$ both. Then $I(\overline{D_1}) = \Phi(K) = I(\overline{D})$. 
\end{proof}

By applying Proposition~\ref{prop1}, we present the following theorem. 

\begin{theorem}\label{th2}
Let $\overline{K}$ be a flat virtual knotoid, and $\overline{D}$ diagram of $\overline{K}$. Then 
\begin{itemize}
\item[(1)] For any $n \in \mathbb{Z} _ { > 0 }$, the $n$-th flat writhe $f_n (\overline{D})$ is an invariant of $\overline{K}$. 
\item[(2)] The affine index polynomial  $Q_{\overline{D}}(t)$ is an invariant of $\overline{K}$.
\end{itemize}
\end{theorem}

\begin{proof}
(1) Let $K$ be any overlying virtual knotoid of $\overline{D}$ with diagram $D$. 
Let us demonstrate that for $n \in \mathbb{Z}_{>0}$ the following relation holds 
\begin{equation}
f_n (\overline{D}) = w_n (K) - w_{- n} (K).  \label{eqn:KPV}
\end{equation} 
Indeed, using $W_D (c) = \textrm{sgn} (c) W_+(c) = \textrm{sgn} (c) W_+ (\overline{c})$ and $\mathrm{sgn (c) } = \pm 1$ we get 
\begin{eqnarray*}
f_n (\overline{ D } )
&=& \sum_{|W_+(\overline{c}) = n|} \mathrm{sign} \left( W_+(\overline{c}) \right)= \sum_{W_{+}(\overline{c})=n} \mathrm{sign} \left(W_+(\overline{c}) \right) + \sum _ {W_{+}(\overline{c})=-n} \mathrm{sign} \left(W_+(\overline{c})\right)  \\ 
 & = & \sum_{W_+(\overline{c})=n} 1 + \sum_{W_+(\overline{c})=-n} (-1)
= \sum_{\mathrm{sgn} (c) W_D(c)=n} 1 + \sum_{\mathrm{sgn}(c) W_D(c) = -n} (-1) \\ 
&=&  
\sum_{W_D(c) = n} \mathrm{sgn} (c) - \sum_{W_D (c) = -n} \mathrm{sgn} (c)  = w_n (K) - w_{-n} (K), 
\end{eqnarray*}
where $\overline{c}$ is a flat crossing of $\overline{D}$, $c$ is a corresponding classical crossing of~$D$, and $\mathrm{sign} (x)$ denotes the sign of $x \in \mathbb Z$. By Proposition~\ref{prop1}, $f_{n}(\overline{D})$ is an invariant of $\overline{K}$.

(2) Let us demonstrate that the following relation holds: 
$$
Q_{\overline{D}}(t) = P_{K}^{+}(t) - P_{K}^{ - } (t ^ {-1} ). 
$$
Indeed, by item (1), we have
\begin{eqnarray*}
Q_{\overline{D}} (t)
&=& \sum _ {n \in \mathbb{Z}_{>0}} f_n (\overline{D}) t^n  
 =  \sum_{n\in \mathbb{Z}_{>0} } \left(w_{n} (K) - w_{-n} (K) \right) t^n   
= \sum_{n \in \mathbb{Z}_{>0 }} w_n (K) t^n - \sum_{n\in \mathbb{Z}_{>0}} w_{-n} (K) t^n \\ 
&=& \sum_{n \in \mathbb{Z}_{>0}} w_n (K) t^n - \sum_{m \in \mathbb{Z}_{<0}} w_{m} (K) (t^{-1})^m    
= P_{K}^{+} (t) - P_{K}^{-} ( t ^{-1}) . 
\end{eqnarray*}
Analogously, by Proposition~\ref{prop1}, $Q_{\overline{D}} (t)$ is an invariant of $\overline{K}$.
\end{proof}

\begin{remark}{\rm 
The formula~(\ref{eqn:KPV})  was used in~\cite{KPV} for defining F-polynomials for oriented virtual knot. In~\cite{KPV} the $n$-th flat writhe $f_n (\overline{D})$ was named $n$-th dwrithe and denotes by $\nabla J_n(D) = J_n(D) - J_{-n}(D)$, where $J_n(D)$ was a notation for $n$-th writhe. 
}
\end{remark}

\begin{remark} {\rm 
For the trivial flat virtual knotoid $\overline{K_0}$, we have $Q _ {\overline{K_0}} ( t )=0$.
} 
\end{remark}

\subsection{Singular virtual knotoids}

The definition of singular virtual knotoid can be done analogously to the definition of singular virtual knot from~\cite{Hen}. A singular crossing is indicated with a black dot as in Fig.~\ref{fig4}~(d).  

\begin{definition} {\rm 
A \textit{singular virtual knotoid diagram} is a virtual knotoid diagram with finite number of \textit{singular crossings}.
}
\end{definition}

Thus, a singular virtual knotoid diagram can has classical, virtual, or singular crossings. 
We define \textit{generalized singular Reidemeister moves} as a collection $\textit{gs}\mathcal{R} = \{\textit{g}\mathcal{R}, S_2, S_3, S^m_3\}$ of  generalized Reidemeister moves together with moves $S_2$, $S_3$, $S_3^{m}$ involving singular crossings given in Fig.~\ref{fig10}. 
\begin{figure}[htbp]
\begin{center}
\tikzset{every picture/.style={line width=1.0pt}}  
\begin{tikzpicture}[x=0.75pt,y=0.75pt,yscale=-1,xscale=1]	
\draw    (34.95,193.66) -- (19.75,178.24) -- (53.78,143.8) ;
\draw    (23.43,213.4) -- (57.64,178.69) -- (23.89,143.35) ;
\draw  [fill={rgb, 255:red, 28; green, 26; blue, 26 }  ,fill opacity=1 ] (33.34,159.02) .. controls (33.34,156.03) and (35.77,153.6) .. (38.76,153.6) .. controls (41.76,153.6) and (44.18,156.03) .. (44.18,159.02) .. controls (44.18,162.01) and (41.76,164.44) .. (38.76,164.44) .. controls (35.77,164.44) and (33.34,162.01) .. (33.34,159.02) -- cycle ;
\draw    (44.35,202.44) -- (55.65,213.97) ;
\draw    (135.32,163.49) -- (150.44,178.99) -- (116.21,213.24) ;
\draw    (146.95,143.81) -- (112.55,178.34) -- (146.1,213.86) ;
\draw  [fill={rgb, 255:red, 28; green, 26; blue, 26 }  ,fill opacity=1 ] (136.73,198.14) .. controls (136.72,201.13) and (134.28,203.54) .. (131.28,203.53) .. controls (128.29,203.51) and (125.88,201.07) .. (125.9,198.08) .. controls (125.91,195.08) and (128.35,192.67) .. (131.34,192.69) .. controls (134.34,192.71) and (136.75,195.14) .. (136.73,198.14) -- cycle ;
\draw    (125.97,154.66) -- (114.73,143.06) ;
\draw    (491.43,141.08) -- (435.07,212.41) ;
\draw    (434.91,141.42) -- (491.59,212.08) ;
\draw    (429.64,152.16) -- (496.09,152.16) ;
\draw    (604.11,140.41) -- (547.76,211.74) ;
\draw    (547.21,141.28) -- (604.27,212.38) ;
\draw    (541.63,199.36) -- (609.08,199.36) ;
\draw  [fill={rgb, 255:red, 23; green, 22; blue, 22 }  ,fill opacity=1 ] (457.83,176.75) .. controls (457.83,173.76) and (460.26,171.33) .. (463.25,171.33) .. controls (466.24,171.33) and (468.67,173.76) .. (468.67,176.75) .. controls (468.67,179.74) and (466.24,182.17) .. (463.25,182.17) .. controls (460.26,182.17) and (457.83,179.74) .. (457.83,176.75) -- cycle ;
\draw   (477.42,152.16) .. controls (477.42,149.17) and (479.84,146.74) .. (482.83,146.74) .. controls (485.83,146.74) and (488.25,149.17) .. (488.25,152.16) .. controls (488.25,155.15) and (485.83,157.58) .. (482.83,157.58) .. controls (479.84,157.58) and (477.42,155.15) .. (477.42,152.16) -- cycle ;
\draw   (437.75,152.18) .. controls (437.75,149.18) and (440.17,146.76) .. (443.17,146.76) .. controls (446.16,146.76) and (448.59,149.18) .. (448.59,152.18) .. controls (448.59,155.17) and (446.16,157.59) .. (443.17,157.59) .. controls (440.17,157.59) and (437.75,155.17) .. (437.75,152.18) -- cycle ;
\draw   (588.52,199.28) .. controls (588.52,196.29) and (590.95,193.86) .. (593.94,193.86) .. controls (596.93,193.86) and (599.36,196.29) .. (599.36,199.28) .. controls (599.36,202.27) and (596.93,204.7) .. (593.94,204.7) .. controls (590.95,204.7) and (588.52,202.27) .. (588.52,199.28) -- cycle ;
\draw   (552.11,199.28) .. controls (552.11,196.29) and (554.54,193.86) .. (557.53,193.86) .. controls (560.52,193.86) and (562.95,196.29) .. (562.95,199.28) .. controls (562.95,202.27) and (560.52,204.7) .. (557.53,204.7) .. controls (554.54,204.7) and (552.11,202.27) .. (552.11,199.28) -- cycle ;
\draw  [fill={rgb, 255:red, 29; green, 27; blue, 27 }  ,fill opacity=1 ] (570.32,176.83) .. controls (570.32,173.84) and (572.75,171.41) .. (575.74,171.41) .. controls (578.73,171.41) and (581.16,173.84) .. (581.16,176.83) .. controls (581.16,179.82) and (578.73,182.25) .. (575.74,182.25) .. controls (572.75,182.25) and (570.32,179.82) .. (570.32,176.83) -- cycle ;
\draw    (222.05,156.53) -- (266.63,211.87) ;
\draw    (208.64,140.47) -- (215.92,149.03) ;
\draw    (260.12,149.43) -- (267.35,140.21) ;
\draw    (254.62,156.53) -- (210.32,212.8) ;
\draw    (365.39,196.39) -- (320.84,141.56) ;
\draw    (378.9,212.91) -- (371.62,204.37) ;
\draw    (327.41,204) -- (320.19,213.22) ;
\draw    (333.1,196.39) -- (377.15,140.59) ;
\draw    (272.11,152.94) -- (204.18,153.13) ;
\draw    (382.49,199.79) -- (314.56,199.98) ;
\draw  [fill={rgb, 255:red, 30; green, 26; blue, 26 }  ,fill opacity=1 ] (233.08,176.51) .. controls (233.08,173.52) and (235.51,171.09) .. (238.5,171.09) .. controls (241.49,171.09) and (243.92,173.52) .. (243.92,176.51) .. controls (243.92,179.5) and (241.49,181.93) .. (238.5,181.93) .. controls (235.51,181.93) and (233.08,179.5) .. (233.08,176.51) -- cycle ;
\draw  [fill={rgb, 255:red, 39; green, 37; blue, 37 }  ,fill opacity=1 ] (343.62,175.51) .. controls (343.62,172.52) and (346.04,170.09) .. (349.03,170.09) .. controls (352.03,170.09) and (354.45,172.52) .. (354.45,175.51) .. controls (354.45,178.5) and (352.03,180.93) .. (349.03,180.93) .. controls (346.04,180.93) and (343.62,178.5) .. (343.62,175.51) -- cycle ;
\draw   (93.13,177.86) -- (98.84,180.39) -- (93.13,182.92) ;
\draw    (98.12,180.55) -- (71.02,180.55) ;
\draw   (76.13,183.1) -- (70.31,180.69) -- (75.9,178.04) ;
\draw   (300.75,177.53) -- (306.46,180.06) -- (300.75,182.59) ;
\draw    (305.74,180.22) -- (278.64,180.22) ;
\draw   (283.75,182.76) -- (277.93,180.35) -- (283.51,177.71) ;
\draw   (524.64,176.86) -- (530.35,179.39) -- (524.64,181.92) ;
\draw    (529.64,179.55) -- (502.54,179.55) ;
\draw   (507.64,182.1) -- (501.82,179.69) -- (507.41,177.04) ;
\draw (283.29,160.39) node [anchor=north west][inner sep=0.75pt]  [font=\small]  {$S_{3}$};
\draw (76.16,160.06) node [anchor=north west][inner sep=0.75pt]  [font=\small]  {$S_{2}$};
\draw (508.44,158.33) node [anchor=north west][inner sep=0.75pt]  [font=\small]  {$S_{3}^{m}$};
\end{tikzpicture}
\caption{Moves $S_2$, $S_3$ and $S_3^{m}$ involving singular crossings. \label{fig10}}
\end{center}
\end{figure}
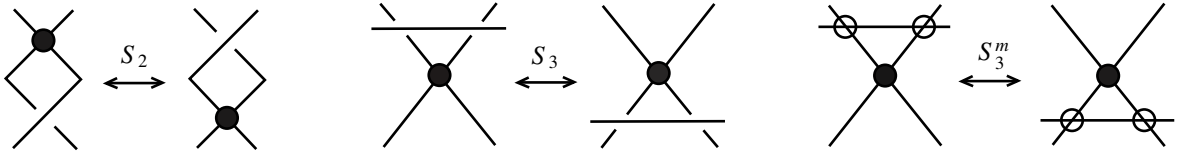

\begin{definition} {\rm
Two singular virtual knotoid diagrams are said to be \textit{equivalent} if they are related by a finite sequence of moves, where each move is either from $\textit{gs}\mathcal{R}$ or isotopy of $S^2$. 
A \textit{singular virtual knotoid} is an equivalence class of singular virtual knotoid diagrams.
}
\end{definition}

\begin{definition} {\rm 
A \textit{flat singular virtual knotoid diagram} is a singular virtual knotoid diagram with no distinguishing  between over / under type in classical crossings, with both types be represented by a flat crossing.
}
\end{definition}

Thus, a flat singular virtual knotoid diagram can has flat, virtual, or singular crossings. We define \textit{ generalized flat singular Reidemeister moves} as a collection $\textit{gfs}\mathcal{R} = \{\textit{gf}\mathcal{R}, fS_2, fS_3, fS^m_3\}$ of  generalized flat Reidemeister moves, given in Fig.~\ref{fig7}, together with flat versions $fS_2$, $fS_3$, and $fS_3^{m}$ of singular moves $S_2$, $S_3$ and $S_3^{m}$, respectively, shown in Fig.~\ref{fig10}.

\begin{definition} {\rm
Two flat singular virtual knotoid diagrams are said to be \textit{equivalent} if they are related by a finite sequence of moves, where each move is either from $\textit{gfs}\mathcal{R}$ or isotopy of $S^2$. 
A \textit{flat singular virtual knotoid} is an equivalence class of flat singular virtual knotoid diagrams.
}
\end{definition}

In the following, we recall the definition of multi-knotoids due to Turaev~\cite{Tur1}.

\begin{definition}
{\rm 
A \textit{multi-knotoid diagram} on $S^2$ is an immersion of a single oriented segment and several oriented circles in $S^2$, endowed with over/under-crossing data. A \textit{multi-knotoid} is an equivalence class of such diagrams, considered up to the same equivalence relation as for knotoids.
}
\end{definition}

\begin{definition} \label{def2.18}
{\rm 
A \textit{virtual multi-knotoid diagram} is a multi-knotoid diagram with virtual crossings. 
A \textit{virtual multi-knotoid} is an equivalence class of such diagrams, considered up to the same equivalence as for virtual knotoids. 
}
\end{definition}

We consider a flat virtual multi-knotoid diagram as a virtual multi-knotoid diagram where the under/over-crossing information of each classical crossing is forgotten. Intuitively, flat virtual multi-knotoid diagram can be thought of as a shadow (or, a flattening) of a virtual multi-knotoid diagram.

\section{Invariants of virtual knotoids} \label{section3}

In this section, we define smoothing invariant $\mathcal{F}$ for virtual knotoids, which is constructed by 0-smoothing at classical crossings.  In \cite{Pet}, Petit introduced another smoothing invariant $\mathcal{L}$, which is constructed by 1-smoothing at classical crossings. In addition, Petit defined the gluing invariant $\mathcal{G}$, which is a universal Vassiliev invariant. 

\subsection{$n$-th derivative and Vassiliev invariant} Any virtual knotoid invariant $V$ with values in an abelian group $A$ can be extended to an invariant of singular virtual knotoids. Let $V$ be a virtual knotoid invariant with values in an abelian group $A$. 
Let $K$ be a singular virtual knotoid diagram with $n \geq 1$ singular crossings. Choosing a singular crossing in $K$, we obtain $K_+$ by resolving the singular crossing into a positive crossing, and $K_-$ by resolving the singular crossing into a negative crossing, see Fig.~\ref{fig104}. Diagrams of $K_+$ and $K_-$ have $(n-1)$ singular crossings.

\begin{figure}[!ht]
\begin{center}
\tikzset{every picture/.style={line width=1.0pt}}  
\begin{tikzpicture}[x=0.75pt,y=0.75pt,yscale=-1,xscale=1]	
\draw    (490.26,442.05) -- (431.17,502.42) ;
\draw   (431.41,447.9) -- (430.66,441.97) -- (436.43,443.17) ;
\draw   (484.79,443.32) -- (490.51,441.91) -- (489.98,447.87) ;
\draw    (465.09,476.88) -- (492.09,503.71) ;
\draw    (189.81,441.61) -- (165.76,466.15) ;
\draw    (130.6,441.67) -- (189.93,501.92) ;
\draw   (131.25,447.46) -- (130.51,441.53) -- (136.27,442.73) ;
\draw   (184.34,442.88) -- (190.06,441.47) -- (189.53,447.43) ;
\draw    (155.09,477.15) -- (130.72,501.98) ;
\draw    (430.17,441.66) -- (455.76,467.54) ;
\draw    (342.82,441.65) -- (283.74,502.03) ;
\draw    (283.62,441.72) -- (342.94,501.96) ;
\draw   (284.27,447.51) -- (283.52,441.58) -- (289.29,442.77) ;
\draw   (337.35,442.93) -- (343.07,441.52) -- (342.54,447.47) ;
\draw  [fill={rgb, 255:red, 14; green, 13; blue, 13 }  ,fill opacity=1 ] (307.86,471.84) .. controls (307.86,468.85) and (310.29,466.42) .. (313.28,466.42) .. controls (316.27,466.42) and (318.7,468.85) .. (318.7,471.84) .. controls (318.7,474.83) and (316.27,477.26) .. (313.28,477.26) .. controls (310.29,477.26) and (307.86,474.83) .. (307.86,471.84) -- cycle ;
\draw    (371.11,469.66) -- (403.03,469.17) ;
\draw   (400.74,466.29) -- (403.76,469.11) -- (400.81,472) ;
\draw    (252.76,470.66) -- (220.84,470.17) ;
\draw   (223.14,467.29) -- (220.11,470.11) -- (223.06,473) ;
\draw (323.6,465.93) node [anchor=north west][inner sep=0.75pt]  [font=\large]  {$c$};
\draw (147,504.66) node [anchor=north west][inner sep=0.75pt]  [font=\normalsize]   {$K_{-}$};
\draw (447,505.66) node [anchor=north west][inner sep=0.75pt]  [font=\normalsize]   {$K_{+}$};
\draw (304,507.62) node [anchor=north west][inner sep=0.75pt]  [font=\normalsize]   {$K$};
\end{tikzpicture}
\caption{Two resolutions of a singular crossing $c$ in $K$.} \label{fig104}
\end{center}
\end{figure}
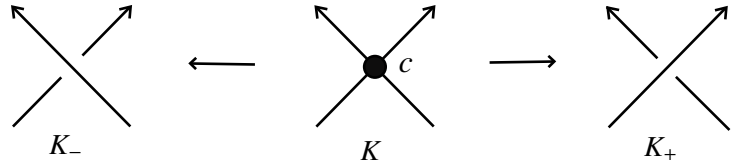

\begin{definition} {\rm 
 Define the \textit{$n$-th derivative} $V^{(n)}$, $n \geq 1$, of $V$ by the following recursive relation: 
\begin{equation}
V^{(n)}(K)=V^{(n-1)}(K_{+})-V^{(n-1)}(K_{-}) \label{(4.1)}
\end{equation}
with $V^{(0)} = V$. 
}
\end{definition}

Let us view the derivative of an invariant as a recursive definition for the invariant $V$ of a singular virtual knotoid $K$ with $n$ singular crossings. We see that the derivative gives us a value in $A$ for $V(K)$ that does not depend on the order in which we resolve the $n$ singular crossings.

\begin{definition} {\rm 
A virtual knotoid invariant $V$, is said to be a \textit{Vassiliev invariant of order $\leq n$} if its extension on singular virtual knotoids vanishes on all singular virtual knotoids with more than $n$ singular crossings. Furthermore, $V$ is said to be of order $n$ if it is of order $\leq n$ and not of order $\leq n-1$.
}
\end{definition}

We may now define the universal invariant of order one for virtual knotoids analogously the definition of universality for Vassiliev invariants of order one for virtual knots from~\cite{Hen}.

\begin{definition} [\cite{Hen}]  {\rm 
A Vassiliev invariant $\bold{U}_1$, of order one for virtual knotoids is said to be the \textit{universal invariant of order one} if every other Vassiliev invariant $V$ of order one for virtual knotoids can be recovered from $\bold{U}_1$ using only the first derivative of $V$ and the values of $V$ on one representative for each homotopy class of virtual knotoids.} 
\end{definition}

\subsection{The $0$-smoothing invariant} 

Let $D$ be a virtual knotoid diagram and $c$ a classical crossing of $D$. There are two kinds of resolution of the crossing $c$, called \textit{0-smoothing} and \textit{1-smoothing}, see Fig.~\ref{fig14}. Recall that the 0-smoothing and 1-smoothing were used in~\cite{KPV} and called \textit{smoothing against orientation} and \textit{smoothing along orientation}, respectively. This 0-smoothing  was used to construct a family of $F$-polynomials for virtual knots. Calculations of $F$-polynomials for tabulated virtual knots are presented in~\cite{IV}. 
\begin{figure}[!ht]
\begin{center}
\tikzset{every picture/.style={line width=1.0pt}} 
\begin{tikzpicture}[x=0.75pt,y=0.75pt,yscale=-1,xscale=1]
\draw    (230.55,406.08) -- (174.4,406.18) ;
\draw   (177.47,408.86) -- (174.42,406.36) -- (177.43,403.8) ;
\draw    (301.26,373.97) -- (242.17,434.35) ;
\draw   (242.41,379.83) -- (241.66,373.9) -- (247.43,375.09) ;
\draw   (295.79,375.25) -- (301.51,373.84) -- (300.98,379.79) ;
\draw    (276.09,408.8) -- (303.09,435.64) ;
\draw    (303.31,452.03) -- (279.26,476.57) ;
\draw    (244.1,452.1) -- (303.43,512.34) ;
\draw   (244.75,457.89) -- (244.01,451.96) -- (249.77,453.15) ;
\draw   (297.84,453.31) -- (303.56,451.9) -- (303.03,457.85) ;
\draw    (268.59,487.57) -- (244.22,512.41) ;
\draw    (241.17,373.58) -- (266.76,399.47) ;
\draw   (159.54,375.69) -- (165.27,374.32) -- (164.69,380.27) ;
\draw   (104.8,429.24) -- (99.01,430.36) -- (99.85,424.44) ;
\draw    (98.26,376.05) .. controls (122.21,400.48) and (138.88,404.48) .. (164.95,374.96) ;
\draw    (165.87,429.36) .. controls (142.88,404.82) and (123.55,404.15) .. (99.18,430.23) ;
\draw   (97.54,460.28) -- (96.44,454.49) -- (102.36,455.34) ;
\draw   (163.89,502.38) -- (164.89,508.19) -- (158.99,507.24) ;
\draw    (96.76,454.55) .. controls (120.71,478.98) and (137.38,482.98) .. (163.45,453.46) ;
\draw    (164.37,507.86) .. controls (141.38,483.32) and (122.05,482.65) .. (97.68,508.73) ;
   \draw   (382.13,380.9) -- (380.73,375.18) -- (386.69,375.72) ;
\draw    (382.89,442.17) .. controls (407.17,418.07) and (411.06,401.38) .. (381.37,375.5) ;
   \draw   (430.79,375.73) -- (436.59,374.69) -- (435.68,380.6) ;
\draw    (436.27,375.22) .. controls (411.88,398.37) and (411.34,417.7) .. (437.57,441.9) ;
   \draw   (384.8,456.9) -- (383.4,451.18) -- (389.35,451.72) ;
\draw    (385.55,518.17) .. controls (409.84,494.07) and (413.73,477.38) .. (384.04,451.5) ;
   \draw   (433.46,451.73) -- (439.26,450.69) -- (438.35,456.6) ;
\draw    (438.93,451.22) .. controls (414.54,474.37) and (414,493.7) .. (440.24,517.9) ;
\draw    (230.3,486.58) -- (174.15,486.68) ;
\draw    (322.39,407.87) -- (378.53,407.93) ;
\draw   (376.32,405.18) -- (379.12,407.88) -- (376.35,410.61) ;
\draw    (324.89,486.49) -- (381.03,486.17) ;
\draw   (378.74,483.29) -- (381.76,486.11) -- (378.81,489) ;
\draw   (176.62,489.24) -- (173.58,486.74) -- (176.58,484.19) ;
\draw (175.12,392.65) node [anchor=north west][inner sep=0.75pt]  [font=\tiny]  {$0-\text{smoothing}$};
\draw (320.83,392.88) node [anchor=north west][inner sep=0.75pt]  [font=\tiny]  {$1-\text{smoothing}$};
\draw (323.23,472.37) node [anchor=north west][inner sep=0.75pt]  [font=\tiny]  {$1-\text{smoothing}$};
\draw (175.12,472.94) node [anchor=north west][inner sep=0.75pt]  [font=\tiny]  {$0-\text{smoothing}$};
\draw (281.17,396.97) node [anchor=north west][inner sep=0.75pt]  [font=\large]  {$c$};
\draw (282.6,476.86) node [anchor=north west][inner sep=0.75pt]  [font=\large]  {$c$};		
\end{tikzpicture}
\caption{Two resolutions of classical crossing $c \in D$.} \label{fig14}
\end{center}
\end{figure}
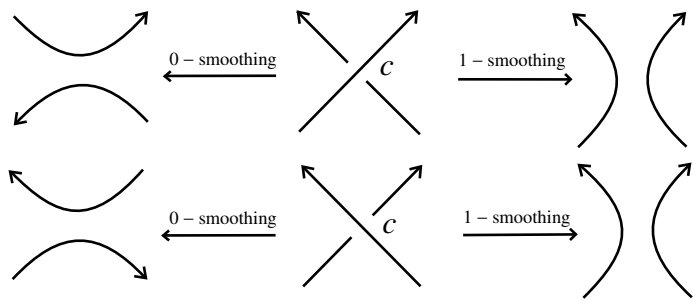

After applying 0-smoothing (1-smoothing) at $c \in D$, we will get a virtual knotoid (multi-knotoid) diagram. We use notation $D_c$ (${}_cD$) for the obtained virtual knotoid (multi-knotoid) diagram and $\overline{D_c}$ ($\overline{{}_cD}$) for the non-oriented underlying flat virtual knotoid (multi-knotoid) diagram. We will call a \textit{flattening} the transformation $D_c$ to $\overline{D_{c}}$ (${}_cD$ to $\overline{{}_cD}$). 

\begin{definition}\label{def4.2} {\rm 
Let $D$ be an oriented virtual knotoid diagram and $D_c$ a virtual knotoid diagram obtained by $0$-smoothing $D$ at a classical crossing $c \in D$. Denote by $[\overline{D}]$ and $[\overline{D_c}]$ the equivalence classes of the flat virtual diagrams obtained by flattening of $D$ and $D_c$, respectively. The \textit{0-smoothing polynomial} $\mathcal{F} (D)$ for $D$ is defined by the formula 
\begin{equation}
\mathcal{F} (D) = \sum_{c} w_D(c) \, [\overline{D_c}] - w (D) \, [\overline{D}],  \label{eqn:defF}
\end{equation}
where the sum runs over all the classical crossings of $D$ with $w_D(c) = \operatorname{sgn}(c)$ and $\displaystyle w(D) = \sum_c w_D(c)$. 
    }
\end{definition}

Remark that if $D$ doesn't have classical crossings then $\mathcal{F} (D) = 0$.

\begin{theorem}\label{th3.1}
Let $D$ be a virtual knotoid diagram. Then $\mathcal{F} (D)$ is a virtual knotoid invariant.
\end{theorem}

\begin{proof} 	
Consider the behavior of $\mathcal{F} (D)$ under moves $\Omega_1$, $\Omega_2$, $\Omega_3$ and $\Omega^m_3$.
	
(1) $\Omega_1$-move.   
Let $D$ and $E$ be virtual knotoid diagrams related by an $\Omega_1$-move, and $c_* \in D$ be the classical crossing  involved in $\Omega_1$. It is clear from Fig.~\ref{fig15} that $[\overline{D_{c_*}} ]= [\overline{E}] = [\overline{D}]$ and $w(E) = w(D) - w_D(c_*)$. Notice that if a classical crossing $c \in D$ is not involved in $\Omega_1$-move then $w_D(c)=w_E(c)$, $[\overline{D_c}] = [\overline{E_c}]$. 

\begin{figure}[!htb]
\begin{center}
\tikzset{every picture/.style={line width=1.0pt}}
\begin{tikzpicture}[x=0.75pt,y=0.75pt,yscale=-0.8,xscale=0.8]
\draw  [color={rgb, 255:red, 224; green, 20; blue, 20 }  ,draw opacity=1 ] (198.08,107.47) .. controls (198.08,81.96) and (218.85,61.28) .. (244.47,61.28) .. controls (270.09,61.28) and (290.86,81.96) .. (290.86,107.47) .. controls (290.86,132.98) and (270.09,153.67) .. (244.47,153.67) .. controls (218.85,153.67) and (198.08,132.98) .. (198.08,107.47) -- cycle ;
\draw    (234.97,115.78) .. controls (176.44,75.28) and (345.29,82.56) .. (209.8,136.77) ;
\draw    (249.01,125.4) .. controls (257.21,131.52) and (267.16,135.6) .. (279.45,136.77) ;
\draw  [color={rgb, 255:red, 224; green, 20; blue, 20 }  ,draw opacity=1 ] (41.44,107.85) .. controls (41.44,82.33) and (62.21,61.65) .. (87.83,61.65) .. controls (113.45,61.65) and (134.22,82.33) .. (134.22,107.85) .. controls (134.22,133.36) and (113.45,154.04) .. (87.83,154.04) .. controls (62.21,154.04) and (41.44,133.36) .. (41.44,107.85) -- cycle ;
\draw    (128.95,127.6) .. controls (121.05,124.69) and (110.37,125.13) .. (102.91,106.33) .. controls (96.47,88.26) and (74.46,93.5) .. (70.72,107.5) .. controls (67.91,120.52) and (56.16,125.48) .. (46.72,127.02) ;
\draw    (177.4,114.94) -- (151,115.09) ;
\draw   (157.43,112.06) -- (151.03,115.01) -- (157.38,118.07) ;
\draw  [color={rgb, 255:red, 224; green, 20; blue, 20 }  ,draw opacity=1 ] (360.41,107.14) .. controls (360.41,81.63) and (381.18,60.94) .. (406.8,60.94) .. controls (432.43,60.94) and (453.2,81.63) .. (453.2,107.14) .. controls (453.2,132.65) and (432.43,153.33) .. (406.8,153.33) .. controls (381.18,153.33) and (360.41,132.65) .. (360.41,107.14) -- cycle ;
\draw  [color={rgb, 255:red, 224; green, 20; blue, 20 }  ,draw opacity=1 ] (520.44,107.51) .. controls (520.44,82) and (541.21,61.32) .. (566.83,61.32) .. controls (592.45,61.32) and (613.22,82) .. (613.22,107.51) .. controls (613.22,133.03) and (592.45,153.71) .. (566.83,153.71) .. controls (541.21,153.71) and (520.44,133.03) .. (520.44,107.51) -- cycle ;
\draw    (607.95,127.27) .. controls (600.05,124.36) and (589.37,124.8) .. (581.91,106) .. controls (575.47,87.93) and (553.46,93.17) .. (549.72,107.16) .. controls (546.91,120.19) and (535.16,125.14) .. (525.72,126.69) ;
\draw    (473,113.93) -- (499.4,114.09) ;
\draw   (492.97,111.06) -- (499.37,114.02) -- (493.02,117.08) ;
\draw    (447.45,128.27) .. controls (440.72,125.79) and (431.97,125.74) .. (424.91,114.14) .. controls (423.68,112.13) and (422.51,109.77) .. (421.41,107) .. controls (414.97,88.93) and (392.96,94.17) .. (389.22,108.16) .. controls (386.41,121.19) and (374.66,126.14) .. (365.22,127.69) ;
\draw    (337.4,116.44) -- (311,116.59) ;
\draw   (317.43,113.56) -- (311.03,116.51) -- (317.38,119.57) ;
\draw   (332.13,113.56) -- (338.54,116.52) -- (332.19,119.58) ;
\draw   (262.58,112.01) -- (254.39,113.77) -- (257.38,105.98) ;
\draw   (410.82,100.07) -- (403.86,95.42) -- (411.58,92.16) ;
\draw (235.85,126.36) node [anchor=north west][inner sep=0.75pt]  [font=\normalsize]  {$c_{*}$};
\draw (235.75,164.07) node [anchor=north west][inner sep=0.75pt]    {$D$};
\draw (77.6,161.3) node [anchor=north west][inner sep=0.75pt]    {$\overline{D_{c_{*}}}$};
\draw (398.08,163.74) node [anchor=north west][inner sep=0.75pt]    {$E$};
\draw (562.2,160.57) node [anchor=north west][inner sep=0.75pt]    {$\overline{E}$};
\draw (312.4,89) node [anchor=north west][inner sep=0.75pt]    {$\Omega _{1}$};
\end{tikzpicture}
\caption{$0$-smoothing and flattening related to $\Omega_1$.} \label{fig15}
\end{center}
\end{figure} 
	
Therefore, 
$$
\mathcal{F} (D) = \sum_{c} w_D(c) \, [\overline{D_c}] - w(D) \, [\overline{D}] = \sum_{c \neq c_*} w_D(c) \, [\overline{D_c}] + w_D(c_*) \, [\overline{D_{c_*}}] - w(D) \, [\overline{D}]  ,
$$

and 
\begin{eqnarray*}
\mathcal{F} (E) &= & \sum_c w_E (c) \, [\overline{E_c}] - w (E) \, [\overline{E}]    
=  \sum_{c \neq c_*} w_D (c) \,  [\overline{D_c}] - \Big( w(D) - w_D(c_*) \Big) \, [\overline{E}] \\ 
& = & \sum_{c \neq c_*} w_D(c) \, [\overline{D_c}] + w_D(c_*) \, [\overline{E}] - w(D) \, [\overline{E}]   
=  \sum _ { c\neq c_* } w_D(c) \, [\overline{D_c}] + w_D(c_*) \, [\overline{D_{c_*}}] - w(D) \, [\overline{D}]   
=  \mathcal{F} (D) ,
\end{eqnarray*}

(2) $\Omega_2$-move.
	Let $D$ and $E$ be virtual knotoid diagrams related by an $\Omega_2$-move, and $c_1, c_2 \in D$ are classical crossings  involved in $\Omega_2$. It is clear, see Fig.~\ref{fig16}, that $[\overline{D_{c_{1}}}] = [\overline{D_{c_{2}}}]$ and $w_D(c_{1}) = -w_D(c_{2})$. Moreover,  $[\overline{E}] = [\overline{D}]$ and $w(E)=w(D)-w_D(c_{1})-w_D(c_{2})$. Notice that $w_D(c)=w_E(c)$ and $[\overline{D_c}] = [\overline{E_c}]$ if $c$ is a classical crossing not involved in $\Omega_2$ move. 
	\begin{figure}[!htb]
		\begin{center}
			\tikzset{every picture/.style={line width=1.0pt}}
			\begin{tikzpicture}[x=0.75pt,y=0.75pt,yscale=-0.8,xscale=0.8]	
				\draw  [color={rgb, 255:red, 224; green, 20; blue, 20 }  ,draw opacity=1 ] (140.08,258.8) .. controls (140.08,233.29) and (160.85,212.61) .. (186.47,212.61) .. controls (212.09,212.61) and (232.86,233.29) .. (232.86,258.8) .. controls (232.86,284.32) and (212.09,305) .. (186.47,305) .. controls (160.85,305) and (140.08,284.32) .. (140.08,258.8) -- cycle ;
				\draw    (504.95,290.33) .. controls (388.95,252.33) and (556.55,253.93) .. (440.55,290.33) ;
				\draw  [color={rgb, 255:red, 224; green, 20; blue, 20 }  ,draw opacity=1 ] (-1.56,258.51) .. controls (-1.56,233) and (19.21,212.32) .. (44.83,212.32) .. controls (70.45,212.32) and (91.22,233) .. (91.22,258.51) .. controls (91.22,284.03) and (70.45,304.71) .. (44.83,304.71) .. controls (19.21,304.71) and (-1.56,284.03) .. (-1.56,258.51) -- cycle ;
				\draw    (128.4,265.21) -- (102,265.05) ;
				\draw   (108.43,268.08) -- (102.03,265.12) -- (108.38,262.06) ;
				\draw  [color={rgb, 255:red, 224; green, 20; blue, 20 }  ,draw opacity=1 ] (282.01,257.51) .. controls (282.01,232) and (302.78,211.32) .. (328.41,211.32) .. controls (354.03,211.32) and (374.8,232) .. (374.8,257.51) .. controls (374.8,283.03) and (354.03,303.71) .. (328.41,303.71) .. controls (302.78,303.71) and (282.01,283.03) .. (282.01,257.51) -- cycle ;
				\draw    (385.57,265.08) -- (411.97,264.94) ;
				\draw   (405.51,261.99) -- (411.94,264.86) -- (405.63,268) ;
				\draw  [color={rgb, 255:red, 224; green, 20; blue, 20 }  ,draw opacity=1 ] (426.58,258.51) .. controls (426.58,233) and (447.35,212.32) .. (472.97,212.32) .. controls (498.59,212.32) and (519.36,233) .. (519.36,258.51) .. controls (519.36,284.03) and (498.59,304.71) .. (472.97,304.71) .. controls (447.35,304.71) and (426.58,284.03) .. (426.58,258.51) -- cycle ;
				\draw    (242.47,264.5) -- (268.87,264.36) ;
				\draw   (248.2,267.55) -- (241.8,264.6) -- (248.15,261.54) ;
				\draw  [color={rgb, 255:red, 224; green, 20; blue, 20 }  ,draw opacity=1 ] (563.58,258.51) .. controls (563.58,233) and (584.35,212.32) .. (609.97,212.32) .. controls (635.59,212.32) and (656.36,233) .. (656.36,258.51) .. controls (656.36,284.03) and (635.59,304.71) .. (609.97,304.71) .. controls (584.35,304.71) and (563.58,284.03) .. (563.58,258.51) -- cycle ;
				\draw    (433.47,233.94) .. controls (440.74,246.52) and (494.52,256.44) .. (515.35,241.61) ;
				\draw    (300.86,221.93) .. controls (362.66,248.24) and (346.36,276.79) .. (302.86,295.93) ;
				\draw    (352.86,219.93) .. controls (346.86,220.62) and (335.16,229.53) .. (332.36,233.13) ;
				\draw    (322.36,274.33) .. controls (316.76,267.93) and (312.36,256.73) .. (321.56,245.13) ;
				\draw    (352.86,295.93) .. controls (347.66,295.13) and (336.36,290.33) .. (332.76,285.13) ;
				\draw   (320.16,255.17) -- (316.52,262.72) -- (312.2,255.58) ;
				\draw   (345.71,256.1) -- (341.3,263.22) -- (337.76,255.66) ;
				\draw    (570.47,233.94) .. controls (577.74,246.52) and (631.52,256.44) .. (652.35,241.61) ;
				\draw    (648.96,283.73) .. controls (641.74,271.12) and (588,260.99) .. (567.12,275.74) ;
				\draw    (525.57,262.08) -- (551.97,261.94) ;
				\draw    (154.21,226.81) .. controls (270.29,264.56) and (102.69,263.32) .. (218.61,226.67) ;
				\draw    (225.81,283.04) .. controls (218.51,270.48) and (164.71,260.68) .. (143.91,275.55) ;
				\draw    (5.07,233.94) .. controls (12.34,246.52) and (66.12,256.44) .. (86.95,241.61) ;
				\draw    (83.56,283.73) .. controls (76.34,271.12) and (22.6,260.99) .. (1.71,275.74) ;
				\draw   (545.51,258.99) -- (551.94,261.86) -- (545.63,265) ;
				\draw (177.75,311.41) node [anchor=north west][inner sep=0.75pt]    {$\overline{D_{c_{2}}}$};
				\draw (322.79,315.41) node [anchor=north west][inner sep=0.75pt]    {$D$};
				\draw (461.82,312.09) node [anchor=north west][inner sep=0.75pt]    {$\overline{D_{c_{1}}}$};
				\draw (336.18,231.64) node [anchor=north west][inner sep=0.75pt]  [font=\small]  {$c_{1}$};
				\draw (336.68,274.64) node [anchor=north west][inner sep=0.75pt]  [font=\small]  {$c_{2}$};
				\draw (525.35,242.19) node [anchor=north west][inner sep=0.75pt]  [font=\small]  {$f\Omega _{1}$};
				\draw (102.35,242.19) node [anchor=north west][inner sep=0.75pt]  [font=\small]  {$f\Omega _{1}$};
			\end{tikzpicture}	
			\caption{$0$-smoothing and flattening related to $\Omega_2$.} \label{fig16}
		\end{center}
	\end{figure}

	Therefore,  
	\begin{eqnarray*}
		\mathcal{F} (D) &= & \sum_c w_D(c) \, [\overline{D_c}] - w(D) \, [\overline{D}]  = \sum_ {c \notin \{ c_{1},c_{2}\}} w_D(c) \, [\overline{D_c}] + w_D(c_{1}) \, [\overline{D_{c_{1}}}] + w_D(c_{2}) \, [\overline{D_{c_{2}}}] - w (D) \, [\overline{D}] \\ 
		& = & \sum_{ c \notin \{c_{1}, c_{2}\}} w_D(c) \, [\overline{D_c}] - w (D) \, [\overline{D}],
	\end{eqnarray*}
	and 
	\begin{eqnarray*}
		\mathcal{F} (E) & = &\sum_{ c \notin \{c_{1}, c_{2}\}} w_E(c) \, [\overline{E}_{c}] - w (E) \, [\overline{E}]    
		=  \sum_{c \notin \{c_{1}, c_{2}\}} w_D(c) \, [\overline{D_c}] - \Big(w(D)-w_D(c_{1})-w_D(c_{2}) \Big) \, [\overline{D}] \\ 
		& = & \sum_{c \notin \{ c_{1}, c_{2}\}} w_D(c) \, [\overline{D_c}] - w(D) \, [\overline{D}]  
		=  \mathcal{F} ( D ),
	\end{eqnarray*}
	
	(3) $\Omega_3$-move. 
	Let $D$ and $E$ be virtual knotoid diagrams related by an $\Omega_3$-move, and $c_1, c_2, c_3 \in D$ are classical crossings involved in $\Omega_3$. Denote by  $e_1, e_2, e_3$ the corresponding crossings in $E$. It is clear, see Fig.~\ref{fig17}, that for $i=1,2,3$ the equalities $[\overline{D_{c_i}}] = [\overline{E_{e_i}}]$ and $w_D (c_i) = w_{E} (e_i)$ hold. Moreover, $[\overline{E}] = [\overline{D}]$ and $w(E) = w(D)$. Notice that $w_D(c)=w_E(c)$ and $[\overline{E_c}] = [\overline{D_c}]$ if $c$ is a classical crossing not involves in the $\Omega_3$-move.   
	\begin{figure}[!htb]
		\begin{center}
			\tikzset{every picture/.style={line width=1.0pt}} 
			\begin{tikzpicture}[x=0.75pt,y=0.75pt,yscale=-0.8,xscale=0.8]
				\draw  [color={rgb, 255:red, 224; green, 20; blue, 20 }  ,draw opacity=1 ] (179.08,413.14) .. controls (179.08,387.63) and (199.85,366.94) .. (225.47,366.94) .. controls (251.09,366.94) and (271.86,387.63) .. (271.86,413.14) .. controls (271.86,438.65) and (251.09,459.33) .. (225.47,459.33) .. controls (199.85,459.33) and (179.08,438.65) .. (179.08,413.14) -- cycle ;
				\draw  [color={rgb, 255:red, 224; green, 20; blue, 20 }  ,draw opacity=1 ] (37.44,412.85) .. controls (37.44,387.33) and (58.21,366.65) .. (83.83,366.65) .. controls (109.45,366.65) and (130.22,387.33) .. (130.22,412.85) .. controls (130.22,438.36) and (109.45,459.04) .. (83.83,459.04) .. controls (58.21,459.04) and (37.44,438.36) .. (37.44,412.85) -- cycle ;
				\draw   (52.92,433.95) -- (45.6,429.89) -- (53.02,426) ;
				\draw    (142.41,570.92) -- (168.81,571.06) ;
				\draw   (162.38,568.03) -- (168.78,570.98) -- (162.43,574.05) ;
				\draw  [color={rgb, 255:red, 224; green, 20; blue, 20 }  ,draw opacity=1 ] (321.01,412.85) .. controls (321.01,387.33) and (341.78,366.65) .. (367.41,366.65) .. controls (393.03,366.65) and (413.8,387.33) .. (413.8,412.85) .. controls (413.8,438.36) and (393.03,459.04) .. (367.41,459.04) .. controls (341.78,459.04) and (321.01,438.36) .. (321.01,412.85) -- cycle ;
				\draw    (394.56,450.22) .. controls (392.44,435.85) and (346.12,406.79) .. (321.29,412.88) ;
				\draw  [color={rgb, 255:red, 224; green, 20; blue, 20 }  ,draw opacity=1 ] (465.58,412.85) .. controls (465.58,387.33) and (486.35,366.65) .. (511.97,366.65) .. controls (537.59,366.65) and (558.36,387.33) .. (558.36,412.85) .. controls (558.36,438.36) and (537.59,459.04) .. (511.97,459.04) .. controls (486.35,459.04) and (465.58,438.36) .. (465.58,412.85) -- cycle ;
				\draw    (556.95,401.89) .. controls (542.88,398.25) and (498.09,429.62) .. (494.05,454.87) ;
				\draw    (52.14,445.9) -- (109.38,374.96) ;
				\draw    (60.58,374.16) -- (80.9,400.84) ;
				\draw    (88.5,410.86) -- (115.38,446.16) ;
				\draw    (41.11,429.76) -- (58.74,429.92) ;
				\draw    (70.46,429.92) -- (96.28,430.1) ;
				\draw    (109.74,430.1) -- (126.09,430.25) ;
				\draw   (112.93,436.05) -- (114.01,444.35) -- (106.47,440.69) ;
				\draw   (99.15,381.6) -- (106.82,378.24) -- (105.4,386.5) ;
				\draw    (182.91,429.76) -- (267.89,430.25) ;
				\draw    (248.56,372.71) .. controls (236.36,380.6) and (229.15,434.81) .. (245.01,454.87) ;
				\draw    (202.77,453.6) .. controls (215,445.75) and (222.42,391.58) .. (206.65,371.45) ;
				\draw    (342.3,451.01) -- (399.55,380.08) ;
				\draw    (344.18,373.51) .. controls (344.79,388.03) and (387.85,421.75) .. (413.18,418.26) ;
				\draw    (479.25,381.27) -- (534.05,453.27) ;
				\draw    (470.15,430.49) .. controls (484.6,432.01) and (524.26,394.37) .. (524.53,368.8) ;
				\draw  [color={rgb, 255:red, 224; green, 20; blue, 20 }  ,draw opacity=1 ] (180.41,564.78) .. controls (180.41,539.27) and (201.18,518.58) .. (226.8,518.58) .. controls (252.43,518.58) and (273.2,539.27) .. (273.2,564.78) .. controls (273.2,590.29) and (252.43,610.98) .. (226.8,610.98) .. controls (201.18,610.98) and (180.41,590.29) .. (180.41,564.78) -- cycle ;
				\draw  [color={rgb, 255:red, 224; green, 20; blue, 20 }  ,draw opacity=1 ] (38.77,564.49) .. controls (38.77,538.98) and (59.54,518.29) .. (85.16,518.29) .. controls (110.78,518.29) and (131.55,538.98) .. (131.55,564.49) .. controls (131.55,590) and (110.78,610.68) .. (85.16,610.68) .. controls (59.54,610.68) and (38.77,590) .. (38.77,564.49) -- cycle ;
				\draw   (54.75,550.84) -- (47.43,546.78) -- (54.86,542.89) ;
				\draw    (140.92,419.28) -- (167.32,419.42) ;
				\draw   (160.88,416.39) -- (167.29,419.34) -- (160.94,422.41) ;
				\draw  [color={rgb, 255:red, 224; green, 20; blue, 20 }  ,draw opacity=1 ] (322.35,564.49) .. controls (322.35,538.98) and (343.12,518.29) .. (368.74,518.29) .. controls (394.36,518.29) and (415.13,538.98) .. (415.13,564.49) .. controls (415.13,590) and (394.36,610.68) .. (368.74,610.68) .. controls (343.12,610.68) and (322.35,590) .. (322.35,564.49) -- cycle ;
				\draw    (395.89,601.87) .. controls (393.78,587.49) and (347.45,558.43) .. (322.62,564.52) ;
				\draw  [color={rgb, 255:red, 224; green, 20; blue, 20 }  ,draw opacity=1 ] (466.91,564.49) .. controls (466.91,538.98) and (487.68,518.29) .. (513.31,518.29) .. controls (538.93,518.29) and (559.7,538.98) .. (559.7,564.49) .. controls (559.7,590) and (538.93,610.68) .. (513.31,610.68) .. controls (487.68,610.68) and (466.91,590) .. (466.91,564.49) -- cycle ;
				\draw    (558.28,553.54) .. controls (544.21,549.89) and (499.42,581.26) .. (495.38,606.51) ;
				\draw    (59.97,603.02) -- (117.22,532.08) ;
				\draw    (54.92,530.81) -- (80.45,564.18) ;
				\draw    (90.8,577.76) -- (110.07,603) ;
				\draw    (42.94,546.65) -- (60.57,546.81) ;
				\draw    (72.3,546.81) -- (98.11,546.99) ;
				\draw    (111.57,546.99) -- (127.92,547.14) ;
				\draw   (104.93,589.25) -- (106.01,597.55) -- (98.47,593.89) ;
				\draw   (107.37,538.13) -- (115.04,534.78) -- (113.63,543.03) ;
				\draw    (184.24,546.9) -- (269.22,547.39) ;
				\draw    (249.89,524.35) .. controls (237.69,532.24) and (230.48,586.45) .. (246.34,606.51) ;
				\draw    (204.1,605.24) .. controls (216.33,597.4) and (223.76,543.22) .. (207.98,523.09) ;
				\draw    (343.64,602.66) -- (400.88,531.72) ;
				\draw    (345.52,525.16) .. controls (346.13,539.67) and (389.18,573.39) .. (414.51,569.9) ;
				\draw    (480.58,532.92) -- (535.38,604.92) ;
				\draw    (471.48,582.13) .. controls (485.94,583.66) and (525.6,546.01) .. (525.86,520.44) ;
				\draw (216.75,466.43) node [anchor=north west][inner sep=0.75pt]    {$\overline{D_{c_{1}}}$};
				\draw (73.6,471.3) node [anchor=north west][inner sep=0.75pt]    {$D$};
				\draw (361.79,469.43) node [anchor=north west][inner sep=0.75pt]    {$\overline{D_{c_{2}{}}}$};
				\draw (500.82,466.43) node [anchor=north west][inner sep=0.75pt]    {$\overline{D_{c_{3}}}$};
				\draw (78.2,383.54) node [anchor=north west][inner sep=0.75pt]  [font=\small]  {$c_{1}$};
				\draw (101.5,414.54) node [anchor=north west][inner sep=0.75pt]  [font=\small]  {$c_{2}$};
				\draw (53.1,414.54) node [anchor=north west][inner sep=0.75pt]  [font=\small]  {$c_{3}$};
				\draw (218.08,618.07) node [anchor=north west][inner sep=0.75pt]    {$\overline{E_{e_{1}}}$};
				\draw (74.93,624.07) node [anchor=north west][inner sep=0.75pt]    {$E$};
				\draw (363.13,618.07) node [anchor=north west][inner sep=0.75pt]    {$\overline{E_{e_{2}{}}}$};
				\draw (502.16,618.07) node [anchor=north west][inner sep=0.75pt]    {$\overline{E_{e_{3}}}$};
				\draw (78.58,579.57) node [anchor=north west][inner sep=0.75pt]  [font=\small]  {$e_{1}$};
				\draw (55,548.6) node [anchor=north west][inner sep=0.75pt]  [font=\small]  {$e_{2}$};
				\draw (101.57,548.6) node [anchor=north west][inner sep=0.75pt]  [font=\small]  {$e_{3}$};			
			\end{tikzpicture}
			\caption{$0$-smoothing and flattening related to $\Omega_3$.} \label{fig17}
		\end{center}
	\end{figure} 
	
	Therefore, 
	\begin{eqnarray*}
		\mathcal{F} (D) & = & \sum_c w_D(c) \, [\overline{D_c}] - w(D) \, [\overline{D}] \\
		& = & \sum_{ c \notin \{c_1, c_2, c_3 \}} w_D (c) \, [\overline{D_c}] + w_D(c_1) \, [\overline{D_{c_1}}] + w_D (c_2) \, [\overline{D_{c_2}}] + w_D(c_3) \, [\overline{D_{c_3}}] - w (D) \, [\overline{D}] ,
	\end{eqnarray*}
	
	and
	\begin{eqnarray*} 
		\mathcal{F} (E) &=&  \sum_c w_E (c) \, [\overline{E_c}] - w(E) \, [\overline{E}] \\ 
		&=& \sum_{c \notin \{e_1, e_2, e_3 \}} w_E(c) \, [\overline{E_c}] + w_E(e_1) \, [\overline{E_{e_1}}] + w_E(e_2) \, [\overline{E_{e_2}}]  + w_E(e_3) \, [\overline{E_{e_3}}] - w (E) \, [\overline{E}] \\ 
		&=& \sum_{c \notin \{c_1, c_2, c_3\}} w_D (c) \, [\overline{D_c}] + w_D(c_1) \, [\overline{D_{c_1}}] + w_D(c_2) \, [\overline{D_{c_2}}] + w_D(c_3) \, [\overline{D_{c_3}}] - w (D) \, [\overline{D}]   
		= \mathcal{F} (D),
	\end{eqnarray*}

	(4) $\Omega^{m}_3$-move.  
	Let $D$ and $E$ be virtual knotoid diagrams related by an $\Omega^{m}_3$-move, and $c_* \in D$  a classical crossing involved in  $\Omega^{m}_3$. Denote by $e_*$ the corresponding classical crossing in $E$.  It is clear, see Fig.~\ref{fig18}, that $[\overline{E_{e_*}}] = [\overline{D_{c_*}}]$ and  $w_E(e_*)=w_D(c_*)$. Moreover, $[\overline{E}] = [\overline{D}]$ and $w(E)=w(D)$. Notice that $w_D(c)=w_E(c)$ and $[\overline{E_c}] = [\overline{D_c}]$ if $c$ is a classical crossing not involved in $\Omega^{m}_3$. Therefore,  $\mathcal{F} (E) = \mathcal{F} (D)$.
	
	\begin{figure}[!htb]
		\begin{center}
			\tikzset{every picture/.style={line width=1.0pt}} 
			\begin{tikzpicture}[x=0.75pt,y=0.75pt,yscale=-1,xscale=1]
				\draw  [color={rgb, 255:red, 224; green, 20; blue, 20 }  ,draw opacity=1 ] (60.17,1339.57) .. controls (60.17,1317.56) and (78,1299.72) .. (100,1299.72) .. controls (122,1299.72) and (139.83,1317.56) .. (139.83,1339.57) .. controls (139.83,1361.58) and (122,1379.42) .. (100,1379.42) .. controls (78,1379.42) and (60.17,1361.58) .. (60.17,1339.57) -- cycle ;
				\draw   (131.55,1346.02) -- (137.88,1349.43) -- (131.56,1352.87) ;
				\draw    (72.79,1368.08) -- (121.94,1306.89) ;
				\draw    (80.04,1306.2) -- (97.48,1329.21) ;
				\draw    (104.01,1337.85) -- (127.09,1368.31) ;
				\draw   (124.98,1359.58) -- (125.91,1366.74) -- (119.44,1363.59) ;
				\draw   (113.15,1312.61) -- (119.74,1309.72) -- (118.52,1316.84) ;
				\draw  [color={rgb, 255:red, 224; green, 20; blue, 20 }  ,draw opacity=1 ] (520.67,1469.16) .. controls (520.67,1447.12) and (538.56,1429.26) .. (560.63,1429.26) .. controls (582.7,1429.26) and (600.59,1447.12) .. (600.59,1469.16) .. controls (600.59,1491.19) and (582.7,1509.06) .. (560.63,1509.06) .. controls (538.56,1509.06) and (520.67,1491.19) .. (520.67,1469.16) -- cycle ;
				\draw  [color={rgb, 255:red, 224; green, 20; blue, 20 }  ,draw opacity=1 ] (60.17,1469.41) .. controls (60.17,1447.37) and (78.06,1429.51) .. (100.13,1429.51) .. controls (122.2,1429.51) and (140.09,1447.37) .. (140.09,1469.41) .. controls (140.09,1491.44) and (122.2,1509.3) .. (100.13,1509.3) .. controls (78.06,1509.3) and (60.17,1491.44) .. (60.17,1469.41) -- cycle ;
				\draw   (130.63,1452.79) -- (137.01,1456.16) -- (130.69,1459.65) ;
				\draw    (474.92,1341.21) -- (501.32,1341.35) ;
				\draw   (494.88,1338.33) -- (501.29,1341.27) -- (494.94,1344.34) ;
				\draw    (78.43,1502.69) -- (127.74,1441.42) ;
				\draw    (74.07,1440.32) -- (96.07,1469.14) ;
				\draw    (104.98,1480.87) -- (121.58,1502.67) ;
				\draw   (120.15,1494.79) -- (121.09,1501.96) -- (114.59,1498.8) ;
				\draw   (121.26,1444.14) -- (127.86,1441.25) -- (126.65,1448.38) ;
				\draw    (61.68,1349.12) -- (137.93,1349.47) ;
				\draw    (62.92,1455.56) -- (137.53,1456.21) ;
				\draw  [color={rgb, 255:red, 224; green, 20; blue, 20 }  ,draw opacity=1 ] (209.69,1339.24) .. controls (209.69,1317.19) and (227.52,1299.32) .. (249.51,1299.32) .. controls (271.5,1299.32) and (289.33,1317.19) .. (289.33,1339.24) .. controls (289.33,1361.28) and (271.5,1379.16) .. (249.51,1379.16) .. controls (227.52,1379.16) and (209.69,1361.28) .. (209.69,1339.24) -- cycle ;
				\draw  [color={rgb, 255:red, 224; green, 20; blue, 20 }  ,draw opacity=1 ] (370.44,1339.42) .. controls (370.44,1317.42) and (388.29,1299.58) .. (410.31,1299.58) .. controls (432.33,1299.58) and (450.18,1317.42) .. (450.18,1339.42) .. controls (450.18,1361.42) and (432.33,1379.25) .. (410.31,1379.25) .. controls (388.29,1379.25) and (370.44,1361.42) .. (370.44,1339.42) -- cycle ;
				\draw   (442.29,1345.73) -- (448.63,1349.14) -- (442.3,1352.58) ;
				\draw    (383.07,1367.92) -- (432.27,1306.75) ;
				\draw    (390.33,1306.06) -- (407.79,1329.06) ;
				\draw    (414.32,1337.7) -- (437.43,1368.15) ;
				\draw   (391.46,1313.1) -- (390.34,1305.97) -- (396.9,1308.95) ;
				\draw   (425.64,1309.64) -- (432.23,1306.75) -- (431.02,1313.87) ;
				\draw  [color={rgb, 255:red, 224; green, 20; blue, 20 }  ,draw opacity=1 ] (210.43,1469.13) .. controls (210.43,1447.09) and (228.32,1429.23) .. (250.39,1429.23) .. controls (272.46,1429.23) and (290.35,1447.09) .. (290.35,1469.13) .. controls (290.35,1491.16) and (272.46,1509.03) .. (250.39,1509.03) .. controls (228.32,1509.03) and (210.43,1491.16) .. (210.43,1469.13) -- cycle ;
				\draw  [color={rgb, 255:red, 224; green, 20; blue, 20 }  ,draw opacity=1 ] (369.92,1469.71) .. controls (369.92,1447.68) and (387.81,1429.81) .. (409.88,1429.81) .. controls (431.95,1429.81) and (449.84,1447.68) .. (449.84,1469.71) .. controls (449.84,1491.74) and (431.95,1509.61) .. (409.88,1509.61) .. controls (387.81,1509.61) and (369.92,1491.74) .. (369.92,1469.71) -- cycle ;
				\draw   (440.54,1452.94) -- (446.75,1456.6) -- (440.29,1459.8) ;
				\draw    (159.92,1339.21) -- (186.32,1339.35) ;
				\draw   (179.88,1336.33) -- (186.29,1339.27) -- (179.94,1342.34) ;
				\draw    (388.19,1502.99) -- (437.49,1441.72) ;
				\draw    (383.83,1440.62) -- (405.82,1469.45) ;
				\draw    (414.74,1481.17) -- (431.34,1502.97) ;
				\draw   (384.82,1447.93) -- (384.08,1440.74) -- (390.49,1444.07) ;
				\draw   (430.76,1444.82) -- (437.37,1441.92) -- (436.15,1449.05) ;
				\draw   (392.81,1348.92) .. controls (392.81,1345.92) and (395.24,1343.5) .. (398.23,1343.5) .. controls (401.22,1343.5) and (403.65,1345.92) .. (403.65,1348.92) .. controls (403.65,1351.91) and (401.22,1354.33) .. (398.23,1354.33) .. controls (395.24,1354.33) and (392.81,1351.91) .. (392.81,1348.92) -- cycle ;
				\draw   (417.51,1349.04) .. controls (417.51,1346.05) and (419.94,1343.62) .. (422.93,1343.62) .. controls (425.92,1343.62) and (428.35,1346.05) .. (428.35,1349.04) .. controls (428.35,1352.03) and (425.92,1354.46) .. (422.93,1354.46) .. controls (419.94,1354.46) and (417.51,1352.03) .. (417.51,1349.04) -- cycle ;
				\draw    (371.95,1348.97) -- (448.27,1349.31) ;
				\draw    (372.8,1455.86) -- (447.41,1456.51) ;
				\draw  [color={rgb, 255:red, 224; green, 20; blue, 20 }  ,draw opacity=1 ] (520.16,1338.65) .. controls (520.16,1316.67) and (538.03,1298.84) .. (560.08,1298.84) .. controls (582.14,1298.84) and (600.01,1316.67) .. (600.01,1338.65) .. controls (600.01,1360.64) and (582.14,1378.47) .. (560.08,1378.47) .. controls (538.03,1378.47) and (520.16,1360.64) .. (520.16,1338.65) -- cycle ;
				\draw    (523.45,1352.98) -- (596.59,1353.4) ;
				\draw    (582.15,1372.25) .. controls (584.73,1352.43) and (551.17,1317.95) .. (538.26,1371.39) ;
				\draw    (524.66,1320.4) .. controls (531.84,1330.65) and (578.69,1335.12) .. (595.44,1320.82) ;
				\draw    (286.64,1353.97) -- (213.7,1354.83) ;
				\draw    (232.42,1375.01) .. controls (242.53,1367.64) and (246.26,1320.54) .. (231.77,1303.95) ; 
				\draw   (245.93,1354.33) .. controls (245.98,1357.32) and (243.6,1359.79) .. (240.6,1359.84) .. controls (237.61,1359.89) and (235.14,1357.51) .. (235.09,1354.52) .. controls (235.04,1351.52) and (237.42,1349.06) .. (240.42,1349) .. controls (243.41,1348.95) and (245.88,1351.33) .. (245.93,1354.33) -- cycle ;
				\draw   (258.2,1359.53) .. controls (255.21,1359.57) and (252.74,1357.18) .. (252.7,1354.19) .. controls (252.65,1351.2) and (255.04,1348.74) .. (258.03,1348.69) .. controls (261.03,1348.65) and (263.49,1351.03) .. (263.53,1354.03) .. controls (263.58,1357.02) and (261.19,1359.48) .. (258.2,1359.53) -- cycle ;
				\draw    (265.84,1374.74) .. controls (255.73,1367.37) and (252,1320.27) .. (266.49,1303.68) ;
				\draw    (523.72,1454.87) -- (596.92,1454.45) ;
				\draw    (582.47,1435.56) .. controls (585.05,1455.42) and (551.46,1489.97) .. (538.54,1436.42) ;
				\draw    (524.93,1487.52) .. controls (532.12,1477.24) and (579.01,1472.77) .. (595.77,1487.1) ;
				\draw    (287.26,1453.44) -- (214.07,1454.3) ;
				\draw    (233.59,1505.07) .. controls (243.73,1497.71) and (247.48,1450.63) .. (232.93,1434.05) ;
				\draw    (267.12,1504.79) .. controls (256.98,1497.43) and (253.24,1450.36) .. (267.78,1433.77) ;
				\draw   (88.02,1354.69) .. controls (85.03,1354.74) and (82.56,1352.35) .. (82.52,1349.36) .. controls (82.47,1346.37) and (84.86,1343.9) .. (87.85,1343.86) .. controls (90.85,1343.81) and (93.31,1346.2) .. (93.35,1349.19) .. controls (93.4,1352.19) and (91.01,1354.65) .. (88.02,1354.69) -- cycle ;
				\draw   (112.95,1354.69) .. controls (109.96,1354.74) and (107.5,1352.35) .. (107.45,1349.36) .. controls (107.41,1346.37) and (109.79,1343.9) .. (112.79,1343.86) .. controls (115.78,1343.81) and (118.24,1346.2) .. (118.29,1349.19) .. controls (118.33,1352.19) and (115.94,1354.65) .. (112.95,1354.69) -- cycle ;
				\draw   (545.27,1358.36) .. controls (542.28,1358.41) and (539.82,1356.02) .. (539.77,1353.02) .. controls (539.73,1350.03) and (542.12,1347.57) .. (545.11,1347.52) .. controls (548.1,1347.48) and (550.56,1349.87) .. (550.61,1352.86) .. controls (550.65,1355.85) and (548.26,1358.31) .. (545.27,1358.36) -- cycle ;
				\draw   (576.87,1358.63) .. controls (573.88,1358.67) and (571.42,1356.28) .. (571.37,1353.29) .. controls (571.33,1350.3) and (573.72,1347.84) .. (576.71,1347.79) .. controls (579.7,1347.75) and (582.16,1350.13) .. (582.21,1353.13) .. controls (582.25,1356.12) and (579.86,1358.58) .. (576.87,1358.63) -- cycle ;
				\draw    (158.59,1475.88) -- (184.99,1476.02) ;
				\draw   (178.55,1472.99) -- (184.96,1475.94) -- (178.61,1479.01) ;
				\draw    (476.52,1471.61) -- (502.92,1471.75) ;
				\draw   (496.48,1468.73) -- (502.89,1471.67) -- (496.54,1474.74) ;
				\draw   (545.68,1460.15) .. controls (542.69,1460.2) and (540.23,1457.81) .. (540.18,1454.82) .. controls (540.14,1451.82) and (542.53,1449.36) .. (545.52,1449.32) .. controls (548.51,1449.27) and (550.97,1451.66) .. (551.02,1454.65) .. controls (551.07,1457.64) and (548.68,1460.11) .. (545.68,1460.15) -- cycle ;
				\draw   (576.72,1459.96) .. controls (573.73,1460.01) and (571.27,1457.62) .. (571.22,1454.63) .. controls (571.18,1451.64) and (573.57,1449.17) .. (576.56,1449.13) .. controls (579.55,1449.08) and (582.01,1451.47) .. (582.06,1454.46) .. controls (582.11,1457.46) and (579.72,1459.92) .. (576.72,1459.96) -- cycle ;
				\draw   (85.64,1461.09) .. controls (82.65,1461.13) and (80.19,1458.75) .. (80.14,1455.75) .. controls (80.1,1452.76) and (82.49,1450.3) .. (85.48,1450.25) .. controls (88.47,1450.21) and (90.93,1452.6) .. (90.98,1455.59) .. controls (91.02,1458.58) and (88.64,1461.04) .. (85.64,1461.09) -- cycle ;
				\draw   (116.02,1461.46) .. controls (113.03,1461.51) and (110.56,1459.12) .. (110.52,1456.13) .. controls (110.47,1453.14) and (112.86,1450.67) .. (115.85,1450.63) .. controls (118.85,1450.58) and (121.31,1452.97) .. (121.35,1455.96) .. controls (121.4,1458.96) and (119.01,1461.42) .. (116.02,1461.46) -- cycle ;
				\draw   (259.59,1459.09) .. controls (256.6,1459.13) and (254.13,1456.75) .. (254.09,1453.75) .. controls (254.04,1450.76) and (256.43,1448.3) .. (259.42,1448.25) .. controls (262.42,1448.21) and (264.88,1450.6) .. (264.92,1453.59) .. controls (264.97,1456.58) and (262.58,1459.04) .. (259.59,1459.09) -- cycle ;
				\draw   (240.96,1459.34) .. controls (237.97,1459.38) and (235.51,1457) .. (235.46,1454) .. controls (235.42,1451.01) and (237.81,1448.55) .. (240.8,1448.5) .. controls (243.79,1448.46) and (246.25,1450.85) .. (246.3,1453.84) .. controls (246.34,1456.83) and (243.95,1459.29) .. (240.96,1459.34) -- cycle ;
				\draw   (425.55,1461.59) .. controls (422.56,1461.63) and (420.1,1459.25) .. (420.05,1456.25) .. controls (420.01,1453.26) and (422.4,1450.8) .. (425.39,1450.75) .. controls (428.38,1450.71) and (430.84,1453.1) .. (430.89,1456.09) .. controls (430.93,1459.08) and (428.55,1461.54) .. (425.55,1461.59) -- cycle ;
				\draw   (395.3,1461.46) .. controls (392.31,1461.51) and (389.85,1459.12) .. (389.8,1456.13) .. controls (389.76,1453.14) and (392.15,1450.67) .. (395.14,1450.63) .. controls (398.13,1450.58) and (400.59,1452.97) .. (400.64,1455.96) .. controls (400.68,1458.96) and (398.3,1461.42) .. (395.3,1461.46) -- cycle ;
				\draw (89.15,1383.96) node [anchor=north west][inner sep=0.75pt]    {$D$};
				\draw (103.91,1328.07) node [anchor=north west][inner sep=0.75pt]  [font=\small]  {$c_{*}$};
				\draw (105.92,1468.3) node [anchor=north west][inner sep=0.75pt]  [font=\small]  {$e_{*}$};
				\draw (240.54,1381.1) node [anchor=north west][inner sep=0.75pt]    {$\overline{D_{c_{*}}}$};
				\draw (400.17,1383.96) node [anchor=north west][inner sep=0.75pt]    {$D$};
				\draw (414.56,1328.07) node [anchor=north west][inner sep=0.75pt]  [font=\small]  {$c_{*}$};
				\draw (415.68,1468.3) node [anchor=north west][inner sep=0.75pt]  [font=\small]  {$e_{*}$};
				\draw (141,1551.9) node [anchor=north west][inner sep=0.75pt]    {case (a)};
				\draw (475,1553.67) node [anchor=north west][inner sep=0.75pt]    {case (b)};
				\draw (548.04,1381.1) node [anchor=north west][inner sep=0.75pt]    {$\overline{D_{c_{*}}}$};
				\draw (89.82,1519.79) node [anchor=north west][inner sep=0.75pt]    {$E$};
				\draw (239.7,1516.43) node [anchor=north west][inner sep=0.75pt]    {$\overline{E_{e_{*}}{}}$};
				\draw (399.33,1519.79) node [anchor=north west][inner sep=0.75pt]    {$E$};
				\draw (544.2,1517.43) node [anchor=north west][inner sep=0.75pt]    {$\overline{E_{e_{*}}}$};		
			\end{tikzpicture}
			\caption{$0$-smoothing and flattening related to $\Omega_3^{m}$.} \label{fig18}
		\end{center}
	\end{figure} 	
Thus, $\mathcal{F} (D)$ is a virtual knotoid invariant. 
\end{proof}

Therefore, if $K$ is a virtual knotoid and $D$ is its diagram, then the notation $\mathcal{F} (K) = \mathcal{F} (D)$ is correct. 

In~\cite{Pet}, Petit provided the 1-smoothing invariant for long virtual knots. Since virtual knotoids allow $\Omega_{v}$-move, any virtual knotoid is a long virtual knot. Therefore, the invariants defined for long virtual knots are naturally invariants for virtual knotoids.

\begin{definition} [\cite{Pet}] \label{def3.6}
{\rm 
Let $D$ be an oriented virtual knotoid diagram. Let ${}_{c}D$ be the two-component virtual multi-knotoid obtained by $1$-smoothing $D$ at a classical crossing $c$ in $D$. Denote by $[\overline{{}_cD}]$ and $[\overline{D}]$ the equivalence classes of flattening $\overline{{}_cD}$ of ${}_cD$ and flattening $\overline{D}$ of $D$. Furthermore, let $\overline{D}_{link}$ denote the union of the flattening $\overline{D}$ of $D$ with an unlinked copy of the unknot and denote by $[\overline{D}_{link}]$ the equivalence classes of $\overline{D}_{link}$. Let
\begin{equation}
\mathcal{L}(D)= \sum_c w_{D} (c) \, [\overline{{}_cD}] - w(D)  \, [\overline{D}_{link}], \label{eqn:defL}
\end{equation}
where the sum is taken over all classical crossings of $D$ with $w_{D}(c) = \operatorname{sgn} (c)$ and $\displaystyle w(D) = \sum_c w_{D}(c)$.}  
\end{definition}

\begin{theorem} \cite{Pet}
Let $D$ be a virtual knotoid diagram. Then $\mathcal{L} (D)$ is a virtual knotoid invariant. 
\end{theorem}

\begin{remark}\label{re1} We generalize the chord index axioms for virtual knots given in \cite{Che} to virtual knotoids.  Let $D$ be a virtual knotoid diagram and $c$ be a classical crossing of $D$. Then it is easy to see that the equivalence class $[\overline{D_c}]$ of the flat virtual knotoid diagram $\overline{D}_c$ and the equivalence class $[\overline{{}_cD}]$ of the flat virtual multi-knotoid diagram $\overline{{}_cD}$ satisfy the chord index axioms both. 
\end{remark}

\begin{corollary}
If $K$ is a classical knotoid in $S^2$, then $\mathcal{F} (K)=0$.
\end{corollary}

\begin{proof} 
Let $D$ be a diagram of a classical knotoid $K$. Then after applying 0-smoothing to any classical crossing $c$ of $D$ and flattening of the smoothed classical knotoid diagram, we obtain $\overline{D}_{c}$, which is a flat knotoid diagram of flat classical knotoid $\overline{K}_c$. By Proposition~\ref{prop2}, $\overline{K}_{c}$ and $\overline{K}$ are both flat-equivalent to the trivial flat knotoid $\overline{K_0}$, i.e. $\overline{K}_{c}=\overline{K} = \overline{K_0}$. Then $\mathcal{F} (K) = \sum_{c} w (c) \, \overline{K}_{c}- w (K) \, \overline{K} = 0$.
\end{proof}

\begin{corollary}
If $f$ is a flat virtual knotoid invariant, then $\mathcal{F}_f (K)$ defined for virtual knotoid by 
\begin{equation}
\mathcal{F}_f  (K) = \sum_{c} w_D(c) f([\overline{D_c}]) - w(D) f([\overline{D}]) \label{eqn:f}
\end{equation} 
where $D$ is a diagram of $K$, is a virtual knotoid invariant.
\end{corollary}

\begin{proof} To prove invariance of $\mathcal{F}_f (K)$ we should to check its invariance under moves $\Omega_1$, $\Omega_2$, $\Omega_3$ and $\Omega^m_3$. The arguments follow the proof of Theorem~\ref{th3.1} by replacing $[\overline{D}]$ and $[\overline{E}]$ with $f([\overline{D}])$ and  $f (]\overline{E}])$ as well as $[\overline{D_c}]$ and $[\overline{E_e}]$ with  $f([\overline{D_c}])$ and $f([\overline{E_e}])$. 
\end{proof} 

The following example shows that invariant $\mathcal{F}$ is non-trivial.

\begin{example}\label{exp4.1} {\rm 
Consider a virtual knotoid $K$, presented by a diagram $D$ in Fig.~\ref{fig19}. We will show that  $\mathcal{F} ( K ) \neq 0$. $D$ has four classical crossings $c_1$, $c_2$, $c_3$, $c_4$, and one virtual crossing. From Fig.~\ref{fig19} we have $w_D(c_1) = w_D(c_2) = w_D(c_3) = -1$, $w_D(c_4) =1$ and hence $w(D)=-2$. 
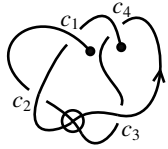
\begin{figure}[htbp]
\begin{center}
\tikzset{every picture/.style={line width=1.0pt}}  
\begin{tikzpicture}[x=0.75pt,y=0.75pt,yscale=-1.0,xscale=1.0]
\draw    (212.39,63.98) .. controls (184.39,42.48) and (218.34,15.95) .. (243.28,44.49) ;
\draw    (221.06,69.65) .. controls (232.06,76.15) and (235.84,80.45) .. (238.93,86.18) ; 
\draw    (258.89,71.81) .. controls (263.06,63.13) and (238.73,53.15) .. (252.56,37.15) ;
\draw    (259.56,30.65) .. controls (281.91,12.09) and (288.71,83.96) .. (257.16,76.69) ;
\draw    (257.06,80.81) .. controls (255.89,85.03) and (246.27,97.73) .. (238.93,86.18) ;
\draw    (257.16,76.69) .. controls (250.27,75.29) and (241.2,73.33) .. (233.6,80.62) ;
\draw    (233.6,80.62) .. controls (232.2,82.35) and (208.19,93.47) .. (217.25,66.11) ;
\draw    (217.25,66.11) .. controls (221.25,56.55) and (226.8,45.87) .. (231.56,41.15) ;
\draw    (238.06,33.31) .. controls (255.07,15.87) and (257.16,37.75) .. (258.88,42.42) ;
\draw  [fill={rgb, 255:red, 17; green, 16; blue, 16 }  ,fill opacity=1 ] (257.01,42.43) .. controls (257.02,43.46) and (257.86,44.29) .. (258.89,44.28) .. controls (259.93,44.27) and (260.76,43.43) .. (260.75,42.4) .. controls (260.74,41.37) and (259.89,40.54) .. (258.86,40.55) .. controls (257.83,40.56) and (257,41.4) .. (257.01,42.43) -- cycle ;
\draw  [fill={rgb, 255:red, 17; green, 16; blue, 16 }  ,fill opacity=1 ] (241.41,44.5) .. controls (241.42,45.53) and (242.26,46.36) .. (243.29,46.35) .. controls (244.32,46.35) and (245.15,45.5) .. (245.14,44.47) .. controls (245.14,43.44) and (244.29,42.61) .. (243.26,42.62) .. controls (242.23,42.63) and (241.4,43.47) .. (241.41,44.5) -- cycle ;
\draw   (275.01,59.58) -- (278.51,54.3) -- (280.56,60.29) ;
\draw   (229,79.7) .. controls (229,76.71) and (231.42,74.29) .. (234.42,74.29) .. controls (237.41,74.29) and (239.83,76.71) .. (239.83,79.7) .. controls (239.83,82.7) and (237.41,85.12) .. (234.42,85.12) .. controls (231.42,85.12) and (229,82.7) .. (229,79.7) -- cycle ;
\draw (252.83,18.31) node [anchor=north west][inner sep=0.75pt]  [font=\footnotesize]  {$c_{4}$};
\draw (257.35,80.3) node [anchor=north west][inner sep=0.75pt]  [font=\footnotesize]  {$c_{3}$};
\draw (226.62,24.69) node [anchor=north west][inner sep=0.75pt]  [font=\footnotesize]  {$c_{1}$};
\draw (202.43,65.88) node [anchor=north west][inner sep=0.75pt]  [font=\footnotesize]  {$c_{2}$};
\end{tikzpicture}
\caption{A virtual knotoid diagram $D$.\label{fig19}}
\end{center}
\end{figure} 

The flat virtual knotoid diagram $\overline{D}$ obtained by flattening of $D$ is presented in Fig.~\ref{fig20}. 
Let us denote a trivial knotoid diagram by $D_0$ and remark that $\overline{D_0} = D_0$. It easy to see, that after moves $f\Omega_2$, $\Omega_v$ and $f\Omega_2$ the diagram $\overline{D}$ can be reduces to $\overline{D_0}$, thus $[\overline{D}] = [\overline{D_0}]$.   
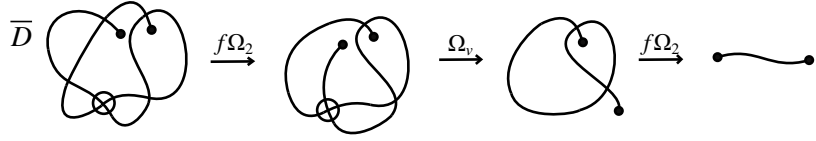
\begin{figure}[htbp]
\begin{center} 
\tikzset{every picture/.style={line width=1.0pt}}  
\begin{tikzpicture}[x=0.75pt,y=0.75pt,yscale=-1.0,xscale=1.0]
\draw    (208.74,137.09) .. controls (205.16,116.62) and (229.52,109.81) .. (245.28,129.15) ;
\draw    (208.74,137.09) .. controls (212.93,159.29) and (230.04,151.51) .. (240.93,170.84) ; 
\draw    (259.35,163.97) .. controls (263.02,150.71) and (238.66,134.86) .. (256.51,120.61) ;
\draw    (256.51,120.61) .. controls (278.86,102.05) and (290.71,168.62) .. (259.16,161.35) ;
\draw    (259.35,163.97) .. controls (258.18,168.18) and (248.27,182.4) .. (240.93,170.84) ;
\draw    (259.16,161.35) .. controls (252.27,159.96) and (243.36,158.03) .. (235.6,165.29) ;
\draw    (235.6,165.29) .. controls (234.2,167.02) and (208.76,182.66) .. (217.82,155.29) ;
\draw    (217.82,155.29) .. controls (221.82,145.73) and (233.65,125.07) .. (238.41,120.34) ;
\draw    (238.41,120.34) .. controls (255.42,102.9) and (259.16,122.42) .. (260.88,127.08) ;
\draw  [fill={rgb, 255:red, 17; green, 16; blue, 16 }  ,fill opacity=1 ] (259.01,127.1) .. controls (259.02,128.13) and (259.86,128.96) .. (260.89,128.95) .. controls (261.93,128.94) and (262.76,128.1) .. (262.75,127.07) .. controls (262.74,126.04) and (261.89,125.21) .. (260.86,125.22) .. controls (259.83,125.22) and (259,126.07) .. (259.01,127.1) -- cycle ;
\draw  [fill={rgb, 255:red, 17; green, 16; blue, 16 }  ,fill opacity=1 ] (243.41,129.17) .. controls (243.42,130.2) and (244.26,131.03) .. (245.29,131.02) .. controls (246.32,131.01) and (247.15,130.17) .. (247.14,129.14) .. controls (247.14,128.11) and (246.29,127.28) .. (245.26,127.29) .. controls (244.23,127.3) and (243.4,128.14) .. (243.41,129.17) -- cycle ;
\draw    (290.65,143.17) -- (311.42,143.03) ;
\draw   (309.31,141.05) -- (311.99,142.93) -- (309.35,144.89) ;
\draw    (357.29,134.55) .. controls (350.45,140.46) and (342.04,155.04) .. (352.93,174.38) ;
\draw    (381.95,170.17) .. controls (394.13,154.17) and (350.66,138.4) .. (368.51,124.14) ;
\draw    (368.51,124.14) .. controls (390.86,105.58) and (402.71,172.16) .. (371.16,164.89) ;
\draw    (381.95,170.17) .. controls (379.77,173.44) and (360.27,185.93) .. (352.93,174.38) ;
\draw    (371.16,164.89) .. controls (364.27,163.49) and (357.23,163.73) .. (347.6,168.82) ;
\draw    (347.6,168.82) .. controls (337.45,173.76) and (332.34,167.06) .. (329.82,158.82) ;
\draw    (329.82,158.82) .. controls (323.45,135.85) and (341.9,127.28) .. (350.41,123.88) ;
\draw    (350.41,123.88) .. controls (365.68,118.17) and (371.16,125.95) .. (372.88,130.62) ;
\draw  [fill={rgb, 255:red, 17; green, 16; blue, 16 }  ,fill opacity=1 ] (371.01,130.63) .. controls (371.02,131.66) and (371.86,132.49) .. (372.89,132.48) .. controls (373.93,132.47) and (374.76,131.63) .. (374.75,130.6) .. controls (374.74,129.57) and (373.89,128.74) .. (372.86,128.75) .. controls (371.83,128.76) and (371,129.6) .. (371.01,130.63) -- cycle ;
\draw  [fill={rgb, 255:red, 17; green, 16; blue, 16 }  ,fill opacity=1 ] (355.42,134.57) .. controls (355.43,135.6) and (356.27,136.43) .. (357.31,136.42) .. controls (358.34,136.41) and (359.17,135.57) .. (359.16,134.54) .. controls (359.15,133.51) and (358.31,132.68) .. (357.28,132.69) .. controls (356.24,132.7) and (355.41,133.54) .. (355.42,134.57) -- cycle ;
\draw   (344.36,167.6) .. controls (344.36,164.61) and (346.78,162.18) .. (349.78,162.18) .. controls (352.77,162.18) and (355.19,164.61) .. (355.19,167.6) .. controls (355.19,170.59) and (352.77,173.02) .. (349.78,173.02) .. controls (346.78,173.02) and (344.36,170.59) .. (344.36,167.6) -- cycle ;
\draw    (496.5,167.98) .. controls (500.31,155.07) and (456.16,141.15) .. (474.01,126.89) ;
\draw    (474.01,126.89) .. controls (496.36,108.33) and (504.65,160.38) .. (476.66,167.64) ;
\draw    (476.66,167.64) .. controls (475.26,169.37) and (446.79,173.98) .. (442.25,159.18) ;
\draw    (442.25,159.18) .. controls (435.45,143.18) and (452.65,126.78) .. (457.45,123.98) ;
\draw    (457.45,123.98) .. controls (474.25,115.58) and (476.66,128.7) .. (478.38,133.37) ;
\draw  [fill={rgb, 255:red, 17; green, 16; blue, 16 }  ,fill opacity=1 ] (476.51,133.38) .. controls (476.52,134.41) and (477.36,135.24) .. (478.39,135.23) .. controls (479.43,135.22) and (480.26,134.38) .. (480.25,133.35) .. controls (480.24,132.32) and (479.39,131.49) .. (478.36,131.5) .. controls (477.33,131.51) and (476.5,132.35) .. (476.51,133.38) -- cycle ;
\draw  [fill={rgb, 255:red, 17; green, 16; blue, 16 }  ,fill opacity=1 ] (494.63,168) .. controls (494.64,169.03) and (495.48,169.86) .. (496.51,169.85) .. controls (497.54,169.84) and (498.37,169) .. (498.36,167.97) .. controls (498.36,166.94) and (497.51,166.11) .. (496.48,166.12) .. controls (495.45,166.13) and (494.62,166.97) .. (494.63,168) -- cycle ;
\draw    (405.9,142.17) -- (426.67,142.03) ;
\draw   (424.36,139.85) -- (427.04,142.03) -- (424.4,144.29) ;
\draw    (506.55,142.17) -- (527.32,142.03) ;
\draw   (525.81,139.84) -- (528.05,141.93) -- (525.85,144.09) ;
\draw    (592.5,142.37) .. controls (572.5,148.21) and (566.35,134.67) .. (546.31,140.87) ;
\draw  [fill={rgb, 255:red, 17; green, 16; blue, 16 }  ,fill opacity=1 ] (544.45,140.89) .. controls (544.45,141.92) and (545.3,142.75) .. (546.33,142.74) .. controls (547.36,142.73) and (548.19,141.89) .. (548.18,140.86) .. controls (548.17,139.83) and (547.33,139) .. (546.3,139.01) .. controls (545.27,139.02) and (544.44,139.86) .. (544.45,140.89) -- cycle ;
\draw  [fill={rgb, 255:red, 17; green, 16; blue, 16 }  ,fill opacity=1 ] (590.63,142.38) .. controls (590.64,143.41) and (591.49,144.24) .. (592.52,144.23) .. controls (593.55,144.22) and (594.38,143.38) .. (594.37,142.35) .. controls (594.36,141.32) and (593.52,140.49) .. (592.49,140.5) .. controls (591.46,140.51) and (590.63,141.35) .. (590.63,142.38) -- cycle ;
\draw   (231.33,164.04) .. controls (231.33,161.04) and (233.76,158.62) .. (236.75,158.62) .. controls (239.74,158.62) and (242.17,161.04) .. (242.17,164.04) .. controls (242.17,167.03) and (239.74,169.46) .. (236.75,169.46) .. controls (233.76,169.46) and (231.33,167.03) .. (231.33,164.04) -- cycle ; 
\draw (290.59,127.52) node [anchor=north west][inner sep=0.75pt]  [font=\footnotesize]  {$f\Omega_{2}$};
\draw (408.84,128.02) node [anchor=north west][inner sep=0.75pt]  [font=\footnotesize]  {$\Omega_{v}$};
\draw (505.59,127.52) node [anchor=north west][inner sep=0.75pt]  [font=\footnotesize]  {$f\Omega_{2}$};
\draw (188,120) node [anchor=north west][inner sep=0.75pt]  [font=\normalsize]  {$\overline{D}$};			
\end{tikzpicture}
\caption{Reducing the flat virtual knotoid  $\overline{D}$ to trivial knotoid diagram $\overline{D}_0$.} \label{fig20}
\end{center}
\end{figure}

 It can be seen from Fig.~\ref{fig21} and Fig.~\ref{fig22} that  $[\overline{D_{c_1}}] =  [\overline{D_{c_2}}]$ and $[\overline{D_{c_3}}] = [\overline{D_{c_4}}] = [\overline{D_0}]$. Therefore, 
\begin{eqnarray*}
	\mathcal{F} ( K )& = &\sum_{c} w_{D}(c) \, [\overline{D_c}] - w (D) \, [\overline{D}]   
	=  - [\overline{D_{c_{1}}}] - [\overline{D_{c_{2}}}] - [\overline{D_{c_{3}}}] + [\overline{D_{c_{4}}}] - \left(-2 \, [\overline{D_0}]  \right) \\ 
	& = & -[\overline{D_{c_{1}}}] - [\overline{D_{c_{1}}}]  -  [\overline{D_0}] + [\overline{D_0}] + 2 [\overline{D_0}]   
	=  2 ([\overline{D_0}] - [\overline{D_{c_{1}}}]).   
\end{eqnarray*}

\begin{figure}[htbp]
\begin{center}
\tikzset{every picture/.style={line width=1.0pt}}  
\begin{tikzpicture}[x=0.75pt,y=0.75pt,yscale=-1.0,xscale=1.0]
\draw    (376.31,235.64) -- (404.02,235.71) ;
\draw   (402.39,233.9) -- (404.52,235.76) -- (402.42,237.71) ;
\draw    (190.74,229.75) .. controls (191.95,212.35) and (215.86,208.17) .. (222.1,207.44) ;
\draw    (190.74,229.75) .. controls (192.12,245.37) and (212.04,244.18) .. (222.93,263.51) ;
\draw    (241.35,256.64) .. controls (245.02,243.38) and (220.66,227.53) .. (238.51,213.28) ;
\draw    (238.51,213.28) .. controls (260.86,194.72) and (272.71,261.29) .. (241.16,254.02) ;
\draw    (241.35,256.64) .. controls (240.18,260.85) and (230.27,275.07) .. (222.93,263.51) ;
\draw    (241.16,254.02) .. controls (234.27,252.62) and (225.36,250.69) .. (217.6,257.96) ;
\draw    (217.6,257.96) .. controls (216.2,259.68) and (190.76,275.32) .. (199.82,247.96) ;
\draw    (199.82,247.96) .. controls (204.12,227.03) and (220.99,214.35) .. (227.28,221.82) ;
\draw    (222.1,207.44) .. controls (235.43,206.33) and (244.45,213.35) .. (242.88,219.75) ;
\draw  [fill={rgb, 255:red, 17; green, 16; blue, 16 }  ,fill opacity=1 ] (241.01,219.76) .. controls (241.02,220.79) and (241.86,221.62) .. (242.89,221.61) .. controls (243.93,221.61) and (244.76,220.76) .. (244.75,219.73) .. controls (244.74,218.7) and (243.89,217.87) .. (242.86,217.88) .. controls (241.83,217.89) and (241,218.73) .. (241.01,219.76) -- cycle ;
\draw  [fill={rgb, 255:red, 17; green, 16; blue, 16 }  ,fill opacity=1 ] (225.41,221.84) .. controls (225.42,222.87) and (226.26,223.7) .. (227.29,223.69) .. controls (228.32,223.68) and (229.15,222.84) .. (229.14,221.81) .. controls (229.14,220.78) and (228.29,219.95) .. (227.26,219.96) .. controls (226.23,219.96) and (225.4,220.81) .. (225.41,221.84) -- cycle ;
\draw   (213.33,256.7) .. controls (213.33,253.71) and (215.76,251.29) .. (218.75,251.29) .. controls (221.74,251.29) and (224.17,253.71) .. (224.17,256.7) .. controls (224.17,259.7) and (221.74,262.12) .. (218.75,262.12) .. controls (215.76,262.12) and (213.33,259.7) .. (213.33,256.7) -- cycle ;
\draw    (290.12,247.09) .. controls (287.12,218.09) and (303.85,203.81) .. (319.61,223.15) ;
\draw    (301.45,240.76) .. controls (303.45,252.76) and (311.79,256.43) .. (315.27,264.84) ;
\draw    (333.69,257.97) .. controls (337.36,244.71) and (312.99,228.86) .. (330.84,214.61) ;
\draw    (330.84,214.61) .. controls (353.19,196.05) and (365.04,262.62) .. (333.49,255.35) ;
\draw    (333.69,257.97) .. controls (332.52,262.18) and (322.6,276.4) .. (315.27,264.84) ;
\draw    (333.49,255.35) .. controls (326.6,253.96) and (317.69,252.03) .. (309.93,259.29) ;
\draw    (309.93,259.29) .. controls (308.53,261.02) and (292.45,269.09) .. (290.12,247.09) ;
\draw    (301.45,240.76) .. controls (299.85,225.42) and (307.99,219.07) .. (312.74,214.34) ;
\draw    (312.74,214.34) .. controls (329.75,196.9) and (333.49,216.42) .. (335.21,221.08) ;
\draw  [fill={rgb, 255:red, 17; green, 16; blue, 16 }  ,fill opacity=1 ] (333.34,221.1) .. controls (333.35,222.13) and (334.19,222.96) .. (335.23,222.95) .. controls (336.26,222.94) and (337.09,222.1) .. (337.08,221.07) .. controls (337.07,220.04) and (336.23,219.21) .. (335.2,219.22) .. controls (334.16,219.22) and (333.33,220.07) .. (333.34,221.1) -- cycle ;
\draw  [fill={rgb, 255:red, 17; green, 16; blue, 16 }  ,fill opacity=1 ] (317.74,223.17) .. controls (317.75,224.2) and (318.59,225.03) .. (319.62,225.02) .. controls (320.66,225.01) and (321.49,224.17) .. (321.48,223.14) .. controls (321.47,222.11) and (320.63,221.28) .. (319.59,221.29) .. controls (318.56,221.3) and (317.73,222.14) .. (317.74,223.17) -- cycle ;
\draw   (305.66,258.04) .. controls (305.66,255.04) and (308.09,252.62) .. (311.08,252.62) .. controls (314.08,252.62) and (316.5,255.04) .. (316.5,258.04) .. controls (316.5,261.03) and (314.08,263.46) .. (311.08,263.46) .. controls (308.09,263.46) and (305.66,261.03) .. (305.66,258.04) -- cycle ;
\draw    (430.74,231.25) .. controls (431.95,213.85) and (455.86,209.67) .. (462.1,208.94) ;
\draw    (430.74,231.25) .. controls (432.12,246.87) and (452.04,245.68) .. (462.93,265.01) ;
\draw    (481.35,258.14) .. controls (485.02,244.88) and (460.66,229.03) .. (478.51,214.78) ;
\draw    (478.51,214.78) .. controls (500.86,196.22) and (512.71,262.79) .. (481.16,255.52) ;
\draw    (481.35,258.14) .. controls (480.18,262.35) and (470.27,276.57) .. (462.93,265.01) ;
\draw    (481.16,255.52) .. controls (474.27,254.12) and (465.36,252.19) .. (457.6,259.46) ;
\draw    (457.6,259.46) .. controls (456.2,261.18) and (430.76,276.82) .. (439.82,249.46) ;
\draw    (439.82,249.46) .. controls (444.12,228.53) and (460.99,215.85) .. (467.28,223.32) ;
\draw    (462.1,208.94) .. controls (475.43,207.83) and (484.45,214.85) .. (482.88,221.25) ;
\draw  [fill={rgb, 255:red, 17; green, 16; blue, 16 }  ,fill opacity=1 ] (481.01,221.26) .. controls (481.02,222.29) and (481.86,223.12) .. (482.89,223.11) .. controls (483.93,223.11) and (484.76,222.26) .. (484.75,221.23) .. controls (484.74,220.2) and (483.89,219.37) .. (482.86,219.38) .. controls (481.83,219.39) and (481,220.23) .. (481.01,221.26) -- cycle ;
\draw  [fill={rgb, 255:red, 17; green, 16; blue, 16 }  ,fill opacity=1 ] (465.41,223.34) .. controls (465.42,224.37) and (466.26,225.2) .. (467.29,225.19) .. controls (468.32,225.18) and (469.15,224.34) .. (469.14,223.31) .. controls (469.14,222.28) and (468.29,221.45) .. (467.26,221.46) .. controls (466.23,221.46) and (465.4,222.31) .. (465.41,223.34) -- cycle ;
\draw   (453.33,258.2) .. controls (453.33,255.21) and (455.76,252.79) .. (458.75,252.79) .. controls (461.74,252.79) and (464.17,255.21) .. (464.17,258.2) .. controls (464.17,261.2) and (461.74,263.62) .. (458.75,263.62) .. controls (455.76,263.62) and (453.33,261.2) .. (453.33,258.2) -- cycle ;
\draw (376.05,222.53) node [anchor=north west][inner sep=0.75pt]  [font=\scriptsize]  {$\text{isotopy}$};
\draw (215.8,275.3) node [anchor=north west][inner sep=0.75pt]  [font=\normalsize]  {$\overline{D}_{c_{1}}$};
\draw (326.5,275.3) node [anchor=north west][inner sep=0.75pt]  [font=\normalsize]  {$\overline{D}_{c_{2}}$};
\draw (423.71,275.3) node [anchor=north west][inner sep=0.75pt]  [font=\normalsize]  {$\overline{D}_{c_{1}}$};		
\end{tikzpicture}
\caption{Flat virtual knotoid diagrams $\overline{D_{c_1}}$ and $\overline{D_{c_2}}$.} \label{fig21}
\end{center}
\end{figure}
	
\begin{figure}[htbp]
\begin{center}
\tikzset{every picture/.style={line width=1.0pt}}  
\begin{tikzpicture}[x=0.75pt,y=0.75pt,yscale=-1.0,xscale=1.0]
\draw    (196.74,347.09) .. controls (193.16,326.62) and (217.52,319.81) .. (233.28,339.15) ;
\draw    (196.74,347.09) .. controls (200.93,369.29) and (212.45,357.13) .. (228.93,380.84) ;
\draw    (224.4,374.42) .. controls (246.95,355.13) and (229.95,346.13) .. (244.51,330.61) ;
\draw    (244.51,330.61) .. controls (266.86,312.05) and (267.7,356.13) .. (259.95,368.13) ;
\draw    (259.95,368.13) .. controls (256.7,374.63) and (238.2,392.23) .. (228.93,380.84) ;
\draw    (224.4,374.42) .. controls (223,376.15) and (196.76,392.66) .. (205.82,365.29) ;
\draw    (205.82,365.29) .. controls (209.82,355.73) and (222.13,334.51) .. (226.89,329.79) ;
\draw    (226.41,330.34) .. controls (243.42,312.9) and (247.16,332.42) .. (248.88,337.08) ;
\draw  [fill={rgb, 255:red, 17; green, 16; blue, 16 }  ,fill opacity=1 ] (247.01,337.1) .. controls (247.02,338.13) and (247.86,338.96) .. (248.89,338.95) .. controls (249.93,338.94) and (250.76,338.1) .. (250.75,337.07) .. controls (250.74,336.04) and (249.89,335.21) .. (248.86,335.22) .. controls (247.83,335.22) and (247,336.07) .. (247.01,337.1) -- cycle ;
\draw  [fill={rgb, 255:red, 17; green, 16; blue, 16 }  ,fill opacity=1 ] (231.41,339.17) .. controls (231.42,340.2) and (232.26,341.03) .. (233.29,341.02) .. controls (234.32,341.01) and (235.15,340.17) .. (235.14,339.14) .. controls (235.14,338.11) and (234.29,337.28) .. (233.26,337.29) .. controls (232.23,337.3) and (231.4,338.14) .. (231.41,339.17) -- cycle ;
\draw    (200.74,444.95) .. controls (197.16,424.49) and (221.52,417.67) .. (237.28,437.02) ;
\draw    (200.74,444.95) .. controls (204.93,467.16) and (222.04,459.38) .. (232.93,478.71) ;
\draw    (252.19,467.44) .. controls (248.03,452.46) and (256.03,440.18) .. (248.51,428.48) ;
\draw    (257.13,444.2) .. controls (261.93,403) and (285.11,475.08) .. (251.16,469.22) ;
\draw    (252.19,467.44) .. controls (254.83,476.53) and (240.27,490.27) .. (232.93,478.71) ;
\draw    (251.16,469.22) .. controls (244.27,467.82) and (235.2,465.86) .. (227.6,473.16) ;
\draw    (227.6,473.16) .. controls (226.2,474.88) and (200.76,490.52) .. (209.82,463.16) ;
\draw    (209.82,463.16) .. controls (213.82,453.6) and (225.65,432.94) .. (230.41,428.21) ;
\draw    (230.41,428.21) .. controls (235.86,421.1) and (244.53,421.77) .. (248.51,428.48) ;
\draw  [fill={rgb, 255:red, 17; green, 16; blue, 16 }  ,fill opacity=1 ] (255.26,444.21) .. controls (255.27,445.24) and (256.11,446.07) .. (257.14,446.06) .. controls (258.17,446.05) and (259,445.21) .. (259,444.18) .. controls (258.99,443.15) and (258.14,442.32) .. (257.11,442.33) .. controls (256.08,442.34) and (255.25,443.18) .. (255.26,444.21) -- cycle ;
\draw  [fill={rgb, 255:red, 17; green, 16; blue, 16 }  ,fill opacity=1 ] (235.41,437.04) .. controls (235.42,438.07) and (236.26,438.9) .. (237.29,438.89) .. controls (238.32,438.88) and (239.15,438.04) .. (239.14,437.01) .. controls (239.14,435.98) and (238.29,435.15) .. (237.26,435.16) .. controls (236.23,435.16) and (235.4,436.01) .. (235.41,437.04) -- cycle ;
\draw   (222.83,472.06) .. controls (222.83,469.07) and (225.26,466.64) .. (228.25,466.64) .. controls (231.24,466.64) and (233.67,469.07) .. (233.67,472.06) .. controls (233.67,475.05) and (231.24,477.48) .. (228.25,477.48) .. controls (225.26,477.48) and (222.83,475.05) .. (222.83,472.06) -- cycle ;
\draw    (290.22,350.05) -- (310.98,349.9) ;
\draw    (290.48,450.15) -- (311.25,450.01) ;
\draw    (362.28,339.49) .. controls (345.45,344.09) and (341.45,357.46) .. (357.93,381.18) ;
\draw    (353.4,374.76) .. controls (373.77,357.98) and (358.95,346.46) .. (373.51,330.94) ;
\draw    (373.51,330.94) .. controls (395.86,312.38) and (396.7,356.46) .. (388.95,368.46) ;
\draw    (388.95,368.46) .. controls (385.7,374.96) and (367.2,392.56) .. (357.93,381.18) ;
\draw    (353.4,374.76) .. controls (348.61,377.81) and (336.12,383.09) .. (332.12,363.43) ;
\draw    (332.12,363.43) .. controls (329.79,345.76) and (338.02,328.07) .. (354.12,324.76) ;
\draw    (354.12,324.76) .. controls (373.45,321.43) and (376.16,332.75) .. (377.88,337.42) ;
\draw  [fill={rgb, 255:red, 17; green, 16; blue, 16 }  ,fill opacity=1 ] (376.01,337.43) .. controls (376.02,338.46) and (376.86,339.29) .. (377.89,339.28) .. controls (378.93,339.27) and (379.76,338.43) .. (379.75,337.4) .. controls (379.74,336.37) and (378.89,335.54) .. (377.86,335.55) .. controls (376.83,335.56) and (376,336.4) .. (376.01,337.43) -- cycle ;
\draw  [fill={rgb, 255:red, 17; green, 16; blue, 16 }  ,fill opacity=1 ] (360.41,339.5) .. controls (360.42,340.53) and (361.26,341.36) .. (362.29,341.35) .. controls (363.32,341.35) and (364.15,340.5) .. (364.14,339.47) .. controls (364.14,338.44) and (363.29,337.61) .. (362.26,337.62) .. controls (361.23,337.63) and (360.4,338.47) .. (360.41,339.5) -- cycle ;
\draw   (347.98,374.76) .. controls (347.98,371.76) and (350.41,369.34) .. (353.4,369.34) .. controls (356.39,369.34) and (358.82,371.76) .. (358.82,374.76) .. controls (358.82,377.75) and (356.39,380.17) .. (353.4,380.17) .. controls (350.41,380.17) and (347.98,377.75) .. (347.98,374.76) -- cycle ;
\draw    (410.9,350.24) -- (431.67,350.09) ;
\draw    (461.74,375.09) .. controls (474.94,365.65) and (467.28,346.79) .. (481.84,331.28) ;
\draw    (481.84,331.28) .. controls (504.19,312.72) and (505.03,356.79) .. (497.28,368.79) ;
\draw    (497.28,368.79) .. controls (494.03,375.29) and (475.53,392.89) .. (466.27,381.51) ;
\draw    (461.74,375.09) .. controls (458.94,377.31) and (444.45,383.43) .. (440.45,363.76) ;
\draw    (440.45,363.76) .. controls (438.12,346.09) and (446.36,328.4) .. (462.45,325.09) ;
\draw    (462.45,325.09) .. controls (481.79,321.76) and (484.49,333.08) .. (486.21,337.75) ;
\draw  [fill={rgb, 255:red, 17; green, 16; blue, 16 }  ,fill opacity=1 ] (484.34,337.76) .. controls (484.35,338.79) and (485.19,339.62) .. (486.23,339.61) .. controls (487.26,339.61) and (488.09,338.76) .. (488.08,337.73) .. controls (488.07,336.7) and (487.23,335.87) .. (486.2,335.88) .. controls (485.16,335.89) and (484.33,336.73) .. (484.34,337.76) -- cycle ;
\draw  [fill={rgb, 255:red, 17; green, 16; blue, 16 }  ,fill opacity=1 ] (464.4,381.53) .. controls (464.41,382.56) and (465.25,383.39) .. (466.28,383.38) .. controls (467.31,383.37) and (468.14,382.53) .. (468.14,381.5) .. controls (468.13,380.47) and (467.28,379.64) .. (466.25,379.65) .. controls (465.22,379.65) and (464.39,380.5) .. (464.4,381.53) -- cycle ;
\draw    (509.88,350.2) -- (530.65,350.06) ;
\draw    (595.84,350.39) .. controls (575.84,356.24) and (569.68,342.7) .. (549.65,348.9) ;
\draw  [fill={rgb, 255:red, 17; green, 16; blue, 16 }  ,fill opacity=1 ] (547.78,348.92) .. controls (547.79,349.95) and (548.63,350.78) .. (549.66,350.77) .. controls (550.7,350.76) and (551.52,349.92) .. (551.52,348.89) .. controls (551.51,347.86) and (550.66,347.03) .. (549.63,347.03) .. controls (548.6,347.04) and (547.77,347.89) .. (547.78,348.92) -- cycle ;
\draw  [fill={rgb, 255:red, 17; green, 16; blue, 16 }  ,fill opacity=1 ] (593.97,350.41) .. controls (593.98,351.44) and (594.82,352.27) .. (595.85,352.26) .. controls (596.88,352.25) and (597.71,351.41) .. (597.7,350.38) .. controls (597.7,349.35) and (596.85,348.52) .. (595.82,348.53) .. controls (594.79,348.53) and (593.96,349.38) .. (593.97,350.41) -- cycle ;
\draw    (355.79,478.71) .. controls (335.9,458.7) and (350.18,432.62) .. (360.13,437.02) ;
\draw    (379.98,444.2) .. controls (384.78,403) and (407.97,475.08) .. (374.02,469.22) ;
\draw    (374.02,469.22) .. controls (367.12,467.82) and (358.06,465.86) .. (350.46,473.16) ;
\draw    (350.46,473.16) .. controls (349.05,474.88) and (327.9,482.62) .. (326.47,460.33) ;
\draw    (326.47,460.33) .. controls (326.18,445.76) and (332.75,430.62) .. (343.9,422.91) ;
\draw    (343.9,422.91) .. controls (351.04,417.19) and (365.11,418.57) .. (371.36,428.48) ;
\draw  [fill={rgb, 255:red, 17; green, 16; blue, 16 }  ,fill opacity=1 ] (378.12,444.21) .. controls (378.12,445.24) and (378.97,446.07) .. (380,446.06) .. controls (381.03,446.05) and (381.86,445.21) .. (381.85,444.18) .. controls (381.84,443.15) and (381,442.32) .. (379.97,442.33) .. controls (378.94,442.34) and (378.11,443.18) .. (378.12,444.21) -- cycle ;
\draw  [fill={rgb, 255:red, 17; green, 16; blue, 16 }  ,fill opacity=1 ] (358.26,437.04) .. controls (358.27,438.07) and (359.12,438.9) .. (360.15,438.89) .. controls (361.18,438.88) and (362.01,438.04) .. (362,437.01) .. controls (361.99,435.98) and (361.15,435.15) .. (360.12,435.16) .. controls (359.09,435.16) and (358.26,436.01) .. (358.26,437.04) -- cycle ;
\draw   (345.69,472.06) .. controls (345.69,469.07) and (348.11,466.64) .. (351.11,466.64) .. controls (354.1,466.64) and (356.53,469.07) .. (356.53,472.06) .. controls (356.53,475.05) and (354.1,477.48) .. (351.11,477.48) .. controls (348.11,477.48) and (345.69,475.05) .. (345.69,472.06) -- cycle ;
\draw    (375.05,467.44) .. controls (370.88,452.46) and (378.88,440.18) .. (371.36,428.48) ;
\draw    (375.05,467.44) .. controls (377.69,476.53) and (363.12,490.27) .. (355.79,478.71) ;
\draw    (478.18,471.09) .. controls (495.32,452.8) and (487.32,431.09) .. (484.79,428.48) ;
\draw    (495.32,456.8) .. controls (483.61,464.52) and (455.04,469.66) .. (447.32,459.94) ;
\draw    (447.32,459.94) .. controls (441.79,453.18) and (446.18,430.62) .. (457.32,422.91) ;
\draw    (457.32,422.91) .. controls (464.18,418.8) and (477.61,418.52) .. (484.79,428.48) ;
\draw  [fill={rgb, 255:red, 17; green, 16; blue, 16 }  ,fill opacity=1 ] (493.46,456.82) .. controls (493.46,457.85) and (494.31,458.68) .. (495.34,458.67) .. controls (496.37,458.66) and (497.2,457.82) .. (497.19,456.79) .. controls (497.18,455.76) and (496.34,454.93) .. (495.31,454.94) .. controls (494.28,454.94) and (493.45,455.79) .. (493.46,456.82) -- cycle ;
\draw  [fill={rgb, 255:red, 17; green, 16; blue, 16 }  ,fill opacity=1 ] (476.31,471.1) .. controls (476.32,472.13) and (477.17,472.96) .. (478.2,472.95) .. controls (479.23,472.95) and (480.06,472.1) .. (480.05,471.07) .. controls (480.04,470.04) and (479.2,469.21) .. (478.17,469.22) .. controls (477.13,469.23) and (476.3,470.07) .. (476.31,471.1) -- cycle ;
\draw    (412.04,450.24) -- (432.81,450.1) ;
\draw    (513.22,450.25) -- (533.98,450.11) ;
\draw    (599.17,450.45) .. controls (579.17,456.29) and (573.02,442.75) .. (552.98,448.95) ;
\draw  [fill={rgb, 255:red, 17; green, 16; blue, 16 }  ,fill opacity=1 ] (551.11,448.97) .. controls (551.12,450) and (551.96,450.83) .. (553,450.82) .. controls (554.03,450.81) and (554.86,449.97) .. (554.85,448.94) .. controls (554.84,447.91) and (554,447.08) .. (552.97,447.09) .. controls (551.93,447.1) and (551.1,447.94) .. (551.11,448.97) -- cycle ;
\draw  [fill={rgb, 255:red, 17; green, 16; blue, 16 }  ,fill opacity=1 ] (597.3,450.46) .. controls (597.31,451.49) and (598.15,452.32) .. (599.18,452.31) .. controls (600.22,452.3) and (601.05,451.46) .. (601.04,450.43) .. controls (601.03,449.4) and (600.19,448.57) .. (599.15,448.58) .. controls (598.12,448.59) and (597.29,449.43) .. (597.3,450.46) -- cycle ;
\draw   (309.5,348.02) -- (311.63,349.89) -- (309.53,351.83) ;
\draw   (430.17,348.12) -- (432.3,349.99) -- (430.2,351.93) ;
\draw   (529.21,348.12) -- (531.34,349.99) -- (529.24,351.93) ;
\draw   (309.72,448.12) -- (311.85,449.99) -- (309.75,451.93) ;
\draw   (431.13,448.34) -- (433.26,450.21) -- (431.16,452.16) ;
\draw   (532.24,448.12) -- (534.37,449.99) -- (532.27,451.93) ;
\draw   (218.98,374.42) .. controls (218.98,371.43) and (221.41,369) .. (224.4,369) .. controls (227.39,369) and (229.82,371.43) .. (229.82,374.42) .. controls (229.82,377.41) and (227.39,379.84) .. (224.4,379.84) .. controls (221.41,379.84) and (218.98,377.41) .. (218.98,374.42) -- cycle ;
\draw (290.15,334.39) node [anchor=north west][inner sep=0.75pt]  [font=\footnotesize]  {$f\Omega_{2}$};
\draw (290.42,434.5) node [anchor=north west][inner sep=0.75pt]  [font=\footnotesize]  {$f\Omega_{2}$};
\draw (414.84,334.39) node [anchor=north west][inner sep=0.75pt]  [font=\footnotesize]  {$\Omega_{v}$};
\draw (509.92,334.39) node [anchor=north west][inner sep=0.75pt]  [font=\footnotesize]  {$f\Omega_{1}$};
\draw (170.88,335.74) node [anchor=north west][inner sep=0.75pt]  [font=\normalsize]  {$\overline{D}_{c_{3}}$};
\draw (175.21,438.13) node [anchor=north west][inner sep=0.75pt]  [font=\normalsize]  {$\overline{D}_{c_{4}}$};
\draw (415.98,434.5) node [anchor=north west][inner sep=0.75pt]  [font=\footnotesize]  {$\Omega_{v}$};
\draw (513.25,434.5) node [anchor=north west][inner sep=0.75pt]  [font=\footnotesize]  {$f\Omega_{1}$};

\end{tikzpicture}
\caption{Flat virtual knotoid diagrams $\overline{D_{c_3}}$ and $\overline{D_{c_4}}$. } \label{fig22}
\end{center}
\end{figure}
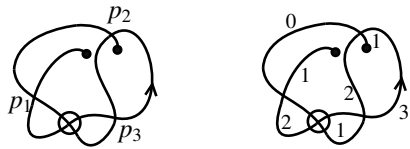
\begin{figure}[htbp]
\begin{center}
\tikzset{every picture/.style={line width=1.0pt}}  
\begin{tikzpicture}[x=0.75pt,y=0.75pt,yscale=-1.0,xscale=1.0]
\draw    (205.31,546.42) .. controls (206.53,529.02) and (230.43,524.84) .. (236.67,524.11) ;
\draw    (205.31,546.42) .. controls (206.69,562.04) and (226.62,560.85) .. (237.5,580.18) ;
\draw    (255.92,573.31) .. controls (259.59,560.05) and (235.23,544.2) .. (253.08,529.95) ;
\draw    (253.08,529.95) .. controls (275.43,511.39) and (287.28,577.96) .. (255.73,570.69) ;
\draw    (255.92,573.31) .. controls (254.75,577.52) and (244.84,591.74) .. (237.5,580.18) ;
\draw    (255.73,570.69) .. controls (248.84,569.29) and (239.93,567.36) .. (232.17,574.63) ;
\draw    (232.17,574.63) .. controls (230.77,576.35) and (205.33,591.99) .. (214.39,564.63) ;
\draw    (214.39,564.63) .. controls (218.69,543.7) and (235.56,531.02) .. (241.85,538.49) ;
\draw    (236.67,524.11) .. controls (250.01,523) and (259.03,530.02) .. (257.45,536.42) ;
\draw  [fill={rgb, 255:red, 17; green, 16; blue, 16 }  ,fill opacity=1 ] (255.58,536.43) .. controls (255.59,537.46) and (256.43,538.29) .. (257.46,538.28) .. controls (258.5,538.28) and (259.33,537.43) .. (259.32,536.4) .. controls (259.31,535.37) and (258.47,534.54) .. (257.43,534.55) .. controls (256.4,534.56) and (255.57,535.4) .. (255.58,536.43) -- cycle ;
\draw  [fill={rgb, 255:red, 17; green, 16; blue, 16 }  ,fill opacity=1 ] (239.98,538.51) .. controls (239.99,539.54) and (240.83,540.37) .. (241.86,540.36) .. controls (242.89,540.35) and (243.72,539.51) .. (243.72,538.48) .. controls (243.71,537.45) and (242.86,536.62) .. (241.83,536.63) .. controls (240.8,536.63) and (239.97,537.48) .. (239.98,538.51) -- cycle ;
\draw   (227.9,573.37) .. controls (227.9,570.38) and (230.33,567.96) .. (233.32,567.96) .. controls (236.31,567.96) and (238.74,570.38) .. (238.74,573.37) .. controls (238.74,576.37) and (236.31,578.79) .. (233.32,578.79) .. controls (230.33,578.79) and (227.9,576.37) .. (227.9,573.37) -- cycle ;
\draw    (330.88,545.57) .. controls (332.1,528.16) and (356,523.98) .. (362.24,523.26) ;
\draw    (330.88,545.57) .. controls (332.26,561.18) and (352.19,559.99) .. (363.08,579.32) ;
\draw    (381.5,572.45) .. controls (385.17,559.19) and (360.8,543.34) .. (378.65,529.09) ;
\draw    (378.65,529.09) .. controls (401,510.53) and (412.85,577.1) .. (381.3,569.83) ;
\draw    (381.5,572.45) .. controls (380.33,576.66) and (370.41,590.88) .. (363.08,579.32) ;
\draw    (381.3,569.83) .. controls (374.41,568.44) and (365.5,566.51) .. (357.74,573.77) ;
\draw    (357.74,573.77) .. controls (356.34,575.5) and (330.9,591.14) .. (339.97,563.77) ;
\draw    (339.97,563.77) .. controls (344.26,542.85) and (361.13,530.17) .. (367.42,537.63) ;
\draw    (362.24,523.26) .. controls (375.58,522.15) and (384.6,529.16) .. (383.02,535.56) ;
\draw  [fill={rgb, 255:red, 17; green, 16; blue, 16 }  ,fill opacity=1 ] (381.15,535.58) .. controls (381.16,536.61) and (382,537.44) .. (383.04,537.43) .. controls (384.07,537.42) and (384.9,536.58) .. (384.89,535.55) .. controls (384.88,534.52) and (384.04,533.69) .. (383.01,533.7) .. controls (381.97,533.7) and (381.14,534.55) .. (381.15,535.58) -- cycle ;
\draw  [fill={rgb, 255:red, 17; green, 16; blue, 16 }  ,fill opacity=1 ] (365.55,537.65) .. controls (365.56,538.68) and (366.4,539.51) .. (367.43,539.5) .. controls (368.47,539.49) and (369.3,538.65) .. (369.29,537.62) .. controls (369.28,536.59) and (368.43,535.76) .. (367.4,535.77) .. controls (366.37,535.78) and (365.54,536.62) .. (365.55,537.65) -- cycle ;
\draw   (353.47,572.52) .. controls (353.47,569.52) and (355.9,567.1) .. (358.89,567.1) .. controls (361.89,567.1) and (364.31,569.52) .. (364.31,572.52) .. controls (364.31,575.51) and (361.89,577.94) .. (358.89,577.94) .. controls (355.9,577.94) and (353.47,575.51) .. (353.47,572.52) -- cycle ;
\draw   (271.51,559.11) -- (275.01,553.83) -- (277.06,559.82) ;
\draw   (397.26,556.36) -- (400.76,551.08) -- (402.81,557.07) ;
\draw (200.43,558.35) node [anchor=north west][inner sep=0.75pt]  [font=\small]  {$p_{1}$};
\draw (250.38,512.69) node [anchor=north west][inner sep=0.75pt]  [font=\small]  {$p_{2}$};
\draw (255.73,572.69) node [anchor=north west][inner sep=0.75pt]  [font=\small]  {$p_{3}$};
\draw (365.72,571.52) node [anchor=north west][inner sep=0.75pt]  [font=\scriptsize]  {$1$};
\draw (383.81,526.54) node [anchor=north west][inner sep=0.75pt]  [font=\scriptsize]  {$1$};
\draw (340.62,516.59) node [anchor=north west][inner sep=0.75pt]  [font=\scriptsize]  {$0$};
\draw (369.96,553.11) node [anchor=north west][inner sep=0.75pt]  [font=\scriptsize]  {$2$};
\draw (397.74,562.53) node [anchor=north west][inner sep=0.75pt]  [font=\scriptsize]  {$3$};
\draw (338.39,567.42) node [anchor=north west][inner sep=0.75pt]  [font=\scriptsize]  {$2$};
\draw (347.86,544.02) node [anchor=north west][inner sep=0.75pt]  [font=\scriptsize]  {$1$};			
\end{tikzpicture}
\caption{An oriented flat virtual knotoid $\overline{D_{c_1}}$ and its arc labelling.} \label{fig23}
\end{center}
\end{figure}

To demonstrate non-triviality of $[\overline{D_{c_1}}]$ we compute the flat affine index polynomial $Q_{\overline{D_{c_1}}} (t)$. Denote crossings of $\overline{D_{c_1}}$ by $p_1$, $p_2$ and $p_3$ and assigning integer labelling to the arcs as in Fig.~\ref{fig23}. Then we get 
$$
W_+(p_1) = 1 - 2 = -1, \quad W_+(p_2) = 3 - 1 = 2, \quad W_+(p_3) = 2-3=-1.
$$ 

Therefore,  
$$
f_1(\overline{D_{c_1}}) = \sum_{|W_+(p)| = 1} \operatorname{sign} (W_+(p)) = -2, \qquad  f_2(\overline{D_{c_1}}) = \sum_{|W_+(p)| = 2} \operatorname{sign} (W_+(p)) = 1,
$$ 
and finally, 
$$
Q_{\overline{D_{c_1}}} (t) = \sum_{n \in \mathbb{Z}_{>0}} f_n (\overline{D_{c_1}}) t^n = t^2 - 2t \neq 0. 
$$
Thus, $[\overline{D_{c_1}}] \neq [\overline{D_0}]$, and then $\mathcal{F} (K) \neq 0$. The example is completed. 
}
\end{example}

Subsequently, we discuss several properties of $\mathcal{F} (K)$.

\begin{definition} {\rm 
Let $D$ be a virtual knotoid diagram. 
\begin{itemize} 
\item Denote by $r(D)$ the orientation-reversed virtual knotoid diagram of $D$, and call $r(D)$ the \textit{orientation-reversed image} of $D$.
\item Denote by $m(D)$ the diagram obtained by switches the under- and over-strands of every classical crossings in $D$, and call $m(D)$ the \textit{mirror image} of $D$.
\end{itemize}
	}
\end{definition}

\begin{theorem} \label{th3.3}
Let $K$ be an oriented virtual knotoid, $r(K)$ and $m(K)$ the  orientation-reversed image and mirror image of $K$ respectively. Then we have
\begin{enumerate}
\item[{(1)}]
$\mathcal{F} (r(K) )=\mathcal{F} ( K )$ and $\mathcal{L} (r(K) )=\mathcal{L} ( K )$;
\item[{(2)}]
$\mathcal{F} ( m(K) )=-\mathcal{F} ( K )$ and $\mathcal{L} ( m(K) )=-\mathcal{L} ( K )$.
\end{enumerate}
\end{theorem}
  
\begin{proof}	  	
Consider a virtual knotoid diagram $D$ of $K$. Let $r(D)$ be the orientation-reversed image of $D$, and $r(D)$ be the mirror image of $D$.  
For a classical crossing $c$ in $D$, we denote the corresponding crossing in $r(D)$ and $m(D)$ by $r(c)$ and $m(c)$ respectively. Integer labeling of $D$ induce integer labeling of $r(D)$ and $m(D)$ as presented in Fig.~\ref{fig111}. 

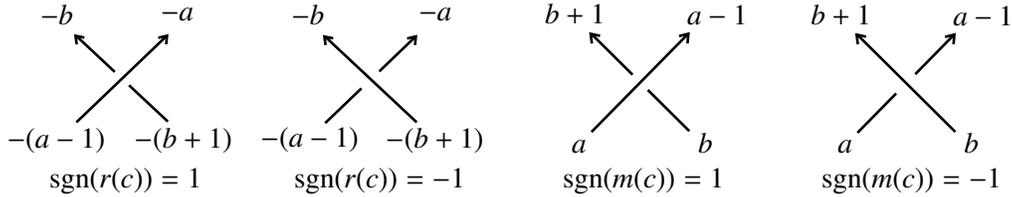
\begin{figure}[!ht]
\begin{center}
\tikzset{every picture/.style={line width=1.0pt}}  
\begin{tikzpicture}[x=0.75pt,y=0.75pt,yscale=-1,xscale=1]
\draw    (260.59,457.28) -- (306.85,500.58) ;
\draw    (180.71,456.99) -- (135.19,500.6) ;
\draw    (135.09,456.91) -- (154.57,475.13) ;
\draw   (135.37,461.22) -- (134.79,456.94) -- (139.24,457.8) ;
\draw   (176.5,457.92) -- (180.9,456.9) -- (180.49,461.2) ;
\draw    (161.75,481.87) -- (181.35,500.21) ;
\draw   (260.87,461.59) -- (260.29,457.31) -- (264.73,458.17) ;
\draw   (301.99,458.28) -- (306.4,457.26) -- (305.99,461.57) ;
\draw    (306.21,457.36) -- (287.94,474.89) ;
\draw    (278.96,483.45) -- (260.68,500.97) ;
\draw    (529.04,456.33) -- (578.49,505.42) ;
\draw    (443.66,456.01) -- (394.99,505.45) ;
\draw    (394.89,455.91) -- (415.71,476.56) ;
\draw   (395.19,460.8) -- (394.57,455.94) -- (399.32,456.93) ;
\draw   (439.15,457.05) -- (443.86,455.9) -- (443.42,460.77) ;
\draw    (423.39,484.22) -- (444.34,505) ;
\draw   (529.33,461.22) -- (528.72,456.36) -- (533.47,457.34) ;
\draw   (573.3,457.47) -- (578.01,456.31) -- (577.57,461.19) ;
\draw    (577.8,456.42) -- (558.27,476.3) ;
\draw    (548.67,486) -- (529.14,505.87) ;
\draw (98.67,502.11) node [anchor=north west][inner sep=0.75pt]  {$-(a-1)$};
\draw (162.82,501.7) node [anchor=north west][inner sep=0.75pt]  {$-(b+1)$};
\draw (115.35,439.81) node [anchor=north west][inner sep=0.75pt]  {$-b$};
\draw (176.1,439.99) node [anchor=north west][inner sep=0.75pt]   {$-a$};
\draw (242,439.81) node [anchor=north west][inner sep=0.75pt]  {$-b$};
\draw (306.24,439.99) node [anchor=north west][inner sep=0.75pt]   {$-a$};
\draw (383.49,507.7) node [anchor=north west][inner sep=0.75pt]   {$a$};
\draw (447.1,504.57) node [anchor=north west][inner sep=0.75pt]   {$b$};
\draw (369.85,439.31) node [anchor=north west][inner sep=0.75pt]   {$b+1$};
\draw (441.6,439.75) node [anchor=north west][inner sep=0.75pt]   {$a-1$};
\draw (517.64,507.7) node [anchor=north west][inner sep=0.75pt]   {$a$};
\draw (581.25,504.57) node [anchor=north west][inner sep=0.75pt]  {$b$};
\draw (504,439.73) node [anchor=north west][inner sep=0.75pt]   {$b+1$};
\draw (575.74,440.17) node [anchor=north west][inner sep=0.75pt]   {$a-1$};
\draw (226.56,501.67) node [anchor=north west][inner sep=0.75pt]  {$-(a-1)$};
\draw (289.63,501.7) node [anchor=north west][inner sep=0.75pt]  {$-(b+1)$};
\draw (119.94,521.87) node [anchor=north west][inner sep=0.75pt]    {$\operatorname{sgn}(r(c))=1$};
\draw (243.44,521.87) node [anchor=north west][inner sep=0.75pt]    {$\operatorname{sgn}(r(c))=-1$};
\draw (378.94,521.87) node [anchor=north west][inner sep=0.75pt]    {$\operatorname{sgn}(m(c))=1$};
\draw (509.44,521.87) node [anchor=north west][inner sep=0.75pt]    {$\operatorname{sgn}(m(c))=-1$};
\end{tikzpicture}
\caption{Assigning labels around classical crossings $r(c) \in r(D)$ and $m(c) \in m(D)$ .} \label{fig111}
\end{center}
\end{figure}	

(1) Now we first prove that $\mathcal{F} (r(K) )=\mathcal{F} ( K )$. 
It is easy to see that $\operatorname{sgn}(r(c)) = \operatorname{sgn}(c)$. Moreover, $w_{D}(c)=w_{r(D)}(r(c))$, $w(D)=w(r(D))$, and $[\overline{D_c}] = [\overline{r(D)_{r(c)}}]$ with $[\overline{D}] = [\overline{r(D)}]$. Then   
\begin{eqnarray*}
\mathcal{F} (r(K)) &= & \mathcal{F} (r(D)) = \sum _ {r(c)} w_{r(D)} (r(c)) \, [\overline{r(D)_{r(c)}}] - w (r(D)) \, [\overline{r(D)}]  \cr 
& = & \sum_{c} w_D(c) \, [\overline{D_c}] - w(D) \, [\overline{D}] 
=\mathcal{F} (D) = \mathcal{F} (K).
\end{eqnarray*}

Next we prove that $\mathcal{L} (r(K) )=\mathcal{L} (K)$. It is easy to see that $\operatorname{sgn}(r(c)) = \operatorname{sgn}(c)$. Moreover, $w_{D}(c)=w_{r(D)}(r(c))$, $w(D)=w(r(D))$, and $[\overline{{}_cD}] = [\overline{{}_{r(c)}r(D)}]$ with $[\overline{D}_{link}] = [\overline{r(D)}_{link}]$. Then   
\begin{eqnarray*}
\mathcal{L} (r(K)) & = & \mathcal{L} (r(D)) = \sum _ {r(c)} w_{r(D)} (r(c)) \, [\overline{{}_{r(c)}r(D)}] - w (r(D)) \, [\overline{r(D)}_{link}] \cr 
& = & \sum_{c} w_D(c) \, [\overline{{}_cD}] - w(D) \, [\overline{D}_{link}] 
=\mathcal{L} (D) = \mathcal{L} (K).
\end{eqnarray*}

(2) We first prove that $\mathcal{F} ( m(K) )=-\mathcal{F} ( K )$. It is easy to see that $\operatorname{sgn}(m(c)) = - \operatorname{sgn}(c)$. Moreover, $w_D (c)=-w_{m(D)}(m(c))$, $w(D)=-w(m(D))$ and $[\overline{D_c}] = [\overline{m(D)_{m(c)}}]$ with  $[\overline{m(D)} ] = [\overline{D}]$. Then  
\begin{eqnarray*}
\mathcal{F} (m(K))&= & \mathcal{F} (m(D)) =  \sum_{m(c)} w_{m(D)} (m(c)) \, [\overline{m(D)_{m(c)}}] - w(m(D)) \, [\overline{m(D)}] 
=\sum_{c} (-w_{D}(c) ) \, [\overline{D_c}] + w(D) \, [\overline{D}] \\
&=& - \Big( \sum_{c} w_D(c) \, [\overline{D_c}] - w(D) \, [\overline{D}] \Big)
=- \mathcal{F} (D) = - \mathcal{F} (K).
\end{eqnarray*}

Next we prove that $\mathcal{L} ( m(K) )=-\mathcal{L} ( K )$.  It is easy to see that $\operatorname{sgn}(m(c)) = - \operatorname{sgn}(c)$. Moreover, $w_D (c)=-w_{m(D)}(m(c))$, $w(D)=-w(m(D))$ and $[\overline{{}_cD}] = [\overline{{}_{m(c)}m(D)}]$ with  $[\overline{m(D)}_{link}] = [\overline{D}_{link}]$. Then  
\begin{eqnarray*}
\mathcal{L} (m(K))&= & \mathcal{L} (m(D)) =  \sum_{m(c)} w_{m(D)} (m(c)) \, [\overline{{}_{m(c)}m(D)}] - w(m(D)) \, [\overline{m(D)}_{link}] 
=\sum_{c} (-w_{D}(c) ) \, [\overline{{}_cD}] + w(D) \, [\overline{D}_{link}] \\
&=& - \Big( \sum_{c} w_D(c) \, [\overline{{}_cD}] - w(D) \, [\overline{D}_{link}] \Big)
	=- \mathcal{L} (D) = - \mathcal{L} (K).
\end{eqnarray*} 
This complete the proof.
\end{proof}

\begin{theorem}\label{th3.4}
$\mathcal{F} ( K )$ is a Vassiliev invariant of order one of virtual knotoids. 
\end{theorem}

\begin{proof}
To prove the statement we need to show that 
\begin{itemize}
\item[(1)] $\mathcal{F}^{(2)}$ vanishes on any singular virtual knotoid with two singular crossings, and 
\item[(2)] There is a singular virtual knotoid with one singular crossing which has nontrivial $\mathcal{F}^{(1)}$.
\end{itemize}	
	
We demonstrate that both statements hold.	
	
(1) Let $D$ be a virtual knotoid diagram and $c_1, c_2$ be two classical crossings. Without loss of generality, we can assume that $w_D(c_1) = w_D(c_2) = 1$. We use notation $D_{-+}$ (resp. $D_{+-}$) to denote the virtual knotoid diagram obtained from $D$ by switching $c_1$ (resp.  $c_2)$), and use $D_{--}$ to denote the diagram obtained from $D$ by switching both $c_1$ and $c_2$. Moreover, $D_{++}= D$.
It is clear that 
$$
[\overline{D_{++}}] = [\overline{D_{+-}}] = [\overline{D_{-+}}] = [\overline{D_{--}}]  
$$
and 
$$
w(D_{++})-w(D_{+-})-w(D_{-+})+w(D_{--})=0. 
$$
Therefore, 
\begin{eqnarray*}
\mathcal{F}^{(2)}(D)&=&\mathcal{F}(D_{++}) - \mathcal{F}(D_{+-}) - \mathcal{F}(D_{-+}) + \mathcal{F}(D_{--}) \\ 
&=& \left(\sum_c w_{D_{++}} (c) \, [\overline{D_{++_c}}] - w(D_{++}) \, [\overline{D_{++}}] \right) 
  - \left( \sum_c w_{D_{+-}} (c) \, [\overline{D_{+-_c}}] - w(D_{+-}) \, [\overline{D_{+-}}] \right) \\ 
& & -\left(\sum_c w_{D_{-+}}(c) \, [\overline{D_{-+_c}}] - w(D_{-+}) \, [\overline{D_{-+}}] \right)  
  +\left( \sum_c w_{D_{--}}(c) \, [\overline{D_{--_c}}] - w(D_{--}) \, [\overline{D_{--}}] \right) \\ 
&=& \sum_c  w_{D_{++}} (c) \, [\overline{D_{++_c}}] - \sum_c w_{D_{+-}} (c) \, [\overline{D_{+-_c}}]  
- \sum_c w_{D_{-+}}(c) \, [\overline{D_{-+_c}}] + \sum_c w_{D_{--}}(c) \, [\overline{D_{--_c}}] \\ 
&=& \left(  w_{D_{++}} (c_1) \, [\overline{D_{++_{c_1}}}] +  w_{D_{++}} (c_2) \, [\overline{D_{++_{c_2}}}] + \sum_{c\neq c_1,c_2 }w_{D_{++}} (c) \, [\overline{D_{++_{c}}}] \right)  \\
& & - \left(  w_{D_{+-}} (c_1) \, [\overline{D_{+-_{c_1}}}] +   w_{D_{+-}} (c_2) \, [\overline{D_{+-_{c_2}}}] + \sum_{c\neq c_1,c_2 }w_{D_{+-}} (c) \, [\overline{D_{+-_{c}}}]  \right) \\
& & - \left(  w_{D_{-+}}(c_1) \, [ \overline{D_{-+_{c_1}}} ] +  w_{D_{-+}}(c_2) \,  [\overline{D_{-+_{c_2}}}] + \sum_{c\neq c_1,c_2 }w_{D_{-+}} (c) \, [\overline{D_{-+_{c}}}]  \right)  \\
& & + \left(  w_{D_{--}}(c_1) \, [\overline{D_{--_{c_1}}}] + w_{D_{--}}(c_2) \, [\overline{D_{--_{c_2}}}] + \sum_{c\neq c_1,c_2 }w_{D_{--}} (c) \, [\overline{D_{--_{c}}}]  \right)  \\ 
&=& \left(  w_{D_{++}} (c_1) \, [\overline{D_{++_{c_1}}}] +  w_{D_{++}} (c_2) \, [\overline{D_{++_{c_2}}}] \right)  
 - \left(  w_{D_{+-}} (c_1) \, [\overline{D_{+-_{c_1}}}] +   w_{D_{+-}} (c_2) \, [\overline{D_{+-_{c_2}}}] \right) \\
& & - \left(  w_{D_{-+}}(c_1) \,  [\overline{D_{-+_{c_1}}}]  +  w_{D_{-+}}(c_2) \,  [\overline{D_{-+_{c_2}}}] \right)
+ \left(  w_{D_{--}}(c_1) \, [\overline{D_{--_{c_1}}}] + w_{D_{--}}(c_2) \, [\overline{D_{--_{c_2}}}] \right)  \\ 
&=& \left([\overline{D_{++_{c_1}}}] + [\overline{D_{++_{c_2}}}] \right) - \left([\overline{D_{+-_{c_1}}}] - [\overline{D_{+-_{c_2}}}] \right)  
 -\left(-[\overline{D_{-+_{c_1}}}] + [\overline{D_{-+_{c_2}}}] \right) + \left(-[\overline{D_{--_{c_1}}}] - [\overline{D_{--_{c_2}}}] \right) \\ 
&=& \left([\overline{D_{++_{c_1}}}] - [\overline{D_{+-_{c_1}}}] \right) + \left( [\overline{D_{++_{c_2}}}] - [\overline{D_{-+_{c_2}}}] \right)   
 + \left( [\overline{D_{+-_{c_2}}}] - [\overline{D_{--_{c_2}}}] \right) + \left([\overline{D_{-+_{c_1}}}] - [\overline{D_{--_{c_1}}}] \right) = 0. 
\end{eqnarray*}

(2) Now consider the singular virtual knotoid diagram $D$ in Fig.~\ref{fig24}.  We use $[D_{+}]$ to denote the virtual knotoid diagram obtained from $D$ by replacing singular crossing $c$ with positive crossing, and use $[D_{-}]$ to denote the virtual knotoid diagram obtained from $D$ by replacing singular crossing $c$ with negative crossing, see Fig.~\ref{fig24}. As well as before, the trivial knotoid is denoted by $[D_0]$. 
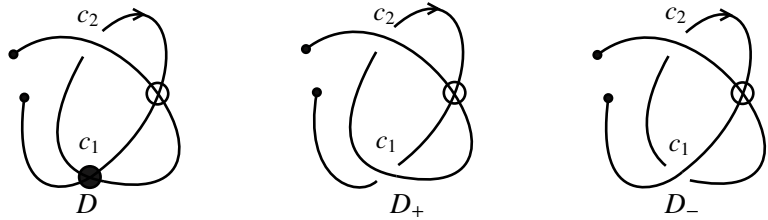
\begin{figure}[htbp]
\begin{center}
\tikzset{every picture/.style={line width=1.0pt}}  
\begin{tikzpicture}[x=0.75pt,y=0.75pt,yscale=-1.0,xscale=1.0]
\draw  [fill={rgb, 255:red, 24; green, 23; blue, 23 }  ,fill opacity=1 ] (213.82,152.99) .. controls (213.82,149.99) and (216.25,147.57) .. (219.24,147.57) .. controls (222.23,147.57) and (224.66,149.99) .. (224.66,152.99) .. controls (224.66,155.98) and (222.23,158.4) .. (219.24,158.4) .. controls (216.25,158.4) and (213.82,155.98) .. (213.82,152.99) -- cycle ;
\draw    (180.72,91.07) .. controls (234.6,49.26) and (311.46,172.46) .. (226.34,155.18) ;
\draw    (215.84,92.2) .. controls (205.27,110.77) and (188.56,148.04) .. (226.34,155.18) ;
\draw    (225.61,79.31) .. controls (261.97,48.01) and (275.38,110.76) .. (223.52,150.13) ;
\draw    (223.52,150.13) .. controls (213.1,159.76) and (178.64,170.69) .. (186.22,112.88) ;
\draw   (240.98,68.16) -- (246.16,70.9) -- (240.98,73.65) ;
\draw    (328.23,88.41) .. controls (382.11,46.6) and (458.97,169.8) .. (373.85,152.52) ;
\draw    (363.66,89.83) .. controls (353.09,108.4) and (336.07,145.38) .. (373.85,152.52) ;
\draw    (372.76,77.28) .. controls (409.8,44.18) and (426.63,107.6) .. (374.78,146.97) ;
\draw    (364,154.98) .. controls (353.58,164.61) and (325.99,153.64) .. (333.74,110.22) ;
\draw    (475.23,91.35) .. controls (529.1,49.54) and (606.99,171.92) .. (521.87,154.64) ;
\draw    (511.09,91.92) .. controls (500.52,110.49) and (493.58,133.73) .. (509.74,147.32) ;
\draw    (519.85,80.07) .. controls (557.22,46.97) and (569.88,111.04) .. (518.03,150.41) ;
\draw    (518.03,150.41) .. controls (507.61,160.04) and (473.14,170.97) .. (480.73,113.17) ;
\draw  [fill={rgb, 255:red, 17; green, 16; blue, 16 }  ,fill opacity=1 ] (331.87,110.24) .. controls (331.88,111.27) and (332.72,112.1) .. (333.75,112.09) .. controls (334.79,112.08) and (335.61,111.24) .. (335.61,110.21) .. controls (335.6,109.18) and (334.75,108.35) .. (333.72,108.36) .. controls (332.69,108.37) and (331.86,109.21) .. (331.87,110.24) -- cycle ;
\draw  [fill={rgb, 255:red, 17; green, 16; blue, 16 }  ,fill opacity=1 ] (473.36,91.37) .. controls (473.36,92.4) and (474.21,93.23) .. (475.24,93.22) .. controls (476.27,93.21) and (477.1,92.37) .. (477.09,91.34) .. controls (477.09,90.31) and (476.24,89.48) .. (475.21,89.49) .. controls (474.18,89.49) and (473.35,90.34) .. (473.36,91.37) -- cycle ;
\draw  [fill={rgb, 255:red, 17; green, 16; blue, 16 }  ,fill opacity=1 ] (326.37,88.42) .. controls (326.37,89.45) and (327.22,90.28) .. (328.25,90.27) .. controls (329.28,90.27) and (330.11,89.42) .. (330.1,88.39) .. controls (330.09,87.36) and (329.25,86.53) .. (328.22,86.54) .. controls (327.19,86.55) and (326.36,87.39) .. (326.37,88.42) -- cycle ;
\draw  [fill={rgb, 255:red, 17; green, 16; blue, 16 }  ,fill opacity=1 ] (478.86,113.18) .. controls (478.87,114.21) and (479.71,115.04) .. (480.74,115.03) .. controls (481.78,115.02) and (482.61,114.18) .. (482.6,113.15) .. controls (482.59,112.12) and (481.75,111.29) .. (480.71,111.3) .. controls (479.68,111.31) and (478.85,112.15) .. (478.86,113.18) -- cycle ;
\draw  [fill={rgb, 255:red, 17; green, 16; blue, 16 }  ,fill opacity=1 ] (184.36,112.9) .. controls (184.36,113.93) and (185.21,114.76) .. (186.24,114.75) .. controls (187.27,114.74) and (188.1,113.9) .. (188.09,112.87) .. controls (188.08,111.84) and (187.24,111.01) .. (186.21,111.02) .. controls (185.18,111.03) and (184.35,111.87) .. (184.36,112.9) -- cycle ;
\draw  [fill={rgb, 255:red, 17; green, 16; blue, 16 }  ,fill opacity=1 ] (178.85,91.08) .. controls (178.86,92.11) and (179.7,92.94) .. (180.74,92.93) .. controls (181.77,92.93) and (182.6,92.08) .. (182.59,91.05) .. controls (182.58,90.02) and (181.74,89.19) .. (180.71,89.2) .. controls (179.67,89.21) and (178.84,90.05) .. (178.85,91.08) -- cycle ;
\draw   (247.94,110.89) .. controls (247.94,107.89) and (250.36,105.47) .. (253.36,105.47) .. controls (256.35,105.47) and (258.78,107.89) .. (258.78,110.89) .. controls (258.78,113.88) and (256.35,116.31) .. (253.36,116.31) .. controls (250.36,116.31) and (247.94,113.88) .. (247.94,110.89) -- cycle ;
\draw   (397.77,110.45) .. controls (397.77,107.45) and (400.2,105.03) .. (403.19,105.03) .. controls (406.18,105.03) and (408.61,107.45) .. (408.61,110.45) .. controls (408.61,113.44) and (406.18,115.86) .. (403.19,115.86) .. controls (400.2,115.86) and (397.77,113.44) .. (397.77,110.45) -- cycle ;
\draw   (543.03,110.66) .. controls (543.03,107.66) and (545.45,105.24) .. (548.44,105.24) .. controls (551.44,105.24) and (553.86,107.66) .. (553.86,110.66) .. controls (553.86,113.65) and (551.44,116.07) .. (548.44,116.07) .. controls (545.45,116.07) and (543.03,113.65) .. (543.03,110.66) -- cycle ;
\draw   (389.56,65.3) -- (394.73,68.04) -- (389.56,70.79) ;
\draw   (537.11,68.41) -- (542.28,71.16) -- (537.11,73.9) ;
\draw (212.05,131.28) node [anchor=north west][inner sep=0.75pt]  [font=\small]  {$c_{1}$};
\draw (211.23,67.02) node [anchor=north west][inner sep=0.75pt]  [font=\small]  {$c_{2}$};
\draw (360.7,66.05) node [anchor=north west][inner sep=0.75pt]  [font=\small]  {$c_{2}$};
\draw (507.72,65.81) node [anchor=north west][inner sep=0.75pt]  [font=\small]  {$c_{2}$};
\draw (361.74,130.54) node [anchor=north west][inner sep=0.75pt]  [font=\small]  {$c_{1}$};
\draw (509.63,132.53) node [anchor=north west][inner sep=0.75pt]  [font=\small]  {$c_{1}$};
\draw (368.68,159.97) node [anchor=north west][inner sep=0.75pt]  [font=\normalsize]  {$D_{+}$};
\draw (507.27,159.93) node [anchor=north west][inner sep=0.75pt]  [font=\normalsize]  {$D_{-}$};
\draw (210.54,159.73) node [anchor=north west][inner sep=0.75pt]  [font=\normalsize]  {$D$};
\end{tikzpicture}
\caption{Singular virtual knotoid diagram $D$, virtual knotoid diagrams $D_{+}$ and $D_{-}$.\label{fig24}}
\end{center}
\end{figure}
\begin{figure}[htbp]
\begin{center}
\tikzset{every picture/.style={line width=1.0pt}}  
\begin{tikzpicture}[x=0.75pt,y=0.75pt,yscale=-1.0,xscale=1.0]
\draw    (252.04,304.22) -- (272.81,304.08) ;
\draw   (268.94,301.13) -- (273.79,304) -- (268.99,307.03) ; 
\draw    (252.09,375.89) -- (272.85,375.75) ;
\draw   (268.98,372.8) -- (273.83,375.67) -- (269.04,378.7) ;
\draw    (357.22,304.39) -- (377.98,304.25) ;
\draw   (374.11,301.3) -- (378.96,304.17) -- (374.17,307.2) ;
\draw    (443.17,304.58) .. controls (423.17,310.43) and (417.02,296.89) .. (396.98,303.09) ; 
\draw  [fill={rgb, 255:red, 17; green, 16; blue, 16 }  ,fill opacity=1 ] (395.11,303.1) .. controls (395.12,304.13) and (395.96,304.96) .. (397,304.96) .. controls (398.03,304.95) and (398.86,304.1) .. (398.85,303.07) .. controls (398.84,302.04) and (398,301.22) .. (396.97,301.22) .. controls (395.93,301.23) and (395.1,302.07) .. (395.11,303.1) -- cycle ; 
\draw  [fill={rgb, 255:red, 17; green, 16; blue, 16 }  ,fill opacity=1 ] (441.3,304.6) .. controls (441.31,305.63) and (442.15,306.45) .. (443.18,306.45) .. controls (444.22,306.44) and (445.05,305.6) .. (445.04,304.57) .. controls (445.03,303.53) and (444.19,302.71) .. (443.15,302.71) .. controls (442.12,302.72) and (441.29,303.56) .. (441.3,304.6) -- cycle ;
\draw    (356.42,375.86) -- (377.18,375.71) ;
\draw   (373.31,372.77) -- (378.16,375.64) -- (373.37,378.66) ;
\draw    (442.37,376.05) .. controls (422.37,381.89) and (416.22,368.35) .. (396.18,374.56) ;
\draw  [fill={rgb, 255:red, 17; green, 16; blue, 16 }  ,fill opacity=1 ] (394.31,374.57) .. controls (394.32,375.6) and (395.16,376.43) .. (396.2,376.42) .. controls (397.23,376.41) and (398.06,375.57) .. (398.05,374.54) .. controls (398.04,373.51) and (397.2,372.68) .. (396.17,372.69) .. controls (395.13,372.7) and (394.3,373.54) .. (394.31,374.57) -- cycle ;
\draw  [fill={rgb, 255:red, 17; green, 16; blue, 16 }  ,fill opacity=1 ] (440.5,376.06) .. controls (440.51,377.09) and (441.35,377.92) .. (442.38,377.91) .. controls (443.42,377.9) and (444.25,377.06) .. (444.24,376.03) .. controls (444.23,375) and (443.39,374.17) .. (442.35,374.18) .. controls (441.32,374.19) and (440.49,375.03) .. (440.5,376.06) -- cycle ;
\draw    (292.29,299.06) .. controls (309.42,307.96) and (320.85,329.31) .. (333.77,316.83) ;
\draw    (308.09,293.92) .. controls (302.79,301.48) and (301.79,309.33) .. (292.79,319.02) ;
\draw    (308.09,293.92) .. controls (329.51,265.73) and (350.9,297.28) .. (333.77,316.83) ;
\draw    (174.86,288.46) .. controls (224.52,259.27) and (260.62,345.8) .. (221.18,332.3) ;
\draw    (207.92,280.85) .. controls (188.85,298.39) and (219.18,316.75) .. (201.42,330.39) ;
\draw    (234.38,279.08) .. controls (240.28,288) and (240.09,301.01) .. (230.33,305.2) ;
\draw    (201.42,330.39) .. controls (193.68,336.52) and (173.2,329.53) .. (178.95,301.85) ;
\draw  [fill={rgb, 255:red, 17; green, 16; blue, 16 }  ,fill opacity=1 ] (177.08,304.5) .. controls (177.09,305.53) and (177.93,306.36) .. (178.96,306.35) .. controls (180,306.34) and (180.82,305.5) .. (180.82,304.47) .. controls (180.81,303.44) and (179.96,302.61) .. (178.93,302.62) .. controls (177.9,302.63) and (177.07,303.47) .. (177.08,304.5) -- cycle ;
\draw  [fill={rgb, 255:red, 17; green, 16; blue, 16 }  ,fill opacity=1 ] (172.99,288.85) .. controls (173,289.88) and (173.84,290.71) .. (174.88,290.7) .. controls (175.91,290.69) and (176.74,289.85) .. (176.73,288.82) .. controls (176.72,287.79) and (175.88,286.96) .. (174.85,286.97) .. controls (173.81,286.98) and (172.98,287.82) .. (172.99,288.85) -- cycle ;
\draw   (226.06,304.06) .. controls (226.06,301.07) and (228.48,298.64) .. (231.48,298.64) .. controls (234.47,298.64) and (236.89,301.07) .. (236.89,304.06) .. controls (236.89,307.05) and (234.47,309.48) .. (231.48,309.48) .. controls (228.48,309.48) and (226.06,307.05) .. (226.06,304.06) -- cycle ;
\draw    (230.33,305.2) .. controls (219.13,310.4) and (205.59,326.69) .. (221.18,332.3) ;
\draw    (207.92,280.85) .. controls (220.08,269.88) and (230.12,272.46) .. (234.38,279.08) ;
\draw  [fill={rgb, 255:red, 17; green, 16; blue, 16 }  ,fill opacity=1 ] (290.42,299.08) .. controls (290.43,300.11) and (291.27,300.94) .. (292.31,300.93) .. controls (293.34,300.92) and (294.17,300.08) .. (294.16,299.05) .. controls (294.15,298.02) and (293.31,297.19) .. (292.28,297.2) .. controls (291.24,297.2) and (290.41,298.05) .. (290.42,299.08) -- cycle ;
\draw  [fill={rgb, 255:red, 17; green, 16; blue, 16 }  ,fill opacity=1 ] (290.92,318.9) .. controls (290.93,319.93) and (291.77,320.76) .. (292.81,320.76) .. controls (293.84,320.75) and (294.67,319.9) .. (294.66,318.87) .. controls (294.65,317.84) and (293.81,317.02) .. (292.78,317.02) .. controls (291.74,317.03) and (290.91,317.87) .. (290.92,318.9) -- cycle ;
\draw    (292.29,371.06) .. controls (309.42,379.96) and (320.85,401.31) .. (333.77,388.83) ;
\draw    (308.09,365.92) .. controls (302.79,373.48) and (301.79,381.33) .. (292.79,391.02) ;
\draw    (308.09,365.92) .. controls (329.51,337.73) and (350.9,369.28) .. (333.77,388.83) ;
\draw  [fill={rgb, 255:red, 17; green, 16; blue, 16 }  ,fill opacity=1 ] (290.42,371.08) .. controls (290.43,372.11) and (291.27,372.94) .. (292.31,372.93) .. controls (293.34,372.92) and (294.17,372.08) .. (294.16,371.05) .. controls (294.15,370.02) and (293.31,369.19) .. (292.28,369.2) .. controls (291.24,369.2) and (290.41,370.05) .. (290.42,371.08) -- cycle ;
\draw  [fill={rgb, 255:red, 17; green, 16; blue, 16 }  ,fill opacity=1 ] (290.92,390.9) .. controls (290.93,391.93) and (291.77,392.76) .. (292.81,392.76) .. controls (293.84,392.75) and (294.67,391.9) .. (294.66,390.87) .. controls (294.65,389.84) and (293.81,389.02) .. (292.78,389.02) .. controls (291.74,389.03) and (290.91,389.87) .. (290.92,390.9) -- cycle ;
\draw    (172.86,397.52) .. controls (222.52,426.71) and (264.34,343.28) .. (219.18,353.68) ;
\draw    (205.92,405.13) .. controls (186.85,387.59) and (217.18,369.23) .. (199.42,355.6) ;
\draw    (232.38,406.9) .. controls (238.28,397.98) and (238.09,384.97) .. (228.33,380.78) ;
\draw    (199.42,355.6) .. controls (191.68,349.46) and (171.2,356.45) .. (176.95,384.13) ;
\draw  [fill={rgb, 255:red, 17; green, 16; blue, 16 }  ,fill opacity=1 ] (175.08,381.48) .. controls (175.09,380.45) and (175.93,379.62) .. (176.96,379.63) .. controls (178,379.64) and (178.82,380.48) .. (178.82,381.51) .. controls (178.81,382.54) and (177.96,383.37) .. (176.93,383.36) .. controls (175.9,383.35) and (175.07,382.51) .. (175.08,381.48) -- cycle ;
\draw  [fill={rgb, 255:red, 17; green, 16; blue, 16 }  ,fill opacity=1 ] (170.99,397.13) .. controls (171,396.1) and (171.84,395.27) .. (172.88,395.28) .. controls (173.91,395.29) and (174.74,396.13) .. (174.73,397.16) .. controls (174.72,398.19) and (173.88,399.02) .. (172.85,399.01) .. controls (171.81,399.01) and (170.98,398.16) .. (170.99,397.13) -- cycle ;
\draw   (226.06,381.93) .. controls (226.06,384.92) and (228.48,387.34) .. (231.48,387.34) .. controls (234.47,387.34) and (236.89,384.92) .. (236.89,381.93) .. controls (236.89,378.93) and (234.47,376.51) .. (231.48,376.51) .. controls (228.48,376.51) and (226.06,378.93) .. (226.06,381.93) -- cycle ;
\draw    (228.33,380.78) .. controls (217.13,375.58) and (203.59,359.29) .. (219.18,353.68) ;
\draw    (205.92,405.13) .. controls (218.08,416.11) and (228.12,413.52) .. (232.38,406.9) ;
\draw (144,354.6) node [anchor=north west][inner sep=0.75pt]  [font=\normalsize]  {$\overline{D}_{{+}_{c_{2}}}$};
\draw (251.09,286.00) node [anchor=north west][inner sep=0.75pt]  [font=\footnotesize]  {$f\Omega_{1}^{v}$};
\draw (251.54,358.4) node [anchor=north west][inner sep=0.75pt]  [font=\footnotesize]  {$f\Omega_{1}^{v}$};
\draw (356.45,287.00) node [anchor=north west][inner sep=0.75pt]  [font=\footnotesize]  {$f\Omega_{1}$};
\draw (356.45,359.4) node [anchor=north west][inner sep=0.75pt]  [font=\footnotesize]  {$f\Omega_{1}$};
\draw (142.86,285.1) node [anchor=north west][inner sep=0.75pt]  [font=\normalsize]  {$\overline{D}_{{+}_{c_{1}}}$};
\end{tikzpicture}
\caption{Flat diagrams $\overline{D_{+_{c_1}}}$ and $\overline{D_{+_{c_2}}}$.} \label{fig25} 
\end{center}
\end{figure}

\begin{figure}[htbp]
\begin{center}
\tikzset{every picture/.style={line width=1.0pt}}  
\begin{tikzpicture}[x=0.75pt,y=0.75pt,yscale=-1.0,xscale=1.0]
\draw    (252.04,304.22) -- (272.81,304.08) ;
\draw   (268.94,301.13) -- (273.79,304) -- (268.99,307.03) ; 
\draw    (252.09,375.89) -- (272.85,375.75) ;
\draw   (268.98,372.8) -- (273.83,375.67) -- (269.04,378.7) ;
\draw    (357.22,304.39) -- (377.98,304.25) ;
\draw   (374.11,301.3) -- (378.96,304.17) -- (374.17,307.2) ;
\draw    (443.17,304.58) .. controls (423.17,310.43) and (417.02,296.89) .. (396.98,303.09) ; 
\draw  [fill={rgb, 255:red, 17; green, 16; blue, 16 }  ,fill opacity=1 ] (395.11,303.1) .. controls (395.12,304.13) and (395.96,304.96) .. (397,304.96) .. controls (398.03,304.95) and (398.86,304.1) .. (398.85,303.07) .. controls (398.84,302.04) and (398,301.22) .. (396.97,301.22) .. controls (395.93,301.23) and (395.1,302.07) .. (395.11,303.1) -- cycle ; 
\draw  [fill={rgb, 255:red, 17; green, 16; blue, 16 }  ,fill opacity=1 ] (441.3,304.6) .. controls (441.31,305.63) and (442.15,306.45) .. (443.18,306.45) .. controls (444.22,306.44) and (445.05,305.6) .. (445.04,304.57) .. controls (445.03,303.53) and (444.19,302.71) .. (443.15,302.71) .. controls (442.12,302.72) and (441.29,303.56) .. (441.3,304.6) -- cycle ;
\draw    (356.42,375.86) -- (377.18,375.71) ;
\draw   (373.31,372.77) -- (378.16,375.64) -- (373.37,378.66) ;
\draw    (442.37,376.05) .. controls (422.37,381.89) and (416.22,368.35) .. (396.18,374.56) ;
\draw  [fill={rgb, 255:red, 17; green, 16; blue, 16 }  ,fill opacity=1 ] (394.31,374.57) .. controls (394.32,375.6) and (395.16,376.43) .. (396.2,376.42) .. controls (397.23,376.41) and (398.06,375.57) .. (398.05,374.54) .. controls (398.04,373.51) and (397.2,372.68) .. (396.17,372.69) .. controls (395.13,372.7) and (394.3,373.54) .. (394.31,374.57) -- cycle ;
\draw  [fill={rgb, 255:red, 17; green, 16; blue, 16 }  ,fill opacity=1 ] (440.5,376.06) .. controls (440.51,377.09) and (441.35,377.92) .. (442.38,377.91) .. controls (443.42,377.9) and (444.25,377.06) .. (444.24,376.03) .. controls (444.23,375) and (443.39,374.17) .. (442.35,374.18) .. controls (441.32,374.19) and (440.49,375.03) .. (440.5,376.06) -- cycle ;
\draw    (292.29,299.06) .. controls (309.42,307.96) and (320.85,329.31) .. (333.77,316.83) ;
\draw    (308.09,293.92) .. controls (302.79,301.48) and (301.79,309.33) .. (292.79,319.02) ;
\draw    (308.09,293.92) .. controls (329.51,265.73) and (350.9,297.28) .. (333.77,316.83) ;
\draw    (174.86,288.46) .. controls (224.52,259.27) and (260.62,345.8) .. (221.18,332.3) ;
\draw    (207.92,280.85) .. controls (188.85,298.39) and (219.18,316.75) .. (201.42,330.39) ;
\draw    (234.38,279.08) .. controls (240.28,288) and (240.09,301.01) .. (230.33,305.2) ;
\draw    (201.42,330.39) .. controls (193.68,336.52) and (173.2,329.53) .. (178.95,301.85) ;
\draw  [fill={rgb, 255:red, 17; green, 16; blue, 16 }  ,fill opacity=1 ] (177.08,304.5) .. controls (177.09,305.53) and (177.93,306.36) .. (178.96,306.35) .. controls (180,306.34) and (180.82,305.5) .. (180.82,304.47) .. controls (180.81,303.44) and (179.96,302.61) .. (178.93,302.62) .. controls (177.9,302.63) and (177.07,303.47) .. (177.08,304.5) -- cycle ;
\draw  [fill={rgb, 255:red, 17; green, 16; blue, 16 }  ,fill opacity=1 ] (172.99,288.85) .. controls (173,289.88) and (173.84,290.71) .. (174.88,290.7) .. controls (175.91,290.69) and (176.74,289.85) .. (176.73,288.82) .. controls (176.72,287.79) and (175.88,286.96) .. (174.85,286.97) .. controls (173.81,286.98) and (172.98,287.82) .. (172.99,288.85) -- cycle ;
\draw   (226.06,304.06) .. controls (226.06,301.07) and (228.48,298.64) .. (231.48,298.64) .. controls (234.47,298.64) and (236.89,301.07) .. (236.89,304.06) .. controls (236.89,307.05) and (234.47,309.48) .. (231.48,309.48) .. controls (228.48,309.48) and (226.06,307.05) .. (226.06,304.06) -- cycle ;
\draw    (230.33,305.2) .. controls (219.13,310.4) and (205.59,326.69) .. (221.18,332.3) ;
\draw    (207.92,280.85) .. controls (220.08,269.88) and (230.12,272.46) .. (234.38,279.08) ;
\draw  [fill={rgb, 255:red, 17; green, 16; blue, 16 }  ,fill opacity=1 ] (290.42,299.08) .. controls (290.43,300.11) and (291.27,300.94) .. (292.31,300.93) .. controls (293.34,300.92) and (294.17,300.08) .. (294.16,299.05) .. controls (294.15,298.02) and (293.31,297.19) .. (292.28,297.2) .. controls (291.24,297.2) and (290.41,298.05) .. (290.42,299.08) -- cycle ;
\draw  [fill={rgb, 255:red, 17; green, 16; blue, 16 }  ,fill opacity=1 ] (290.92,318.9) .. controls (290.93,319.93) and (291.77,320.76) .. (292.81,320.76) .. controls (293.84,320.75) and (294.67,319.9) .. (294.66,318.87) .. controls (294.65,317.84) and (293.81,317.02) .. (292.78,317.02) .. controls (291.74,317.03) and (290.91,317.87) .. (290.92,318.9) -- cycle ;
\draw    (292.29,371.06) .. controls (309.42,379.96) and (320.85,401.31) .. (333.77,388.83) ;
\draw    (308.09,365.92) .. controls (302.79,373.48) and (301.79,381.33) .. (292.79,391.02) ;
\draw    (308.09,365.92) .. controls (329.51,337.73) and (350.9,369.28) .. (333.77,388.83) ;
\draw  [fill={rgb, 255:red, 17; green, 16; blue, 16 }  ,fill opacity=1 ] (290.42,371.08) .. controls (290.43,372.11) and (291.27,372.94) .. (292.31,372.93) .. controls (293.34,372.92) and (294.17,372.08) .. (294.16,371.05) .. controls (294.15,370.02) and (293.31,369.19) .. (292.28,369.2) .. controls (291.24,369.2) and (290.41,370.05) .. (290.42,371.08) -- cycle ;
\draw  [fill={rgb, 255:red, 17; green, 16; blue, 16 }  ,fill opacity=1 ] (290.92,390.9) .. controls (290.93,391.93) and (291.77,392.76) .. (292.81,392.76) .. controls (293.84,392.75) and (294.67,391.9) .. (294.66,390.87) .. controls (294.65,389.84) and (293.81,389.02) .. (292.78,389.02) .. controls (291.74,389.03) and (290.91,389.87) .. (290.92,390.9) -- cycle ;
\draw    (172.86,397.52) .. controls (222.52,426.71) and (264.34,343.28) .. (219.18,353.68) ;
\draw    (205.92,405.13) .. controls (186.85,387.59) and (217.18,369.23) .. (199.42,355.6) ;
\draw    (232.38,406.9) .. controls (238.28,397.98) and (238.09,384.97) .. (228.33,380.78) ;
\draw    (199.42,355.6) .. controls (191.68,349.46) and (171.2,356.45) .. (176.95,384.13) ;
\draw  [fill={rgb, 255:red, 17; green, 16; blue, 16 }  ,fill opacity=1 ] (175.08,381.48) .. controls (175.09,380.45) and (175.93,379.62) .. (176.96,379.63) .. controls (178,379.64) and (178.82,380.48) .. (178.82,381.51) .. controls (178.81,382.54) and (177.96,383.37) .. (176.93,383.36) .. controls (175.9,383.35) and (175.07,382.51) .. (175.08,381.48) -- cycle ;
\draw  [fill={rgb, 255:red, 17; green, 16; blue, 16 }  ,fill opacity=1 ] (170.99,397.13) .. controls (171,396.1) and (171.84,395.27) .. (172.88,395.28) .. controls (173.91,395.29) and (174.74,396.13) .. (174.73,397.16) .. controls (174.72,398.19) and (173.88,399.02) .. (172.85,399.01) .. controls (171.81,399.01) and (170.98,398.16) .. (170.99,397.13) -- cycle ;
\draw   (226.06,381.93) .. controls (226.06,384.92) and (228.48,387.34) .. (231.48,387.34) .. controls (234.47,387.34) and (236.89,384.92) .. (236.89,381.93) .. controls (236.89,378.93) and (234.47,376.51) .. (231.48,376.51) .. controls (228.48,376.51) and (226.06,378.93) .. (226.06,381.93) -- cycle ;
\draw    (228.33,380.78) .. controls (217.13,375.58) and (203.59,359.29) .. (219.18,353.68) ;
\draw    (205.92,405.13) .. controls (218.08,416.11) and (228.12,413.52) .. (232.38,406.9) ;
\draw (144,354.6) node [anchor=north west][inner sep=0.75pt]  [font=\normalsize]  {$\overline{D}_{{-}_{c_{2}}}$};
\draw (251.09,286.00) node [anchor=north west][inner sep=0.75pt]  [font=\footnotesize]  {$f\Omega_{1}^{v}$};
\draw (251.54,358.4) node [anchor=north west][inner sep=0.75pt]  [font=\footnotesize]  {$f\Omega_{1}^{v}$};
\draw (356.45,287.00) node [anchor=north west][inner sep=0.75pt]  [font=\footnotesize]  {$f\Omega_{1}$};
\draw (356.45,359.4) node [anchor=north west][inner sep=0.75pt]  [font=\footnotesize]  {$f\Omega_{1}$};
\draw (142.86,285.1) node [anchor=north west][inner sep=0.75pt]  [font=\normalsize]  {$\overline{D}_{{-}_{c_{1}}}$};
\end{tikzpicture}
\caption{Flat diagrams $\overline{D_{{-}_{c_1}}}$ and $\overline{D_{{-}_{c_2}}}$.} \label{fig26} 
\end{center}
\end{figure}

By Fig.~\ref{fig25} and $w_{D_+} (c_1) = w_{D_+} (c_2) = 1$, we get  
$$ 
\mathcal{F}(D_+) = w_{D_+}(c_{1}) \, [\overline{D_{{+}_{c_1}}}] + w_{D_+}(c_{2})  \, [\overline{D_{{+}_{c_2}}}]  
- w(D_{+}) \, [\overline{D_{+}}]   
= 2 \, [D_{0}] - 2 \, [\overline{D_+}].
$$ 
By Fig.~\ref{fig26} and $w_{D_-} (c_1)=-1$, $w_{D_-}(c_2) = 1$, we get 
$$ 
\mathcal{F}(D_-)  =  w_{D_-}(c_{1}) \, [\overline{D_{{-}_{c_1}}}] + w_{D_-}(c_{2}) \, [\overline{D_{{-}_{c_2}}} ]
- w(D_-) \, [\overline{D_-}] 
 =  [D_0] - [D_0] - 0 \cdot \, [\overline{D_-}]  = 0.
$$ 
Then,
$$ 
\mathcal F^{(1)}(D) = \mathcal{F}(D_+) - \mathcal{F}(D_-) = 2 \, [D_0] - 2 \, [\overline{D_+}]. 
$$
By \cite[Fig.~19]{GK},  $[\overline{D_+}] \neq [D_0]$. Hence, $\mathcal{F}^{(1)}(D) \neq 0$. Therefore, $\mathcal{F} ( K )$ is a Vassiliev invariant of order one. 
\end{proof}

Notice, that the same property holds for the 1-smoothing invariant $\mathcal L(K)$.

\begin{theorem} [\cite{Pet}] \label{th3.5} 
$\mathcal{L}(K)$ is a Vassiliev invariant of order one of virtual knotoids. 
\end{theorem}

\section{The gluing and 0-smoothing invariants} \label{section4}

\subsection{The gluing invariant} In~\cite{Hen}, Henrich presented a direct approach for a universal Vassiliev invariant of order one for virtual knots, which is called the ``gluing invariant'' and is denoted by $\mathcal{G}$. In~\cite{Pet}, Petit defined the gluing invariant for long virtual knots, that is naturally an invariant of virtual knotoids. 

\begin{definition} [\cite{Pet}] \label{def4.3} {\rm 
		Let $K$ be a virtual knotoid with diagram $D$. We denote the resulting diagram in Fig.~\ref{fig100} by $D^c_{\operatorname{glue}}$ and the flat equivalence class of flattening $\overline{D^c_{\operatorname{glue}}}$ by $[\overline{D^c_{\operatorname{glue}}}]$. We denote by $D^0_{\operatorname{sing}}$ the diagram obtained by creating a kink on $D$ (by some $\Omega_1$-move) and then creating a \textit{singular kink} by gluing two strands, see Fig.~\ref{fig101}. Denote by $[\overline{D^{0}_{\operatorname{sing}}}]$ the flat equivalence class of flattening $\overline{D^{0}_{\operatorname{sing}}}$ of $D^{0}_{\operatorname{sing}}$. Notice that the  flat equivalence class $[\overline{D^{0}_{\operatorname{sing}}}]$ is independent of where in the diagram the singular kink is  placed. Define $\mathcal{G}(D)$ by the following formula: 
		\begin{equation}
		\mathcal{G} (D) = \sum_{c} w_D (c) \, [\overline{D^c_{\operatorname{glue}}}] - w(D) \, [\overline{D^0_{\operatorname{sing}}}],  \label{eqn:defG}
		\end{equation}
		where the sum runs over all classical crossings of $D$, $w_D(c) = \operatorname{sgn}(c)$ and $\displaystyle w(D) = \sum_c w_D(c)$. 
	}
\end{definition}

\begin{figure}[!ht]
\begin{center}
\tikzset{every picture/.style={line width=1.0pt}}  
\begin{tikzpicture}[x=0.75pt,y=0.75pt,yscale=-1,xscale=1]
\draw    (477.08,568.99) -- (421.47,623.44) ;
\draw   (421.69,574.27) -- (420.99,568.92) -- (426.42,570) ;
\draw   (471.93,570.14) -- (477.31,568.87) -- (476.81,574.24) ;
\draw    (453.39,600.4) -- (478.8,624.6) ;
\draw    (194.32,568.59) -- (171.68,590.72) ;
\draw    (138.6,568.65) -- (194.43,622.98) ;
\draw   (139.21,573.87) -- (138.51,568.53) -- (143.93,569.6) ;
\draw   (189.17,569.74) -- (194.55,568.47) -- (194.05,573.84) ;
\draw    (161.64,600.64) -- (138.71,623.04) ;
\draw    (420.53,568.64) -- (444.61,591.98) ;
\draw    (338.32,568.63) -- (282.71,623.08) ;
\draw    (282.6,568.69) -- (338.43,623.02) ;
\draw   (283.22,573.91) -- (282.51,568.57) -- (287.94,569.64) ;
\draw   (333.17,569.78) -- (338.56,568.51) -- (338.06,573.88) ;
\draw  [fill={rgb, 255:red, 14; green, 13; blue, 13 }  ,fill opacity=1 ] (305.42,595.86) .. controls (305.42,593.16) and (307.7,590.97) .. (310.52,590.97) .. controls (313.34,590.97) and (315.62,593.16) .. (315.62,595.86) .. controls (315.62,598.55) and (313.34,600.74) .. (310.52,600.74) .. controls (307.7,600.74) and (305.42,598.55) .. (305.42,595.86) -- cycle ;
\draw    (395.67,593.89) -- (365.63,593.45) ;
\draw   (367.79,590.85) -- (364.94,593.39) -- (367.72,596) ;
\draw    (222.83,594.79) -- (252.88,594.35) ;
\draw   (250.72,591.75) -- (253.56,594.3) -- (250.79,596.9) ;
\draw (175.76,589.63) node [anchor=north west][inner sep=0.75pt]  [font=\large]  {$c$};
\draw (157.23,631.27) node [anchor=north west][inner sep=0.75pt]    {$D$};
\draw (439.6,631.27) node [anchor=north west][inner sep=0.75pt]    {$D$};
\draw (294.28,631.04) node [anchor=north west][inner sep=0.75pt]    {$D^c_{\operatorname{glue}}$};
\draw (458.76,589.63) node [anchor=north west][inner sep=0.75pt]  [font=\large]  {$c$};
\end{tikzpicture}
\caption{Gluing two strands in a classical crossing to form a singular crossing.} \label{fig100}
\end{center}
\end{figure}

\begin{figure}[!ht]
\begin{center}
\tikzset{every picture/.style={line width=1.0pt}}  
\begin{tikzpicture}[x=0.75pt,y=0.75pt,yscale=-1,xscale=1]
\draw    (336.04,178.26) .. controls (345.68,196.59) and (345.57,214.77) .. (335.97,232.89) ;
\draw    (286.93,178.08) -- (335.97,232.89) ;
\draw    (336.04,178.26) -- (314.94,201.05) ;
\draw    (173.77,177.48) -- (173.77,232.88) ;
\draw    (307.18,209.27) -- (286.37,231.73) ;
\draw    (222.77,206.58) -- (247.85,206.22) ;
\draw   (246.05,204.15) -- (248.42,206.18) -- (246.1,208.27) ;
\draw    (481.46,177.46) .. controls (491.1,195.79) and (490.98,213.97) .. (481.39,232.09) ;
\draw    (432.35,177.28) -- (481.39,232.09) ;
\draw    (481.46,177.46) -- (460.36,200.25) ;
\draw    (460.36,200.25) -- (431.79,230.93) ;
\draw    (378.19,205.78) -- (403.26,205.42) ;
\draw   (401.46,203.34) -- (403.84,205.38) -- (401.52,207.47) ;
\draw   (170.93,181.1) -- (173.74,177.14) -- (176.69,180.91) ;
\draw   (287.74,183.13) -- (287.12,178.37) -- (291.93,179.33) ;
\draw   (432.55,181.63) -- (431.92,176.88) -- (436.73,177.83) ;
\draw  [fill={rgb, 255:red, 14; green, 13; blue, 13 }  ,fill opacity=1 ] (451.77,204.68) .. controls (451.77,201.99) and (454.05,199.8) .. (456.87,199.8) .. controls (459.69,199.8) and (461.97,201.99) .. (461.97,204.68) .. controls (461.97,207.38) and (459.69,209.57) .. (456.87,209.57) .. controls (454.05,209.57) and (451.77,207.38) .. (451.77,204.68) -- cycle ;
\draw (166.59,238.54) node [anchor=north west][inner sep=0.75pt]    {$D$};
\draw (282.23,240.54) node [anchor=north west][inner sep=0.75pt]    {kink on $D$};
\draw (445.4,236.04) node [anchor=north west][inner sep=0.75pt]    {$D^0_{\operatorname{sing}}$};
\end{tikzpicture}
\caption{Gluing in a kink crossing.} \label{fig101}
\end{center}
\end{figure} 

\begin{theorem} [\cite{Pet}] \label{th4.1}
$\mathcal{G}(D)$ is the universal Vassiliev invariant of order one for virtual knotoids.
\end{theorem}

\begin{remark}
It has been shown in \cite{Pet} that $\mathcal{G}(K)$ is strictly stronger than $\mathcal{L}(K)$. 
\end{remark}

It seems typical that smoothing of a virtual knotoid diagram at a classical crossing gives less information than gluing of it. The following result demonstrate that  the gluing invariant $\mathcal{G}(K)$ is stronger than $0$-smoothing invariant $\mathcal{F}(K)$. 

\begin{theorem}\label{th4.2}
Invariant $\mathcal{G}(K)$ is strictly stronger than invariant $\mathcal{F} (K)$. 
\end{theorem} 

The proof of Theorem~\ref{th4.2} uses the extension of the singular based matrix invariant which we recall in the following subsection. 

\subsection{Singular based matrices}

In \cite{Tur2}, Turaev introduced the notion of virtual strings and their based matrices. 
We recall that for an integer $m \geq 0$ a \textit{virtual string $\alpha$ of rank $m$} (or briefly a \textit{string}) is an oriented circle, $S$, called the \textit{core circle} of $\alpha$, and a distinguished set of $2m$ distinct points of $S$ partitioned into $m$ ordered pairs. These $m$ ordered pairs of points are said to be the \textit{arrows} of $\alpha$.  The set of arrows of $\alpha$ is denoted by $\operatorname{arr} (\alpha)$. The endpoints $a, b \in S$ of the arrow $(a,b) \in \operatorname{arr}(\alpha)$ are called its \textit{tail} and \textit{head}, respectively. The $2m$ distinguished points of $S$ are called the \textit{endpoints} of $\alpha$. The string formed by an oriented circle and an empty set of arrows is called a \textit{trivial virtual string}. Two virtual strings are \textit{homeomorphic} if there is an orientation-preserving homeomorphism of the core circle, transforming the set of arrows of the first string into the set of arrows of the second string. The homeomorphism classes of virtual strings will be also called virtual strings. 

Furthermore, Henrich introduced the notion of singular virtual strings and singular based matrices and their homology classes~\cite{Hen}. To each singular virtual string, Henrich associated one of these matrices. She demonstrated that if two singular virtual strings are homotopic, then their associated singular based matrices are homologous. To make our paper self-contained and to facilitate the reader, we present the necessary definitions and conclusions below.

Analogously to~\cite{Tur2, Hen}, we introduce virtual open strings, corresponding to flat virtual knotoids, is a counterclockwise oriented arc with arrows that are not signed, and the direction of the arrow has a new meaning. To be more precise, let us give the following definitions.

\begin{definition} {\rm 
\begin{itemize}
\item[(i)] For an integer $m\geq 0$, a \textit{virtual open string} $\alpha$ of rank $m$ is an oriented arc $S$, called the \textit{core arc} of $\alpha$, and a distinguished set of $2m$ distinct points of $S$ partitioned into $m$ ordered pairs. The starting point of $S$ is called the \textit{tail} of $\alpha$, and the end point of $S$ is called the \textit{head} of $\alpha$. 
\item[(ii)] These $m$ ordered pairs of points are called \textit{arrows} of $\alpha$, and the collection of all arrows of $\alpha$ is denoted by $\operatorname{arr}(\alpha)$. 
\item[(iii)] The endpoints $a, b\in S$ of an arrow $(a, b) \in \operatorname{arr}(\alpha)$ are called its \textit{tail} and \textit{head}, respectively.  
\item[(iv)] Two virtual open strings are said to be \textit{homeomorphic} if there is an orientation-preserving homeomorphism of the core arc, transforming the set of arrows of the first open string into the set of arrows of the second open string. The homeomorphism classes of virtual open strings will be also called virtual open strings.
\end{itemize}
 }
\end{definition}

Given a virtual open string $\alpha$, we can realize the corresponding flat virtual knotoid diagram as follows. Each arrow in the open string is associated with a crossing in the knotoid diagram. Passing through the head of an arrow while traveling along the oriented core arc $S$ is equivalent to passing through the strand of a crossing indicated by the letter $b$ in the diagram below, see Fig.~\ref{fig127}. Passing through the tail of an arrow while traveling along the core arc $S$ is equivalent to passing through a crossing on the strand $a$, see  Fig.~\ref{fig127}.
\begin{figure}[htbp]
\begin{center}
\tikzset{every picture/.style={line width=1.0pt}}  
\begin{tikzpicture}[x=0.75pt,y=0.75pt,yscale=-0.9, xscale=0.9] 
\draw    (147.86,21.77) -- (79.76,92.61) ;
\draw    (79.62,21.85) -- (148,92.53) ;
\draw   (80.37,28.64) -- (79.51,21.68) -- (86.16,23.08) ;
\draw    (314.66,56.78) -- (234.76,56.32) ;
\draw   (309.96,53.6) -- (315.34,56.76) -- (309.72,59.44) ;
\draw    (245.86,73.74) .. controls (231.21,65.76) and (230.64,46.99) .. (245.67,39.12) ;
\draw   (141.56,23.27) -- (148.15,21.61) -- (147.54,28.6) ;
\draw    (304.95,39.66) .. controls (319.51,47.83) and (319.86,66.59) .. (304.74,74.29) ;
\draw (147.93,9.43) node [anchor=north west][inner sep=0.75pt]  [font=\small]  {$a$};
\draw (69.84,8.67) node [anchor=north west][inner sep=0.75pt]  [font=\small]  {$b$};
\draw (218.27,48.5) node [anchor=north west][inner sep=0.75pt]  [font=\small]  {$a$};
\draw (321.49,46.97) node [anchor=north west][inner sep=0.75pt]  [font=\small]  {$b$};
\end{tikzpicture}
\caption{A flat crossing and its corresponding arrow $(a,b)$.\label{fig127}}
\end{center}
\end{figure}

There are three types of moves (i)-(iii) on virtual open strings. 
\begin{itemize}
\item[(i)] Adding an arrow $(a, b)$ to $\operatorname{arr}(\alpha)$, where $a, b \in S$ are two points such that the arc $ab$ is disjoint from $\operatorname{arr}(\alpha)$.
\item[(ii)] Adding two new arrows, $(a, b)$ and $(b',a')$, where $a, b, a', b' \in S$ are four points satisfy that two arcs on $S$ that are disjoint from $\operatorname{arr}(\alpha)$, the first with endpoints $a$ and $a'$ and the second with endpoints $b$ and $b'$.
\item[(iii)] Replacing $(a',b), (b',c), (c',a)$ with $(a, b'), (b, c'), (c, a')$, where $a, b, a', b', c, c' \in S$ are points satisfy that $(a',b), (b',c), (c',a)$ are arrows of $\alpha$ and the arcs $aa'$, $bb'$, and $cc'$ are all disjoint from other arrows.
\end{itemize}
Two virtual open strings are said to be \textit{homotopic} if they are related by a finite sequence of moves, where each move is (i), (ii), or (iii).

The flat singular virtual knotoids with one singular crossing can be viewed as virtual open strings where one arrow is designated as the preferred arrow, this can be pictured by using a thickened arrow. We will refer to these modified virtual open strings as \textit{singular virtual open strings}. Naturally, the preferred arrow in a singular virtual open string corresponds to the singular crossing in the flat singular virtual knotoid diagram.

We add a singular virtual open string move that allows us to change which arrow is the preferred arrow in the diagram.  
\begin{itemize}
\item[(s-ii)] Suppose $(a,b)$ is the preferred arrow in the diagram and let $(a',b')$ be an arrow in the diagram such that the interior of one of the arcs in the core circle $S$ with endpoints $a$ and $b'$ is disjoint from arr($\alpha$) and the interior of one of the arcs in $S$ with endpoints $a'$ and $b$ is disjoint from arr($\alpha$). Then the designation as the preferred arrow may switch from $(a,b)$ to $(a',b')$. 
\end{itemize}

Note that we still allow moves (i)-(iii) whenever they involve only ordinary arrows. We also allow move (iii), even if one of the arrows is the preferred one. The homotopy equivalence relation that results from this collection of moves corresponds precisely to the equivalence relation on flat singular virtual knotoids with one singular crossing.

In~\cite{Hen}, Henrich introduced singular based matrices for singular virtual strings and their homology classes. She  showed that if two virtual strings are homotopic, then their associated based matrices are homologous. To help distinguish equivalence classes of flat singular virtual knotoids with one singular crossing, we use singular based matrix for singular virtual open strings. Below, we recall necessary definitions and fundamental results about singular based matrices. 

\begin{definition} [\cite{Hen}] {\rm 
A \textit{singular based matrix}, or shortly SBM, is a quadruple $(G, s, d, {\bf b} : G \times G \to H)$, where $H$ is an abelian group, $G$ is a finite set with $s, d\in G$, and the map ${\bf b}$ is skew-symmetric, i.e. ${\bf b}(g,h) = -{\bf b}(h,g)$ for all $g, h \in G$.
} 
\end{definition}

An element $g \in G - \{s, d\}$ is \textit{annihilating} if ${\bf b}(g, h) = 0$ for all $h\in G$. An element $g\in G - \left\{s, d\right\}$ is a \textit{core element} if ${\bf b}(g, h) = {\bf b}(s, h)$ for all $h\in G$. Two elements $g_1, g_2\in G - \left\{s, d\right\}$ are \textit{complementary} if ${\bf b}(g_1, h) + {\bf b}(g_2, h) = {\bf b}(s, h)$ for all $h\in G$. We call a distinguished element $g\in \left\{s, d\right\}$  \textit{annihilating-like} if ${\bf b}(g, h) = 0$ for all $h\in G$. We call $d$ \textit{core-like} if ${\bf b}(d, h) = {\bf b}(s, h)$ for all $h\in G$.

In~\cite{Tur2} Turaev defined elementary extension operations for based matrices, and in~\cite{Hen} Henrich defined analogous operations $\widetilde{M}_1, \widetilde{M}_2, \widetilde{M}_3$ and $N$ which transform SBM into another SBM as follows. 
 
\begin{itemize}
\item $\widetilde{M}_1$ transforms $(G, s, d, {\bf b} : G \times G \to H)$ into ($G_1 = G \sqcup {g}, s, d, {\bf b}_1 : G_1 \times G_1 \to H$) such that ${\bf b}_1$ extends ${\bf b}$ and ${\bf b}_1(g, h) = 0$ for all $h \in G_1$. 
\item $\widetilde{M}_2$ transforms $(G, s, d, {\bf b} : G \times G \to H)$ into ($G_2 = G \sqcup {g}, s, d, {\bf b}_2 : G_2 \times G_2 \to H$) such that ${\bf b}_2$ extends ${\bf b}$ and ${\bf b}_2(g, h) = {\bf b}_2(s, h)$ for all $h\in G_2.$ 
\item $\widetilde{M}_3$ transforms $(G, s, d, {\bf b} : G \times G \to H)$ into  ($G_3 = G \sqcup \{g_i, g_j\}, s, d, {\bf b}_3 : G_3 \times G_3 \to. H)$ such that ${\bf b}_3$ is any skew-symmetric map extending ${\bf b}$ with ${\bf b}_3(g_i, h) + {\bf b}_3(g_j, h) = {\bf b}_3(s, h)$ for all $h\in G_3$. 
\item Suppose there exists an element $g\in G$ such that ${\bf b}(g, h) + {\bf b}(d, h) = {\bf b}(s, h)$ for
all $h\in G$ (i.e. $g$ and $d$ are complementary), then $N$ transforms $(G, s, d, {\bf b})$ into $(G, s, g, {\bf b})$. In effect, the roles of $g$ and $d$ are switched.
\end{itemize}
Operations  $\widetilde{M}_1, \widetilde{M}_2, \widetilde{M}_3$ are called \textit{elementary extensions}  and $N$ is called a \textit{singularity switch}. We denote by $\widetilde{M}_i^{-1}$ the inverse operation for $\widetilde{M_i}$, $i=1,2,3$. 

\begin{definition} \cite[Def.~4.9]{Hen} \label{def:primitive} 
{\rm 
\begin{itemize}
\item[(i)] Two SBMs $(G, s, d, {\bf b})$ and $(G', s', d', {\bf b}')$ are said to be \textit{isomorphic} if there is a bijection $G\to G'$ sending $s$ to $s'$, $d$ to $d'$ and transforming ${\bf b}$ into ${\bf b}'$. 
\item[(ii)] An SBM $(G, s, d, {\bf b})$ is \textit{primitive} if it cannot be obtained from another SBM by an elementary extension, even after applications of the singularity switch operation. 
\item[(iii)] Two SBMs are \textit{homologous} if one can be obtained from the other by a finite number of moves, where each of moves is either $\widetilde{M}_1^{\pm 1}, \widetilde{M}_2^{\pm 1}, \widetilde{M}_3^{\pm 1}$ or $N$.
\end{itemize}
}
\end{definition}

\begin{theorem}\cite[Th.~4.12]{Hen}\label{th5.1}
Given two homologous primitive SBMs, the second can be obtained from the first by an isomorphism or a composition of an isomorphism with a single $\widetilde{M_1}^{-1} \circ N \circ \widetilde{M_2}$, $\widetilde{M_2}^{-1} \circ N \circ \widetilde{M_1}$ or $N$ move. 
\end{theorem}

Now we use theory of SBMs to define an invariant for singular virtual open strings. Given a singular virtual open string $\alpha$ with preferred arrow $(a,b)$, we let $G = G(\alpha) = \Big\{\operatorname{arr}(\alpha) \sqcup \{s\} \Big\}$, where $d$ in $(G, s, d, {\bf b})$ is precisely the preferred arrow $(a,b)$. The map ${\bf b} = {\bf b}(\alpha): G \times G \to \mathbb Z$ is defined according to the following Rule~1 and Rule~2. 

\underline{Rule~1.} For $e \in \operatorname{arr}(\alpha)$, ${\bf b}(e, s)$ is obtained by 1-smoothing the flat virtual knotoid diagram associated to $\alpha$ at the crossing corresponding to $e$. The 1-smoothing produces an oriented virtual multi-knotoid diagram, which is a knotoid diagram with one circular component. We choose an ordering $(\ell_1, \ell_2)$ for the two component. Taking the flat virtual multi-knotoid diagram that is the flattening of the smoothed virtual knotoid diagram, we give signs to each flat crossing in $\ell_1 \cap \ell_2$ as shown in Fig.~\ref{2.10}, then the intersection index is $\displaystyle i(e)=\sum_{c \in \ell_1\cap \ell_2} \operatorname{sgn}(x)$ and we define ${\bf b}(e,s) = i(e)$. 

\begin{figure}[htbp]
	\begin{center}
		\tikzset{every picture/.style={line width=1.0pt}} 
		\begin{tikzpicture}[x=0.75pt,y=0.75pt,yscale=-1,xscale=1]
			\draw    (210.67,2337.8) -- (149.85,2400.22) ; 
			\draw    (149.73,2337.87) -- (210.79,2400.15) ;
			\draw   (150.4,2343.85) -- (149.63,2337.72) -- (155.56,2338.96) ;
			\draw   (205.04,2339.12) -- (210.93,2337.66) -- (210.38,2343.82) ;
			\draw    (346.2,2338.42) -- (285.38,2400.84) ;
			\draw    (285.25,2338.49) -- (346.32,2400.77) ;
			\draw   (285.92,2344.47) -- (285.15,2338.34) -- (291.09,2339.58) ;
			\draw   (340.57,2339.74) -- (346.46,2338.28) -- (345.91,2344.44) ;
			\draw (134.27,2320.64) node [anchor=north west][inner sep=0.75pt]  [font=\normalsize]  {$\ell_2$};
			\draw (210,2321.26) node [anchor=north west][inner sep=0.75pt]  [font=\normalsize]  {$\ell_1$};
			\draw (347.36,2320.64) node [anchor=north west][inner sep=0.75pt]  [font=\normalsize]  {$\ell_2$};
			\draw (270.41,2321.26) node [anchor=north west][inner sep=0.75pt]  [font=\normalsize]  {$\ell_1$};
			\draw (195.76,2365) node [anchor=north west][inner sep=0.75pt]    {$+$};
			\draw (284.46,2365) node [anchor=north west][inner sep=0.75pt]    {$-$};	
		\end{tikzpicture}
		\caption{The $\operatorname{sgn}(c)$ for $c \in \ell_1 \cap \ell_2$.} \label{2.10}
	\end{center}
\end{figure}

\underline{Rule~2.} For $e, f \in \operatorname{arr} (\alpha)$ we define ${\bf b}(e,f)$ as follows. Suppose $e = (a,b)$ and $f = (c,d)$. Let $(ab)^{\circ}$ be the interior of the arc $ab$ and let $(cd)^{\circ}$ be the interior of the arc $cd$. Define $ab \cdot cd$ as the number of arrows with tails in $(ab)^{\circ}$ and heads in $(cd)^{\circ}$ minus the number of arrows with tails in $(cd)^{\circ}$ and heads in $(ab)^{\circ}$. (It should be noted that arc $ab$ (or $cd$) can also be disconnected, that is to say, the tail and the head of $\alpha$ are located between the arc $ab$ (or $cd$)). We define $\epsilon (e,f) = 1$ (respectively $\epsilon(e,f) = -1$ and $\epsilon (e,f) = 0$) is $e$ and $f$ to be linked positively (respectively, negatively and unlinked) as illustrated in Fig.~\ref{fig29}. Then we define ${\bf b}(e,f) = ab \cdot cd +\epsilon(e,f)$. 
\begin{figure}[htbp]
\begin{center}
\tikzset{every picture/.style={line width=1.0pt}}  
\begin{tikzpicture}[x=0.75pt,y=0.75pt,yscale=-1.0,xscale=1.0]  
\draw  [draw opacity=0] (154.11,318.29) .. controls (149.85,319.86) and (145.24,320.72) .. (140.44,320.72) .. controls (118.51,320.72) and (100.73,302.83) .. (100.73,280.76) .. controls (100.73,258.69) and (118.51,240.8) .. (140.44,240.8) .. controls (162.38,240.8) and (180.16,258.69) .. (180.16,280.76) .. controls (180.16,286.17) and (179.09,291.33) .. (177.16,296.03) -- (140.44,280.76) -- cycle ; \draw   (154.11,318.29) .. controls (149.85,319.86) and (145.24,320.72) .. (140.44,320.72) .. controls (118.51,320.72) and (100.73,302.83) .. (100.73,280.76) .. controls (100.73,258.69) and (118.51,240.8) .. (140.44,240.8) .. controls (162.38,240.8) and (180.16,258.69) .. (180.16,280.76) .. controls (180.16,286.17) and (179.09,291.33) .. (177.16,296.03) ;  
\draw    (100.81,279.99) -- (180.13,279.99) ;
\draw    (139.43,241.13) -- (139.58,320.42) ;
\draw   (173.93,277.02) -- (179.79,279.96) -- (173.93,282.91) ;
\draw    (215.36,278.23) -- (294.68,278.23) ;
\draw    (253.98,239.37) -- (254.13,318.66) ;
\draw   (288.48,275.27) -- (294.34,278.21) -- (288.48,281.15) ;
\draw   (257.13,312.74) -- (254.12,318.6) -- (251.27,312.66) ;
\draw   (136.51,247.12) -- (139.42,241.21) -- (142.37,247.1) ;
\draw    (369.03,320.47) -- (368.99,240.65) ;
\draw    (359.25,241.32) -- (359.37,319.83) ;
\draw   (366.04,247.48) -- (368.97,241.58) -- (371.9,247.48) ;
\draw   (362.37,312.71) -- (359.36,318.57) -- (356.52,312.63) ;
\draw  [fill={rgb, 255:red, 17; green, 16; blue, 16 }  ,fill opacity=1 ] (151.63,318.31) .. controls (151.65,319.68) and (152.76,320.78) .. (154.13,320.77) .. controls (155.49,320.76) and (156.59,319.64) .. (156.58,318.27) .. controls (156.57,316.9) and (155.45,315.8) .. (154.09,315.81) .. controls (152.72,315.82) and (151.62,316.94) .. (151.63,318.31) -- cycle ;
\draw  [fill={rgb, 255:red, 17; green, 16; blue, 16 }  ,fill opacity=1 ] (174.68,296.05) .. controls (174.7,297.42) and (175.81,298.52) .. (177.18,298.51) .. controls (178.54,298.5) and (179.64,297.38) .. (179.63,296.01) .. controls (179.62,294.64) and (178.5,293.54) .. (177.14,293.55) .. controls (175.77,293.56) and (174.67,294.68) .. (174.68,296.05) -- cycle ;
\draw  [draw opacity=0] (269.11,316.54) .. controls (264.85,318.11) and (260.24,318.97) .. (255.44,318.97) .. controls (233.51,318.97) and (215.73,301.08) .. (215.73,279.01) .. controls (215.73,256.94) and (233.51,239.05) .. (255.44,239.05) .. controls (277.38,239.05) and (295.16,256.94) .. (295.16,279.01) .. controls (295.16,284.42) and (294.09,289.58) .. (292.16,294.28) -- (255.44,279.01) -- cycle ; \draw   (269.11,316.54) .. controls (264.85,318.11) and (260.24,318.97) .. (255.44,318.97) .. controls (233.51,318.97) and (215.73,301.08) .. (215.73,279.01) .. controls (215.73,256.94) and (233.51,239.05) .. (255.44,239.05) .. controls (277.38,239.05) and (295.16,256.94) .. (295.16,279.01) .. controls (295.16,284.42) and (294.09,289.58) .. (292.16,294.28) ;  
\draw  [fill={rgb, 255:red, 17; green, 16; blue, 16 }  ,fill opacity=1 ] (266.63,316.56) .. controls (266.65,317.93) and (267.76,319.03) .. (269.13,319.02) .. controls (270.49,319.01) and (271.59,317.89) .. (271.58,316.52) .. controls (271.57,315.15) and (270.45,314.05) .. (269.09,314.06) .. controls (267.72,314.07) and (266.62,315.19) .. (266.63,316.56) -- cycle ;
\draw  [fill={rgb, 255:red, 17; green, 16; blue, 16 }  ,fill opacity=1 ] (289.68,294.3) .. controls (289.7,295.67) and (290.81,296.77) .. (292.18,296.76) .. controls (293.54,296.75) and (294.64,295.63) .. (294.63,294.26) .. controls (294.62,292.89) and (293.5,291.79) .. (292.14,291.8) .. controls (290.77,291.81) and (289.67,292.93) .. (289.68,294.3) -- cycle ;
\draw  [draw opacity=0] (382.76,317.74) .. controls (378.5,319.31) and (373.9,320.17) .. (369.1,320.17) .. controls (347.16,320.17) and (329.38,302.28) .. (329.38,280.21) .. controls (329.38,258.14) and (347.16,240.25) .. (369.1,240.25) .. controls (391.03,240.25) and (408.81,258.14) .. (408.81,280.21) .. controls (408.81,285.62) and (407.74,290.78) .. (405.81,295.48) -- (369.1,280.21) -- cycle ; \draw   (382.76,317.74) .. controls (378.5,319.31) and (373.9,320.17) .. (369.1,320.17) .. controls (347.16,320.17) and (329.38,302.28) .. (329.38,280.21) .. controls (329.38,258.14) and (347.16,240.25) .. (369.1,240.25) .. controls (391.03,240.25) and (408.81,258.14) .. (408.81,280.21) .. controls (408.81,285.62) and (407.74,290.78) .. (405.81,295.48) ;  
\draw  [fill={rgb, 255:red, 17; green, 16; blue, 16 }  ,fill opacity=1 ] (380.29,317.76) .. controls (380.3,319.13) and (381.41,320.23) .. (382.78,320.22) .. controls (384.14,320.21) and (385.24,319.09) .. (385.23,317.72) .. controls (385.22,316.35) and (384.1,315.25) .. (382.74,315.26) .. controls (381.37,315.27) and (380.28,316.39) .. (380.29,317.76) -- cycle ;
\draw  [fill={rgb, 255:red, 17; green, 16; blue, 16 }  ,fill opacity=1 ] (403.34,295.5) .. controls (403.35,296.87) and (404.46,297.97) .. (405.83,297.96) .. controls (407.19,297.95) and (408.29,296.83) .. (408.28,295.46) .. controls (408.27,294.09) and (407.15,292.99) .. (405.79,293) .. controls (404.42,293.01) and (403.33,294.13) .. (403.34,295.5) -- cycle ;
\draw (132.25,222.53) node [anchor=north west][inner sep=0.75pt]    {$f$};
\draw (181.18,273.49) node [anchor=north west][inner sep=0.75pt]    {$e$};
\draw (296.64,273.99) node [anchor=north west][inner sep=0.75pt]    {$e$};
\draw (245.63,220.36) node [anchor=north west][inner sep=0.75pt]    {$f$};
\draw (363.45,227.08) node [anchor=north west][inner sep=0.75pt]    {$e$};
\draw (350.11,323.11) node [anchor=north west][inner sep=0.75pt]    {$f$};					
\end{tikzpicture}
\centerline{$\epsilon (e,f) = 1$ \quad\qquad $\epsilon(e,f) = -1$ \quad\qquad $\epsilon (e,f) = 0$} 
\caption{Cases when arrows $f$ and $e$ are linked positively, negatively, or are unlinked.}
\label{fig29}
\end{center}
\end{figure}
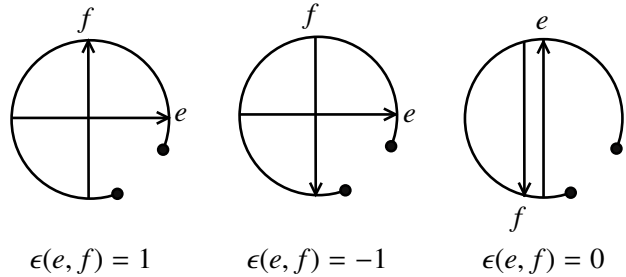

\begin{theorem} \cite[Th. 4.8]{Pet}\label{th5.2}
If two singular virtual open strings are homotopic, then their corresponding SBMs are homologous.
\end{theorem} 

\subsection{Proof of Theorem~\ref{th4.2}} 

We split Theorem~\ref{th4.2} in Lemmas~\ref{lemma4.1} and~\ref{lemma4.2}. 

\begin{lemma} \label{lemma4.1}
If $K$ and $K'$ are two homotopic virtual knotoids such that $\mathcal{F}(K) \neq \mathcal{F}(K')$, then $\mathcal{G}(K) \neq \mathcal{G}(K')$.
\end{lemma}

\begin{proof}
The statement holds from the universality of $\mathcal{G}$.
\end{proof}
 
\begin{lemma} \label{lemma4.2}
There exist two homotopic virtual knotoids $K^1$ and $K^2$ such that $\mathcal{F}(K^1) = \mathcal{F}(K^2)$, 
but $\mathcal{G}(K^1) \neq \mathcal{G}(K^2)$. 
\end{lemma}
 
\begin{proof}
Let $D^{1}$ and $D^{2}$ be diagrams of homotopic virtual knotoids $K^1$ and $K^2$, presented in Fig.~\ref{fig30}. These diagrams are related by two crossing changes applied to crossings denoted by $5$ and $6$. Let us show that $\mathcal F (K^1) = \mathcal F (K^2)$, but $\mathcal G(K^1) \neq \mathcal G(K^2)$.
\begin{figure}[htbp]
\begin{center}
\tikzset{every picture/.style={line width=1.0pt}}  
\begin{tikzpicture}[x=0.75pt,y=0.75pt,yscale=-1,xscale=1]
\draw    (79.44,630.42) .. controls (47,573.89) and (301,573.89) .. (275,620.89) ;
\draw    (79.44,630.42) .. controls (87.1,652.42) and (98.74,634.34) .. (106.74,625.34) ;
\draw    (78,644.89) .. controls (75,621.89) and (93.33,615.56) .. (99.33,625.56) ;
\draw    (105,634.56) .. controls (116.67,648.23) and (129,627.39) .. (135,621.39) ;
\draw    (106.74,625.34) .. controls (117.74,613.34) and (125.6,627.91) .. (131.6,631.91) ;
\draw    (229.89,629.06) .. controls (224.46,624.49) and (211.89,602.77) .. (196.46,622.2) ; 
\draw    (255.67,619.23) .. controls (268.33,596.89) and (295.85,640.02) .. (272,642.89) ;
\draw    (131.6,631.91) .. controls (137.6,636.91) and (145.62,639.78) .. (152.67,631.56) ;
\draw    (275,620.89) .. controls (258.46,648.2) and (255.89,620.49) .. (243,614.89) ;
\draw    (151,567.89) .. controls (153,567.89) and (170.47,574.79) .. (171.47,580.29) ;
\draw    (174.64,593.41) .. controls (180.64,609.41) and (188,606.23) .. (193,613.23) ;
\draw    (243,614.89) .. controls (228.46,607.06) and (217.67,652.56) .. (202.33,625.89) ;
\draw    (159.67,621.56) .. controls (179.22,596.18) and (180.02,650.98) .. (196.46,622.2) ;
\draw    (135,621.39) .. controls (152,599.89) and (188,678.89) .. (202,673.89) ;
\draw    (184.33,652.89) .. controls (196.33,654.89) and (239.1,647.52) .. (272,642.89) ;
\draw    (78,644.89) .. controls (79,659.89) and (113.67,641.89) .. (170.67,651.89) ;
\draw   (201.92,590.28) -- (193.76,586.69) -- (202.02,583.36) ; 
\draw    (330,625.65) .. controls (303,568.29) and (553,571.65) .. (527,618.65) ;
\draw    (330,625.65) .. controls (339.79,642.56) and (350.9,631.02) .. (358.9,622.02) ;
\draw    (330,642.65) .. controls (327,619.65) and (343.12,616.23) .. (349.79,622.89) ;
\draw    (360.45,631.78) .. controls (373.19,643.87) and (381,625.15) .. (387,619.15) ;
\draw    (358.9,622.02) .. controls (369.9,610.02) and (382,626.65) .. (388,630.65) ;
\draw    (505,621.15) .. controls (492.15,648.27) and (467.12,594.73) .. (454.12,612.23) ;
\draw    (505,621.15) .. controls (522.15,592.02) and (548.15,630.52) .. (524,640.65) ;
\draw    (388,630.65) .. controls (397.65,635.78) and (397.79,633.89) .. (404.79,628.89) ;
\draw    (527,618.65) .. controls (521.65,630.18) and (515.25,633.78) .. (509.12,626.56) ;
\draw    (403,565.65) .. controls (405,565.65) and (421.35,571.99) .. (422.35,577.49) ;
\draw    (426.35,591.09) .. controls (432.35,607.09) and (443.55,604.05) .. (452.46,619.87) ;
\draw    (500.79,614.23) .. controls (485.25,588.98) and (473.46,650.37) .. (452.46,619.87) ;
\draw    (414.12,619.56) .. controls (427.65,599.78) and (431.25,650.58) .. (445.79,624.23) ; 
\draw    (387,619.15) .. controls (404,597.65) and (440,676.65) .. (454,671.65) ;
\draw    (436.75,647.49) .. controls (447.95,645.89) and (495.4,654.52) .. (524,640.65) ;
\draw    (330,642.65) .. controls (331,657.65) and (356.17,666.75) .. (419.15,649.09) ;
\draw   (450.52,586.83) -- (442.36,583.25) -- (450.62,579.91) ;
\draw    (249,628.89) .. controls (246.43,632.89) and (241.62,638.58) .. (229.89,629.06) ;
\draw  [fill={rgb, 255:red, 17; green, 16; blue, 16 }  ,fill opacity=1 ] (451.53,671.67) .. controls (451.54,673.04) and (452.66,674.14) .. (454.02,674.12) .. controls (455.38,674.11) and (456.48,673) .. (456.47,671.63) .. controls (456.46,670.26) and (455.34,669.16) .. (453.98,669.17) .. controls (452.62,669.18) and (451.52,670.3) .. (451.53,671.67) -- cycle ;
\draw  [fill={rgb, 255:red, 17; green, 16; blue, 16 }  ,fill opacity=1 ] (400.53,565.67) .. controls (400.54,567.04) and (401.66,568.14) .. (403.02,568.12) .. controls (404.38,568.11) and (405.48,567) .. (405.47,565.63) .. controls (405.46,564.26) and (404.34,563.16) .. (402.98,563.17) .. controls (401.62,563.18) and (400.52,564.3) .. (400.53,565.67) -- cycle ;
\draw  [fill={rgb, 255:red, 17; green, 16; blue, 16 }  ,fill opacity=1 ] (148.53,567.92) .. controls (148.54,569.28) and (149.66,570.38) .. (151.02,570.37) .. controls (152.38,570.36) and (153.48,569.24) .. (153.47,567.87) .. controls (153.46,566.51) and (152.34,565.41) .. (150.98,565.42) .. controls (149.62,565.43) and (148.52,566.55) .. (148.53,567.92) -- cycle ;
\draw  [fill={rgb, 255:red, 17; green, 16; blue, 16 }  ,fill opacity=1 ] (199.53,673.92) .. controls (199.54,675.28) and (200.66,676.38) .. (202.02,676.37) .. controls (203.38,676.36) and (204.48,675.24) .. (204.47,673.87) .. controls (204.46,672.51) and (203.34,671.41) .. (201.98,671.42) .. controls (200.62,671.43) and (199.52,672.55) .. (199.53,673.92) -- cycle ;
\draw   (74.02,630.42) .. controls (74.02,627.43) and (76.44,625) .. (79.44,625) .. controls (82.43,625) and (84.85,627.43) .. (84.85,630.42) .. controls (84.85,633.41) and (82.43,635.84) .. (79.44,635.84) .. controls (76.44,635.84) and (74.02,633.41) .. (74.02,630.42) -- cycle ;
\draw   (122.81,628.94) .. controls (122.81,625.95) and (125.24,623.52) .. (128.23,623.52) .. controls (131.22,623.52) and (133.65,625.95) .. (133.65,628.94) .. controls (133.65,631.93) and (131.22,634.36) .. (128.23,634.36) .. controls (125.24,634.36) and (122.81,631.93) .. (122.81,628.94) -- cycle ;
\draw   (220.24,623.8) .. controls (220.24,620.8) and (222.67,618.38) .. (225.66,618.38) .. controls (228.65,618.38) and (231.08,620.8) .. (231.08,623.8) .. controls (231.08,626.79) and (228.65,629.21) .. (225.66,629.21) .. controls (222.67,629.21) and (220.24,626.79) .. (220.24,623.8) -- cycle ;
\draw   (270.24,618.65) .. controls (270.24,615.66) and (272.67,613.23) .. (275.66,613.23) .. controls (278.65,613.23) and (281.08,615.66) .. (281.08,618.65) .. controls (281.08,621.64) and (278.65,624.07) .. (275.66,624.07) .. controls (272.67,624.07) and (270.24,621.64) .. (270.24,618.65) -- cycle ;
\draw   (326.24,626.65) .. controls (326.24,623.66) and (328.66,621.23) .. (331.65,621.23) .. controls (334.65,621.23) and (337.07,623.66) .. (337.07,626.65) .. controls (337.07,629.64) and (334.65,632.07) .. (331.65,632.07) .. controls (328.66,632.07) and (326.24,629.64) .. (326.24,626.65) -- cycle ;
\draw   (376.81,625.22) .. controls (376.81,622.23) and (379.23,619.81) .. (382.23,619.81) .. controls (385.22,619.81) and (387.64,622.23) .. (387.64,625.22) .. controls (387.64,628.22) and (385.22,630.64) .. (382.23,630.64) .. controls (379.23,630.64) and (376.81,628.22) .. (376.81,625.22) -- cycle ;
\draw   (472.75,618.88) .. controls (472.75,615.89) and (475.18,613.46) .. (478.17,613.46) .. controls (481.16,613.46) and (483.59,615.89) .. (483.59,618.88) .. controls (483.59,621.87) and (481.16,624.3) .. (478.17,624.3) .. controls (475.18,624.3) and (472.75,621.87) .. (472.75,618.88) -- cycle ;
\draw   (523.09,614.37) .. controls (523.09,611.37) and (525.52,608.95) .. (528.51,608.95) .. controls (531.5,608.95) and (533.93,611.37) .. (533.93,614.37) .. controls (533.93,617.36) and (531.5,619.78) .. (528.51,619.78) .. controls (525.52,619.78) and (523.09,617.36) .. (523.09,614.37) -- cycle ;
\draw (167,681.18) node [anchor=north west][inner sep=0.75pt]  [font=\large]  {$D^{1}$};
\draw (424,680.18) node [anchor=north west][inner sep=0.75pt]  [font=\large]  {$D^{2}$};
\draw (96.03,635.61) node [anchor=north west][inner sep=0.75pt]  [font=\normalsize] [align=left] {3};
\draw (153.03,632.94) node [anchor=north west][inner sep=0.75pt]  [font=\normalsize] [align=left] {4};
\draw (193.7,628.27) node [anchor=north west][inner sep=0.75pt]  [font=\normalsize] [align=left] {5};
\draw (247.36,629.61) node [anchor=north west][inner sep=0.75pt]  [font=\normalsize] [align=left] {6};
\draw (171.7,572.94) node [anchor=north west][inner sep=0.75pt]  [font=\normalsize] [align=left] {1};
\draw (171.03,655.94) node [anchor=north west][inner sep=0.75pt]  [font=\normalsize] [align=left] {2};
\draw (347.67,631.61) node [anchor=north west][inner sep=0.75pt]  [font=\normalsize] [align=left] {3};
\draw (403.61,629.61) node [anchor=north west][inner sep=0.75pt]  [font=\normalsize] [align=left] {4};
\draw (445.36,626.61) node [anchor=north west][inner sep=0.75pt]  [font=\normalsize] [align=left] {5};
\draw (499.36,628.27) node [anchor=north west][inner sep=0.75pt]  [font=\normalsize] [align=left] {6};
\draw (419.16,652.81) node [anchor=north west][inner sep=0.75pt]  [font=\normalsize] [align=left] {2};
\draw (425.56,568.6) node [anchor=north west][inner sep=0.75pt]  [font=\normalsize] [align=left] {1};	
\end{tikzpicture}
\caption{Example showing that $\mathcal{G}$ is strictly stronger than~$\mathcal{F}$.} \label{fig30}
\end{center}
\end{figure}
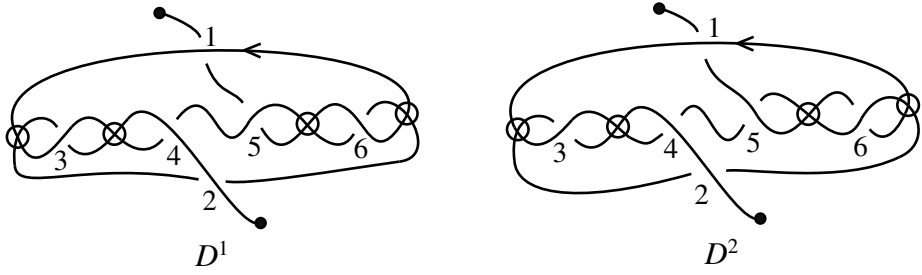

We denote by $D^i_k$, where $i=1,2$ and $k=1,\dots,6$, the virtual knotoid diagram obtained by applying $0$-smoothing at crossing $k$ of $D^{i}$ and by $\overline{D_k^i}$ the flattening of it. It is easy to check that $w(D^1)=0$ and $w(D^2)=0$. Moreover, for $i=1,2$ we get  
\begin{equation*} \label{eq1}
[\overline{D^i_1}] = [\overline{D^i_2}], \quad [\overline{D^i_3}] = [\overline{D^i_4}], \quad [\overline{D^i_5}]  = [\overline{D^i_6}]. 
\end{equation*}
Therefore, 
\begin{equation*} \label{eq3}
\mathcal{F}([D^1])   =   \sum_{k=1}^{6} w_{D^1}(k) \, [\overline{D^1_k}] - w(D^{1}) \, [\overline{D^1}]   
=  + [\overline{D^1_1}] - [\overline{D^1_2}] - [\overline{D^1_3}] + [\overline{D^1_4}] - [\overline{D^1_5}] + [\overline{D^1_6}] - 0 =  0 
\end{equation*}
and 
\begin{equation*} \label{eq4}
\mathcal{F}([D^2]) =  \sum_{k=1}^{6} w_{D^2}(k) \, [\overline{D^2_k}] - w(D^{2}) \, [\overline{D^2}] 
=   + [\overline{D^2_1}] - [\overline{D^2_2}] - [\overline{D^2_3}] + [\overline{D^2_4}] + [\overline{D^2_5}] - [\overline{D^2_6}] - 0 =  0, 
\end{equation*}
where $w_D(k)$ is the sign of crossing $k$ in diagram $D$. Thus, $\mathcal{F}(K^1) = \mathcal{F}(K^2)$.

Now let us show that $\mathcal G (K^1) \neq \mathcal G (K^2)$.
Consider $\mathcal{G}(K^1)$, where $K^1$ has diagram $D^1$. Recall that 
$$
\mathcal{G} (D^1) = \sum_{k=1}^{6} w_{D^1} (k) \, [\overline{(D^1)^k_{\operatorname{glue}}}] - w(D^1) \, [\overline{(D^1)^0_{\operatorname{sing}}}]. 
$$
The crossings $5$ and $6$, contribute the terms $(-1)[\overline{({D^{1})}^5_{glue}}]$ and $[\overline{({D^{1})}^6_{glue}}]$, respectively, where flat singular virtual knotoids associated to ``gluing" $D^1$ at crossings $5$ and $6$ look as the singular virtual open strings pictured in Fig.~\ref{fig31} and compute singular based matrices (SBMs) associated to both of them.
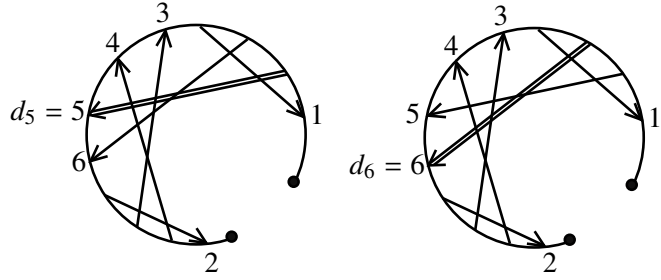
\begin{figure}[htbp]
\begin{center}
\tikzset{every picture/.style={line width=1.0pt}} 
\begin{tikzpicture}[x=0.75pt,y=0.75pt,yscale=-1,xscale=1]
\draw  [draw opacity=0] (224.1,340.9) .. controls (218.2,343.06) and (211.82,344.24) .. (205.16,344.24) .. controls (174.67,344.24) and (149.94,319.45) .. (149.94,288.87) .. controls (149.94,258.29) and (174.67,233.5) .. (205.16,233.5) .. controls (235.66,233.5) and (260.39,258.29) .. (260.39,288.87) .. controls (260.39,297.8) and (258.28,306.23) .. (254.54,313.69) -- (205.16,288.87) -- cycle ; \draw   (224.1,340.9) .. controls (218.2,343.06) and (211.82,344.24) .. (205.16,344.24) .. controls (174.67,344.24) and (149.94,319.45) .. (149.94,288.87) .. controls (149.94,258.29) and (174.67,233.5) .. (205.16,233.5) .. controls (235.66,233.5) and (260.39,258.29) .. (260.39,288.87) .. controls (260.39,297.8) and (258.28,306.23) .. (254.54,313.69) ;  
\draw  [fill={rgb, 255:red, 17; green, 16; blue, 16 }  ,fill opacity=1 ] (251.89,312.67) .. controls (251.9,314.04) and (253.01,315.14) .. (254.38,315.13) .. controls (255.74,315.12) and (256.84,314) .. (256.83,312.63) .. controls (256.82,311.26) and (255.7,310.16) .. (254.34,310.17) .. controls (252.98,310.18) and (251.88,311.3) .. (251.89,312.67) -- cycle ;
\draw  [fill={rgb, 255:red, 17; green, 16; blue, 16 }  ,fill opacity=1 ] (220.6,340.36) .. controls (220.61,341.72) and (221.72,342.83) .. (223.09,342.81) .. controls (224.45,342.8) and (225.55,341.68) .. (225.54,340.32) .. controls (225.53,338.95) and (224.41,337.85) .. (223.05,337.86) .. controls (221.68,337.87) and (220.59,338.99) .. (220.6,340.36) -- cycle ;
\draw    (207.18,234.36) -- (259.86,280.42) ;
\draw    (250.29,256.54) -- (152.29,277.21) ;
\draw    (231.14,240.7) -- (151.95,302.19) ;
\draw    (189.69,236.67) -- (175.42,335.41) ;
\draw    (166.22,251.43) -- (192.91,342.33) ;
\draw    (251.48,258.42) -- (152.51,279.21) ;
\draw    (158.86,319.26) -- (209.48,342.99) ;
\draw   (255.89,271.43) -- (259.04,279.73) -- (250.72,276.68) ;
\draw   (159.64,280.53) -- (151.13,278.87) -- (156.51,272.1) ;
\draw   (185.18,244.08) -- (189.71,236.44) -- (192.52,244.84) ;
\draw   (164.72,259.62) -- (165.88,250.81) -- (171.77,257.44) ;
\draw   (160.53,299.75) -- (151.99,302.17) -- (155.73,294.14) ;
\draw   (205.37,336.06) -- (210.26,343.48) -- (201.47,342.33) ;
\draw  [draw opacity=0] (394.62,342.45) .. controls (388.72,344.61) and (382.34,345.79) .. (375.69,345.79) .. controls (345.19,345.79) and (320.46,321) .. (320.46,290.42) .. controls (320.46,259.84) and (345.19,235.05) .. (375.69,235.05) .. controls (406.18,235.05) and (430.91,259.84) .. (430.91,290.42) .. controls (430.91,299.34) and (428.8,307.77) .. (425.06,315.24) -- (375.69,290.42) -- cycle ; \draw   (394.62,342.45) .. controls (388.72,344.61) and (382.34,345.79) .. (375.69,345.79) .. controls (345.19,345.79) and (320.46,321) .. (320.46,290.42) .. controls (320.46,259.84) and (345.19,235.05) .. (375.69,235.05) .. controls (406.18,235.05) and (430.91,259.84) .. (430.91,290.42) .. controls (430.91,299.34) and (428.8,307.77) .. (425.06,315.24) ; 
\draw  [fill={rgb, 255:red, 17; green, 16; blue, 16 }  ,fill opacity=1 ] (422.41,314.22) .. controls (422.42,315.59) and (423.54,316.69) .. (424.9,316.68) .. controls (426.27,316.66) and (427.36,315.55) .. (427.35,314.18) .. controls (427.34,312.81) and (426.23,311.71) .. (424.86,311.72) .. controls (423.5,311.73) and (422.4,312.85) .. (422.41,314.22) -- cycle ;
\draw  [fill={rgb, 255:red, 17; green, 16; blue, 16 }  ,fill opacity=1 ] (391.12,341.9) .. controls (391.13,343.27) and (392.24,344.37) .. (393.61,344.36) .. controls (394.97,344.35) and (396.07,343.23) .. (396.06,341.86) .. controls (396.05,340.49) and (394.93,339.39) .. (393.57,339.41) .. controls (392.2,339.42) and (391.11,340.54) .. (391.12,341.9) -- cycle ;
\draw    (377.7,235.91) -- (430.38,281.97) ;
\draw    (420.81,258.09) -- (322.81,278.76) ;
\draw    (402.04,242.17) -- (322.48,303.74) ;
\draw    (360.21,238.21) -- (345.95,336.96) ;
\draw    (336.74,252.98) -- (363.43,343.88) ;
\draw    (403.87,243.36) -- (325.07,304.16) ;
\draw    (329.38,320.81) -- (380,344.54) ;
\draw   (426.41,272.97) -- (429.56,281.28) -- (421.24,278.23) ;
\draw   (331.09,304.07) -- (322.43,304.36) -- (326.15,296.56) ;
\draw   (355.71,245.62) -- (360.23,237.98) -- (363.04,246.39) ;
\draw   (335.24,261.17) -- (336.4,252.36) -- (342.29,258.99) ;
\draw   (331.21,280.03) -- (322.36,279.26) -- (328.72,273.08) ;
\draw   (375.89,337.61) -- (380.78,345.02) -- (371.99,343.88) ;
\draw (260.78,272.24) node [anchor=north west][inner sep=0.75pt]  [font=\normalsize]  {$1$};
\draw (183.14,220.43) node [anchor=north west][inner sep=0.75pt]  [font=\normalsize]  {$3$};
\draw (158.03,235.91) node [anchor=north west][inner sep=0.75pt]  [font=\normalsize]  {$4$};
\draw (109.81,269.72) node [anchor=north west][inner sep=0.75pt]  [font=\normalsize]  {$d_5=5$};
\draw (140.81,295.79) node [anchor=north west][inner sep=0.75pt]  [font=\normalsize]  {$6$};
\draw (207.81,346.79) node [anchor=north west][inner sep=0.75pt]  [font=\normalsize]  {$2$};
\draw (431.31,273.79) node [anchor=north west][inner sep=0.75pt]  [font=\normalsize]  {$1$};
\draw (353.66,221.97) node [anchor=north west][inner sep=0.75pt]  [font=\normalsize]  {$3$};
\draw (328.55,237.46) node [anchor=north west][inner sep=0.75pt]  [font=\normalsize]  {$4$};
\draw (309.33,271.27) node [anchor=north west][inner sep=0.75pt]  [font=\normalsize]  {$5$};
\draw (280.93,296.93) node [anchor=north west][inner sep=0.75pt]  [font=\normalsize]  {$d_6=6$};
\draw (378.33,346.79) node [anchor=north west][inner sep=0.75pt]  [font=\normalsize]  {$2$};	
\end{tikzpicture}
\caption{Singular virtual open strings associated to gluing $D^1$ at crossings $5$ and $6$.} \label{fig31} 
\end{center}
\end{figure}

\underline{SBM ${\bf b}_5$.}
Let ${\bf b}_5$ be the skew-symmetric map in the SBM $(G_5, s_5, d_5, {\bf b}_5)$ of the singular virtual open string associated to crossing $5$, as pictured in Fig.~\ref{fig31}, left. (It may also be useful to think of this string as the flat singular virtual knotoid associated to gluing at crossing $5$ in Fig.~\ref{fig30}, left.) 

By Rule~1, ${\bf b}_5(1, s_5)$ is the intersection index of the flat multi-knotoid obtained by 1-smoothing at crossing $1$ where we order the components as followed: the component on the right is $\ell_1$ and the component on the left is $\ell_2$, see Fig.~\ref{2.10}, then ${\bf b}_5(1, s_5)=-2$. Similarly, we get 
$$
{\bf b}_5(2, s_{5})=2, \quad {\bf b}_5(3, s_5)=2, \quad {\bf b}_5(4, s_5)=0, \quad {\bf b}_5(6, s_5)=-2, \quad {\bf b}_5(d_5, s_5)=0.
$$ 
We use Rule~2 to compute ${\bf b}_5(1, 2)$. Denoting $1=(a_1,b_1)$ and $2=(a_2,b_2)$ we get ${\bf b}_5(1, 2) = a_1b_1 \cdot a_2 b_2 + \epsilon(1,2) =0-2+0=-2$. 
Analogously, 
$$
\begin{gathered} 
{\bf b}_5(1, 3)=-2, \quad {\bf b}_5(1, 4)=0, \quad {\bf b}_5(1, 6)=1, \quad {\bf b}_5(1, d_5)=0, \quad {\bf b}_5(2, 3)=1, \quad  {\bf b}_5(2, 4)=2, \quad {\bf b}_5(2, 6) = 2, \cr 
{\bf b}_5(2, d_5)=2, \quad {\bf b}_5(3, 4)=1, \quad {\bf b}_5(3, 6)=2, \quad  {\bf b}_5(3, d_5)=1, \quad {\bf b}_5(4, 6)=1,  \quad {\bf b}_5(4, d_5)=0, \quad {\bf b}_5(6, d_5)=-1,
\end{gathered}
$$
where $d_5=5$ in Fig.~\ref{fig31}, left. We organize the obtained values in a matrix as follows: 
$$
{\bf b}_5 = \qquad \bordermatrix{
	& s_5 & 1 &  2 & 3 & 4 & 6 & d_5  \cr
	s_5   & 0 &  2 &-2 &-2 & 0 & 2 & 0\cr
	1   &-2 &  0 &-2 &-2 & 0 & 1 & 0\cr
	2   & 2 &  2 & 0 & 1 & 2 & 2 & 2\cr
	3   & 2 &  2 &-1 & 0 & 1 & 2 & 1\cr
	4   & 0 &  0 &-2 &-1 & 0 & 1 & 0\cr
	6   &-2 & -1 &-2 &-2 &-1 & 0 &-1\cr 
	d_5   & 0 &  0 &-2 &-1 & 0 & 1 & 0\cr 
}
$$

We observe that matrix ${\bf b}_5$ is primitive. According to the Definition~\ref{def:primitive}, ${\bf b}_5$ is primitive if it cannot be obtained from another SBM by an elementary extension, even after applications of the singularity switch operation. Indeed, ${\bf b}_5$ cannot be  obtained by $\widetilde{M}_1$ since no row in the matrix is entirely consisting of zeros; it cannot be obtained by $\widetilde{M}_2$ since no row in the matrix is equal to $s_5 = [0, 2, -2, -2, 0, 2, 0]$; it cannot be obtained by $\widetilde{M}_3$ since calculations show that there are no two rows in the matrix whose sum is equal to $s_5$; it cannot be obtained by $N$ since calculations shows that no row $g$ satisfies $\text{row } g + \text{row } d_5 = \text{row } s_5$. Thus, ${\bf b}_5$ is primitive.

Furthermore, $d_5 \in {\bf b}_5$ is not annihilating-like, since ${\bf b}_5 (d_5, h) = 0$ not for all $h\in G$, or core-like, since ${\bf b}_5 (d_5, h) = {\bf b}_5 (s_5, h)$ not for all $h\in G$. 

\underline{SBM ${\bf b}_6$.}
Let ${\bf b}_6$ be the skew-symmetric map in the SBM $(G_6, s_6, d_6, {\bf b}_6)$ of the singular virtual open string associated to crossing $6$, as pictured in Fig.~\ref{fig31}, right. (It may also be useful to think of this string as the flat singular virtual knotoid associated to gluing at crossing $6$ in Fig.~\ref{fig30}, right.)

Analogously to ${\bf b}_5$, one can compute values of ${\bf b}_6$ in the SBM of the singular virtual open string associated to crossing $6$ by applying Rule~1 and Rule~2. We represent obtained values in matrix form as follows:
$$ 
{\bf b}_6 = \qquad \bordermatrix{
	& s_6 & 1 &  2 & 3 & 4 & 5 & d_6  \cr
	s_6   & 0 &  2 &-2 &-2 & 0 & 0 & 2\cr
	1   &-2 &  0 &-2 &-2 & 0 & 0 & 1\cr
	2   & 2 &  2 & 0 & 1 & 2 & 2 & 2\cr
	3   & 2 &  2 &-1 & 0 & 1 & 1 & 2\cr
	4   & 0 &  0 &-2 &-1 & 0 & 0 & 1\cr
	5   & 0 &  0 &-2 &-1 & 0 & 0 & 1\cr 
	d_6   &-2 & -1 &-2 &-2 &-1 &-1 & 0\cr 
}
$$ 
where $d_6=6$ in Fig.~\ref{fig31}, right.

By applying for ${\bf b}_6$ the same arguments as used for ${\bf b}_5$, we conclude that ${\bf b}_6$ is primitive, and $d_6 \in {\bf b}_6$ is not annihilating-like or core-like.

\underline{Comparing ${\bf b}_5$ and ${\bf b}_6$.}
Now we determine  whether ${\bf b}_5 $ and ${\bf b}_6 $ are related by an isomorphism or a composition of an isomorphism with a single move: $\widetilde{M_1}^{-1} \circ N \circ \widetilde{M_2}$, $\widetilde{M_2}^{-1} \circ N \circ \widetilde{M_1}$ or $N$.

Observe, that ${\bf b}_5 $ and ${\bf b}_6 $ are not related by an isomorphism. Indeed,  since 
${\bf b}_5 $ has elements $ \{s_5, 1, 2, 3, 4, 6, d_5\}$ and ${\bf b}_6$ has elements $ \{s_6, 1, 2, 3, 4, 5, d_6\} $, there is no bijection $G_5 \to G_6$ sending $s_5$ to $s_6$, $d_5$ to $d_6$ and transforming ${\bf b}_5$ into ${\bf b}_6$.

Moreover, $ \mathbf{b}_5 $ and $ \mathbf{b}_6 $ cannot be related by a composition of an isomorphism with a single move: $\widetilde{M_1}^{-1} \circ N \circ \widetilde{M_2}$, $\widetilde{M_2}^{-1} \circ N \circ \widetilde{M_1}$ or $N$. Indeed, since $\mathbf{b}_5$ and $\mathbf{b}_6$ are primitive,  $\widetilde{M}_1^{-1}$, $\widetilde{M}_2^{-1}$ and $N^{-1}$ cannot be applied to $\mathbf{b}_5$ and $\mathbf{b}_6$.

Thus, by Theorem~\ref{th5.1}, ${\bf b}_5$ and ${\bf b}_6$ are not homologous. It then follows from Theorem~\ref{th5.2} that the two singular virtual open strings in Fig.~\ref{fig31} are not homotopic. Thus, the terms in $\mathcal{G}(K^1)$ corresponding to crossings $5$ and $6$ do not cancel as they did in $\mathcal{F}(K^1)$.

We see that the terms in $\mathcal{G}(K^2)$ corresponding to crossings $5$ and $6$ in $K^2$ are the same as in $\mathcal{G}(K^1)$, except with opposite sign. Thus, $\mathcal{G}(K^1) - \mathcal{G}(K^2)$ is non-zero. Indeed this difference has two terms, one with coefficient $+2$ and one with coefficient $-2$. By the analysis above, these terms do not cancel and, hence, $\mathcal{G}(K^1) \neq \mathcal{G}(K^2)$.

The proof of Lemma is completed. 
\end{proof}

\begin{remark} {\rm 
Based on the results in~\cite{FLV}, let $K$ be a spherical virtual knotoid, there is a three-variable transcendental invariant $H_{K}(t,y,z)$, and $H_{K}(t,y,z)$ is a Vassiliev invariant of order one. On one hand, let $K'$ and $K''$ be two homotopic virtual knotoids, by the universality of $\mathcal{G}$, if $H_{K'}(t,y,z) \neq H_{K''}(t,y,z)$, then $\mathcal{G}(K')\neq \mathcal{G}(K'')$. On the other hand, we know that $K^1$ and $K^2$ are zero height in Fig.~\ref{fig30}, then $H_{K^1}(t,y,z) = H_{K^2}(t,y,z)$ since both vanish to zero, but $\mathcal{G}(K^1) \neq \mathcal{G}(K^2)$. 
}
\end{remark}

\end{document}